\documentclass[11pt,a4paper,reqno]{amsart}
\usepackage[margin=1.1in]{geometry}
\usepackage{amssymb,amsmath,graphicx,amsfonts}
\usepackage{mathrsfs}
\usepackage{multirow}
\usepackage{epsfig}
\usepackage{color}
\usepackage{enumerate,enumitem}
\usepackage{empheq}
\usepackage{cases}
\usepackage{cite}
\usepackage{amsgen,amscd}
\usepackage{booktabs}
\usepackage{latexsym,euscript}
\usepackage{bm,hyperref,subcaption}

\usepackage{placeins}
\hypersetup{hidelinks}

\allowdisplaybreaks
\newtheorem{theorem}{Theorem}[section]

\newtheorem{lemma}[theorem]{Lemma}

\newtheorem{remark}[theorem]{Remark}
\theoremstyle{definition}

\numberwithin{equation}{section}

\providecommand{\Dim}{\operatorname{dim}}            
\providecommand{\dim}{\Dim}

\providecommand*{\Dist}[2]{\operatorname{dist}({#1};{#2})}   
\providecommand*{\Dist}[2]{\Dist{#1}{#2}}

\newcommand{\Ba}{{\boldsymbol{a}}}
\newcommand{\Bb}{{\boldsymbol{b}}}

\newcommand{\Be}{{\boldsymbol{e}}}
\newcommand{\Bf}{{\boldsymbol{f}}}

\newcommand{\Bn}{{\boldsymbol{n}}}

\newcommand{\Bt}{{\boldsymbol{t}}}
\newcommand{\Bu}{{\boldsymbol{u}}}
\newcommand{\Bv}{{\boldsymbol{v}}}
\newcommand{\Bw}{{\boldsymbol{w}}}
\newcommand{\Bx}{{\boldsymbol{x}}}

\newcommand{\Bz}{{\boldsymbol{z}}}

\newcommand{\BU}{{\boldsymbol{U}}}

\newcommand{\BX}{{\boldsymbol{X}}}

\newcommand{\mubf}{\boldsymbol{\mu}}

\newcommand{\Ci}{{\cal I}}

\newcommand{\Cw}{{\cal W}}

\newcommand{\bbI}{\mathbb{I}}

\newcommand{\bbP}{\mathbb{P}}

\newcommand{\bbR}{\mathbb{R}}

\newcommand{\Rd}{\mathrm{d}}

\newcommand{\be}{\begin{eqnarray}}
\newcommand{\ee}{\end{eqnarray}}

\newcommand{\ben}{\begin{eqnarray*}}
\newcommand{\een}{\end{eqnarray*}}

\renewcommand{\Ci}{\mathcal{I}}
\renewcommand{\Cw}{\mathcal{W}}

\renewcommand{\bbI}{\textrm{id}}
\newcommand{\hBn}{\hat{\bm{n}}}

\newcommand{\Tp}{\tilde{p}}
\renewcommand{\TH}{\tilde{H}}

\begin{document}

\title[]{A structure--preserving ALE--BGN--MDR method for
Navier--Stokes free boundary problems with moving contact lines and gravity}
\thanks{The work of Nuo Lei is partially supported by the Hong Kong Scholar Program.}

\author[]{Harald Garcke,\,\, Jiashun Hu,\,\, and\,\, Nuo Lei}
\address{Harald Garcke: Fakult\"at f\"ur Mathematik, Universit\"at Regensburg, Regensburg, Germany.
{\rm Corresponding author. Email address: {\tt harald.garcke@ur.de}}}
\address{Jiashun Hu, Nuo Lei: Department of Applied Mathematics, The Hong Kong Polytechnic University, Hong Kong.
{\rm Email addresses: {\tt jiashun.hu@polyu.edu.hk} and {\tt nuo97.lei@polyu.edu.hk}}}
\address{Nuo Lei: Hua Loo-Keng Center for Mathematical Sciences, Academy of Mathematics and Systems Science, Chinese Academy of Sciences, Beijing, China.}



\maketitle

\begin{abstract}
We propose a gravity-consistent arbitrary Lagrangian--Eulerian finite element method for incompressible Navier--Stokes free-boundary problems with moving contact lines. 
A direct body-force discretization of gravity may fail to ensure consistency between the discrete gravitational work and the variation of the gravitational potential energy on the evolving domain, resulting in an artificial consistency error
and persistent spurious velocities near equilibrium. 
To remove this inconsistency, we reformulate the gravitational potential energy variation as a moving-boundary integral over intermediate ALE configurations and evaluate it exactly using Simpson's quadrature rule. 
This leads to a mildly nonlinear fully discrete scheme in which the gravitational contribution is exactly consistent with the discrete potential-energy variation. 
The proposed method preserves volume exactly, satisfies a discrete energy-dissipation law including gravitational potential energy, and under suitable assumptions, drives the discrete velocity to zero in the long-time regime, thereby excluding persistent gravity-induced spurious velocities. 
Together with the BGN treatment of the free surface and the MDR bulk mesh extension, the scheme maintains accurate interface tracking and good mesh quality near the moving contact line.
Numerical experiments in two and three spatial dimensions confirm the theoretical properties. 
\end{abstract}

\setlength\abovedisplayskip{4pt}
\setlength\belowdisplayskip{4pt}

\section{Introduction}
Incompressible flows with free surfaces and moving contact lines arise in wetting, dewetting, and microfluidic applications. 
Numerical methods for such moving-interface problems are often divided into interface-capturing approaches and front-tracking or body-fitted approaches; the latter explicitly represent the interface by a lower-dimensional moving mesh and are therefore natural for accurate contact-line tracking \cite{BGN2015,Agnese2020,Zhao2020,Zhao2021CMAME,Garcke2023}. 
A central difficulty is that the evolution of the geometry, specifically the liquid--gas interface and the contact line, is coupled with the bulk viscous flow, surface tension, wall wettability, and contact-angle dynamics.
A substantial literature has been devoted to moving contact line models. 
Fixed-contact-angle models prescribe the equilibrium Young angle, whereas dynamic contact-angle models incorporate contact-line physics through either microscopically motivated continuum boundary conditions, such as the generalized Navier slip condition \cite{Qian2003,Qian2006JFM,Guo2024}, 
or effective sharp-interface laws, such as those proposed by Ren and E \cite{Ren2007,Ren2011}. Recent numerical frameworks can accommodate a range of such contact-line laws, including the Ren--E, Cox, Onsager, and contact-angle-hysteresis models \cite{Xu2023}.

Moving contact lines are intrinsically a three-phase junction. Here we consider the standard one-phase free-boundary description, where the gas phase is not solved explicitly but enters through the boundary conditions. The liquid phase satisfies the incompressible Navier--Stokes equations on the moving domain \(\Omega(t)\), whose boundary consists of the fluid--gas free surface \(\Gamma_{\rm F}(t)\) and the wetted fluid--solid boundary \(\Gamma_{\rm s}(t)\).
Body-fitted arbitrary Lagrangian--Eulerian (ALE) methods provide a natural sharp-interface approach for such problems. Since the computational mesh moves with the fluid domain, the free surface and the contact line can be tracked explicitly and contact-angle conditions can be imposed directly. 
For long-time simulations under gravity, however, accurate interface tracking alone is not sufficient. 
The fully discrete method must also preserve the energy-dissipation structure of the continuous system, particularly the rate-of-change law for gravitational potential energy. At the continuous level, gravity is conservative, with potential-energy density \(\rho g z\).
Here, $\rho$ is the mass density, $g$ is the gravitational constant, and $z$ is the vertical coordinate.
In fact, the force $\Bf$ and the gravitational potential energy are defined as 
\[
\Bf=-\nabla(\rho g z)=-\rho g\nabla z,
\qquad
E_{\rm pot}(t)=\int_{\Omega(t)} \rho g z\,\Rd x .
\]
The main difficulty is that \(E_{\rm pot}(t)\) is an integral in the evolving domain \(\Omega(t)\). At the continuous level, the gravitational work is exactly coupled to the temporal variation of this moving-domain integral:
\[
E_{\rm pot}(t_{m+1})-E_{\rm pot}(t_m)
=
\int_{\Omega(t_{m+1})}\rho g z\,\Rd x
-
\int_{\Omega(t_m)}\rho g z\,\Rd x
=
-\int_{t_m}^{t_{m+1}} \int_{\Omega(t)}\Bf\cdot \Bu \,\Rd x\,\Rd t ,
\]
where $\Bu$ is the fluid velocity.
This rate-of-change identity follows from the transport theorem, but it is no longer automatically preserved at the fully discrete level. 

For the Navier--Stokes free-boundary problem with a contact line and gravity considered in this work, the total energy also consists of the kinetic energy, the free-surface energy, the solid-wall energy, namely
\[
E_{\rm kin}(t)= \frac{\rho}{2} \int_{\Omega(t)} |\Bu|^2\,\Rd x,
\qquad
E_{\rm fs}(t) \propto |\Gamma_{\rm F}(t)|,
\qquad
E_{\rm w}(t) \propto |\Gamma_{\rm s}(t)|,
\]
where $\Gamma_F(t)$ is the free surface, $\Gamma_s(t)$ is the part of the solid substrate wetted by the fluid, and $|\cdot|$ denotes surface area; see Figure~\ref{fig:domain}. 
Among these terms, the kinetic and gravitational potential energies are moving-domain volume integrals, whereas the free-surface and wall energies are integrals over moving boundaries. 
The rate-of-change law for the kinetic energy in an ALE formulation was studied in \cite{Gao2024_rotation}, where a first-order structure-preserving method was developed. Related structure-preserving discretizations for moving-domain integrals with linear integrands have been investigated for hyperbolic conservation laws and MMPDE-based moving meshes \cite{TangTang2003SINUM,DziukElliott2013MC}. 
For zeroth-order integrands, the preservation of such a rate-of-change law is often manifested as volume conservation, e.g., for \(\int_{\Omega(t)}1\,\Rd x\), see \cite{Bao2021}. 
For the surface energy, which is proportional to the surface area, the associated rate-of-change law was preserved in \cite{BGN2015} by the well-known BGN formulation. This line of research is closely related to structure-preserving methods for curvature flows, such as surface diffusion, where preserving the area-decreasing and volume-conservation structures is also a primary objective. In this context, several approaches have been developed, including the use of Lagrange multipliers in variational formulations~\cite{Gao-Li-2025,GarckeJiangSuZhang2025SISC,Zhang2026} and the introduction of intermediate geometric quantities based on the homotopy idea of Jiang and Li~\cite{Jiang2021}. These developments have led to mature methods that exactly preserve volume in three dimensions and area in two dimensions~\cite{Bao2021,Garcke2023}. 
Overall, these works demonstrate that preserving the discrete rate-of-change laws of geometric quantities is essential for robust long-time simulations.

However, in ALE free-boundary flows, the gravitational potential energy is more tightly coupled to the Navier--Stokes equations, making the preservation of its rate of change within the total energy dissipation law more subtle.
The domain and boundary geometries are governed by the mesh velocity, whereas the corresponding energy contribution must enter the energy estimate of the momentum equation. 
Thus the discretization has to handle the fact that the momentum test function is not the ALE mesh velocity, and it must pair the gravitational work with the moving-domain potential-energy variation at the fully discrete level. Moreover, the integral is more complicated than volume conservation because the integrand depends on the geometric position of the mesh. These features make the discrete balance between gravitational potential energy and gravitational work nontrivial in an ALE formulation.

In many ALE discretizations for free-boundary flows, gravity is treated directly as a body force on the right-hand side of the momentum equation. Although this treatment is natural on fixed domains, it does not automatically produce the correct potential-energy variation on a moving domain. If the gravitational work is evaluated on the old domain $\Omega_h^m$, while the potential-energy change is determined by the geometric deformation from $\Omega_h^m$ to $\Omega_h^{m+1}$, the discrete energy law contains a residual induced by the geometric defect.
Although the temporal error in the numerical solution is of the same order, the key issue is that this residual is not balanced in the discrete energy-dissipation law and therefore destroys the sign structure needed in the energy estimate. In near-equilibrium and long-time computations, the uncompensated residual may pollute the relaxation process and, in the pure-gravity setting analyzed below, the discrete energy estimate no longer implies the expected decay of the velocity. Related spurious-current phenomena are well documented for surface-tension balanced-force discretizations and two-phase flow computations \cite{Ganesan2007,BGN2013_spurious,Patel2017,Nangia2019}; the mechanism considered here is the mismatch in the rate-of-change law for the gravitational potential energy.

This paper develops a gravity-consistent body-fitted ALE discretization in which the discrete gravitational work exactly matches the variation of the gravitational potential energy, thereby preserving the energy-dissipation structure of the fully discrete system.
The key step is to rewrite the gravitational work as a free-surface geometric flux compatible with the ALE motion, thereby resolving the mismatch between the momentum test function and the ALE mesh velocity. After integration by parts, the bulk divergence term is absorbed into a generalized dynamic pressure, the solid-wall flux vanishes by impermeability, and the remaining free-surface flux is linked to the normal mesh motion through the ALE kinematic condition. We then introduce a family of intermediate geometries between two consecutive discrete free surfaces and represent the gravitational potential-energy variation over one time step as a time integral over these geometries. Since the mesh moves linearly within a time step, the product of the height function and the pulled-back normal vector is a polynomial of degree at most three in the intermediate time variable. 
With the help of Simpson's rule, we can evaluate this integral exactly. The resulting gravitational term is paired with the discrete variation of $\rho g\int_{\Omega_h(t)}z\,\Rd x$, yielding a discrete energy-dissipation law that includes gravitational potential energy. On this basis, we also prove a long-time relaxation property in the pure-gravity setting considered in this paper: as $T\to\infty$, the discrete velocity field decays to zero, excluding persistent spurious velocities caused by gravity-inconsistent discretization.
At the level of the gravitational discretization, the treatment is not tied to a specific contact-angle condition.
We also investigate the dynamic contact-angle models based on generalized Navier slip, and the Ren--E model. For the long-time computations reported in this paper, we use the existing BGN/MDR ALE mesh-update framework \cite{Hu2022,Hu2026_Droplet} to maintain mesh stability.

Let us summarize the main properties of the proposed scheme:
\begin{enumerate}
\item \textbf{Gravity-consistent discretization:}
The gravitational contribution is discretized consistently with the variation of the gravitational potential energy, leading to a total energy-dissipation law including gravity.

\item \textbf{Eliminating spurious velocities:}
The scheme drives the discrete velocity to zero in the long-time regime, excluding gravity-induced spurious velocities near equilibrium.

\item \textbf{Exact volume conservation:}
The proposed ALE-FEM scheme preserves the volume exactly at the fully discrete level.

\item \textbf{Good mesh quality:}
Combined with the BGN formulation and the MDR mesh extension, the method maintains stable meshes for long-time moving-contact-line simulations.
\end{enumerate}

The rest of the paper is organized as follows. 
Section \ref{sec2} introduces the continuous model and the boundary conditions. 
Sections \ref{sec3} and \ref{sec:theorem} present the fully discrete ALE-FEM schemes, and prove structure-preserving properties. 
Section \ref{sec:Newton} discusses the associated nonlinear system and the Newton solver for the system. 
Finally, Section \ref{sec:numerical} provides numerical validation, followed by concluding remarks in Section \ref{sec:conclusion}.

\begin{figure}[htbp]
\centering
\includegraphics[width=0.45\textwidth]{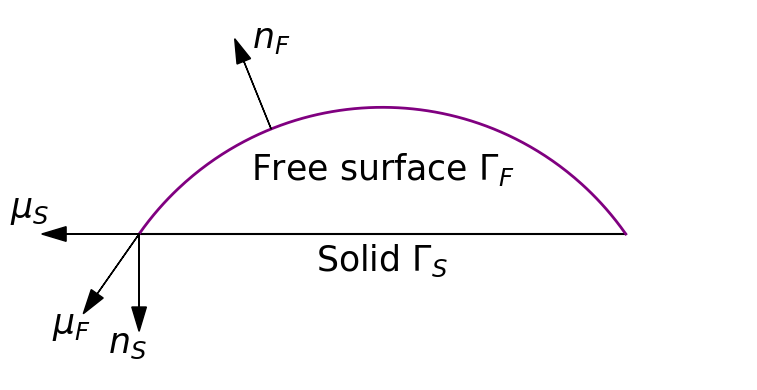}
\includegraphics[width=0.45\textwidth]{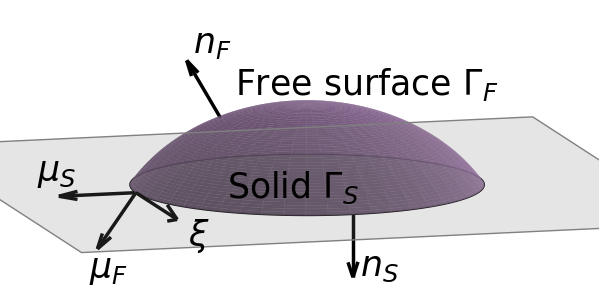}
\caption{Geometrical configurations.}
\label{fig:domain}
\end{figure}

\section{The mathematical model}
\label{sec2}
We consider an incompressible Newtonian fluid droplet with constant density $\rho$, dynamic viscosity $\mu$, and surface tension coefficient $\gamma$, occupying a time-dependent domain $\Omega(t)\subset\mathbb{R}^d$, $d\in\{2,3\}$. The boundary $\partial\Omega(t)$ consists of the free surface $\Gamma_F(t)$ and the solid substrate $\Gamma_s(t)$, which meet at the contact set
$\Gamma_c(t):=\overline{\Gamma_F(t)}\cap\overline{\Gamma_s(t)}.$
We denote by $\Bn_F$ and $\Bn_s$ the unit outward normals to $\Gamma_F(t)$ and $\Gamma_s(t)$, respectively, and by $\bm{\mu}_F$ and $\bm{\mu}_s$ the unit outward conormals to $\Gamma_F(t)$ and $\Gamma_s(t)$ at $\Gamma_c(t)$, respectively. 
The contact angle $\theta_c$ is defined as the angle between $\bm{\mu}_F$ and $\bm{\mu}_s$, as illustrated in Figure~\ref{fig:domain}.

After nondimensionalizing the system using a characteristic length $L$ and velocity $U$, the system can be simplified by using two dimensionless parameters: the Reynolds number and the Weber number,
\[
Re=\frac{\rho U L}{\mu},
\qquad
W\!e=\frac{\rho U^2 L}{\gamma}.
\]

Let $\Bu$ and $p$ denote the dimensionless velocity and pressure, respectively. 
We consider the following incompressible Navier--Stokes equations in $\Omega(t)$ 
\begin{subequations}
\label{2:PDE}
\begin{align}
\partial_t \Bu + \Bu\cdot\nabla \Bu - \nabla\cdot\bm{\sigma} &= -g\Be_z
&& \text{in } \Omega(t), \label{2:PDE-1} \\
\nabla\cdot \Bu &=0
&& \text{in } \Omega(t), \label{2:PDE-2}
\end{align}
\end{subequations}
where $g$ is the dimensionless gravitational acceleration, $\Be_z$ is the unit vector in the vertical direction, and the Cauchy stress tensor and the rate-of-strain tensor are defined by
\[
\bm{\sigma}(\Bu,p)=\frac{2}{Re}\,D(\Bu)-p\,\mathbb{I},
\qquad
D(\Bu)=\frac12\bigl(\nabla\Bu+(\nabla\Bu)^\top\bigr),
\]
where $p$ is the non-dimensional pressure.
For convenience, we consider the $x$-$z$ axis in 2D and the $x$-$y$-$z$ axis in 3D, where $z$ is always the vertical direction. 

On the free surface, the stress balance with surface tension gives the following.
\begin{equation}
\bm{\sigma}\,\Bn_F = -\frac{1}{W\!e}\,H\,\Bn_F
\qquad \text{on } \Gamma_F(t),
\label{2:PDE-bF}
\end{equation}
where $H$ denotes the mean curvature of $\Gamma_F(t)$.

On the solid substrate, we impose the impermeability condition and the Navier slip condition:
\begin{align}
\mathbb{P}_s \left( \bm{\sigma} \Bn_s + \beta_s \Bu \right) 
&= \bm{0},
\quad &&\text{on } \Gamma_s(t)
\label{2:PDE-bs2-1} \\
\Bu \cdot \Bn_s &= 0, 
\quad &&\text{on } \Gamma_s(t), \label{2:PDE-bs2-2}
\end{align} 
where $\mathbb{P}_s=\mathbb{I}-\Bn_s\otimes\Bn_s$ is the tangential projection operator on $\Gamma_s(t)$,
and $\beta_s \ge 0$ is the dimensionless slip coefficient.

At the contact set $\Gamma_c(t)$, the dynamic contact angle $\theta_c$ is prescribed by
the Ren--E model \cite{Ren2007,Ren2011}:
\begin{align}
  \bm{\mu}_F\cdot\bm{\mu}_s &= \cos\theta_c
  && \text{on } \Gamma_c(t), \label{eq:PDE-cont1} \\
  \beta_c\, U_c &= \frac{1}{W\!e}\,
  \bigl(\cos\theta_Y - \cos\theta_c\bigr)
  && \text{on } \Gamma_c(t), \label{eq:PDE-cont2}
\end{align}
where $\theta_Y$ is Young's equilibrium contact angle,
$U_c := \mathbb{P}_s \Bu\cdot\bm{\mu}_s = \Bu\cdot\bm{\mu}_s$ 
is the contact-line speed (the component of the slip velocity normal to
$\Gamma_c$ within $\Gamma_s$), and $\beta_c \ge 0$ is the dimensionless contact-line friction coefficient.
Here \(\bm{\mu}_s\) lies in the tangent plane of the solid boundary, so that \(\Bn_s\cdot \bm{\mu}_s = 0\). 
At the moving contact line, the dynamic contact angle satisfies
\[
\bm{\mu}_F
=
\cos\theta_c\,\bm{\mu}_s
+
\sin\theta_c\,\Bn_s
\qquad \text{on } \Gamma_c(t).
\]

The total energy consists of the kinetic energy of the fluid, the surface energy of the free boundary, the wall--adhesion energy on the solid substrate, and the gravitational potential energy:
\begin{equation}\label{def_energy}
E(t):=
\underbrace{\frac12\int_{\Omega(t)} |\Bu|^2\,\Rd x}_{E_{\rm kin}}
\underbrace{ + \frac{1}{W\!e} |\Gamma_F(t)| }_{E_{\rm fs}}
\underbrace{ -\frac{\cos\theta_Y}{W\!e} |\Gamma_s(t)| }_{E_{\rm w}}
\underbrace{ + g\int_{\Omega(t)} z\,\Rd x}_{E_{\rm pot}}.
\end{equation}
The following dissipation property holds for the continuous model.
\begin{lemma}
The total energy \(E(t)\) defined in \eqref{def_energy} satisfies
\begin{equation}\label{lem:engdisp}
\frac{\Rd}{\Rd t}E(t)
=
-\frac{2}{Re}\int_{\Omega(t)} |D(\Bu)|^2\,\Rd x
-\int_{ \Gamma_s(t)} \beta_s|\mathbb{P}_s\Bu|^2 \Rd s
-\int_{ \Gamma_c(t)} \beta_c |\Bu \cdot \bm \mu_s|^2 \Rd s 
\le 0.
\end{equation}
\end{lemma}

\begin{remark}\upshape
\label{rmk:energy}
The total energy is strictly positive when the reference height is chosen such
that \(z=0\) on the solid boundary \(\Gamma_s(t)\). With this choice, the
gravitational potential energy is non-negative, and the kinetic energy is also
non-negative. 
Since
\(\Gamma_F(t)\) is a curved free interface, while \(\Gamma_s(t)\) is its
projected contact region on the flat substrate, we have
\(|\Gamma_F(t)|>|\Gamma_s(t)|.
\)
Therefore,
\[
\frac{1}{W\!e}
\left(
|\Gamma_F(t)|-\cos\theta_Y|\Gamma_s(t)|
\right)
>0 ,
\]
with $W\!e > 0$.
Thus the surface energy is strictly positive, and consequently the total
energy satisfies
\(E(t)>0\).
\end{remark}

\subsection{ALE weak formulation}
\label{sec3.1}
Testing \eqref{2:PDE-1} with \( \Bv \),
testing \eqref{2:PDE-2} with $q$,
testing \eqref{2:PDE-bs2-2} with $\lambda_s$
and using the Navier slip condition \eqref{2:PDE-bs2-1} together with \eqref{2:PDE-bF}, 
we obtain
\begin{subequations}
\label{eq:continuous_weak}
\begin{align}
&(\partial_t\Bu+(\Bu\cdot\nabla)\Bu,\Bv)
+
\frac{2}{Re}
(D(\Bu),D(\Bv))
- (p,\nabla\cdot\Bv)
+ (\nabla\cdot\Bu,q)
+
\beta_s
\langle
\mathbb P_s\Bu,
\mathbb P_s\Bv
\rangle_{\Gamma_s(t)}
\notag\\
&\qquad
- \langle \kappa_s \Bn_s, \Bv \rangle_{\Gamma_s(t)}
+ \langle \lambda_s \Bn_s, \Bu \rangle_{\Gamma_s(t)}
+
\frac{1}{W\!e}
\langle H\Bn_F,\Bv\rangle_{\Gamma_F(t)}
= -g(\Be_z,\Bv),
\label{eq:weak_momentum}
\end{align}
\end{subequations}
where $\kappa_s$ is the Lagrange multiplier associated with the
impermeability constraint $\Bu\cdot\Bn_s=0$ on $\Gamma_s(t)$, and
$\lambda_s$ denotes the corresponding test function.
For simplicity, in the above formulation, we use the notation 
\[
(\Bu,\Bv):=\int_{\Omega(t)}\Bu\cdot\Bv\,\Rd x,
\qquad
\langle\Bu,\Bv\rangle_{\Gamma(t)}
:=
\int_{\Gamma(t)}\Bu\cdot\Bv\,\Rd s .
\]
For the mean curvature $H$ on $\Gamma_F(t)$, the geometric identity
\(
-\Delta_\Gamma \mathrm{id}=H\Bn_F
\)
and integration by parts on \(\Gamma_F(t)\) give, for any admissible test function \(\bm{\eta}\),
\begin{align}
-\langle &H\Bn_F,\bm{\eta}\rangle_{\Gamma_F(t)}
=
-\langle\nabla_\Gamma\mathrm{id},\nabla_\Gamma\bm{\eta}\rangle_{\Gamma_F(t)}
+
\langle\bm{\mu}_F,\bm{\eta}\rangle_{\Gamma_c(t)}
\notag\\
&\qquad =
-\langle\nabla_\Gamma\mathrm{id},\nabla_\Gamma\bm{\eta}\rangle_{\Gamma_F(t)}
+
\langle
\cos\theta_c\,\bm{\mu}_s+\sin\theta_c\,\Bn_s,
\bm{\eta}
\rangle_{\Gamma_c(t)}
\notag\\
&\qquad =
-\langle\nabla_\Gamma\mathrm{id},\nabla_\Gamma\bm{\eta}\rangle_{\Gamma_F(t)}
+
\langle
\cos\theta_Y\,\bm{\mu}_s
+\sin\theta_c\,\Bn_s
-W\!e\,\beta_c(\Bu\cdot\bm{\mu}_s)\bm{\mu}_s,
\bm{\eta}
\rangle_{\Gamma_c(t)}. 
\label{eq:meanH}
\end{align}
Here we have used the relaxed contact-angle condition in \eqref{eq:PDE-cont2}.

We now rewrite the continuous model in an ALE framework. Let
\(
\Phi(t):\Omega(0)\to\Omega(t)
\)
be the ALE map and let
\(
\Bw=\partial_t\Phi\circ\Phi^{-1}
\)
be the associated mesh velocity. 
The ALE mesh velocity is introduced to track the evolution of the moving fluid domain. 
In particular, the normal velocity of the mesh must coincide with the normal velocity of the fluid on the free surface:
\[
\Bw\cdot\Bn_F=\Bu\cdot\Bn_F
\qquad\text{on }\Gamma_F(t),
\]
while its tangential component may be chosen to improve the mesh quality.
On the solid substrate, the mesh is required to remain attached to the solid boundary. 
Together with the impermeability condition for the
fluid, this gives
\[
\Bw\cdot\Bn_s=\Bu\cdot\Bn_s=0
\qquad\text{on }\Gamma_s(t).
\]
Consequently, at the contact set $\Gamma_c(t)$, the fluid and mesh
velocities have the same component in the conormal direction
$\bm{\mu}_s$:
\[
(\mathbb P_s\Bu)\cdot\bm{\mu}_s
=
(\mathbb P_s\Bw)\cdot\bm{\mu}_s
=
\Bw\cdot\bm{\mu}_s
=
\Bu\cdot\bm{\mu}_s .
\]

The ALE material derivative is defined by
\[
\partial_t^\bullet \Bu
:=
\partial_t\Bu+\Bw\cdot\nabla\Bu.
\]
We introduce the skew-symmetric trilinear form
\[
b(\Bz;\Bu,\Bv)
:=
\frac12(\Bz\cdot\nabla\Bu,\Bv)
-
\frac12(\Bz\cdot\nabla\Bv,\Bu).
\]
Since
$(\Bu-\Bw)\cdot\Bn=0$ on $\Gamma(t)$
and $\nabla\cdot\Bu=0$, 
integration by parts yields
\[
\bigl(((\Bu-\Bw)\cdot\nabla)\Bu,\Bv\bigr)
=
b(\Bu-\Bw;\Bu,\Bv)
+
\frac12
(\Bu,\Bv\,\nabla\cdot\Bw).
\]
Consequently, the first term in \eqref{eq:weak_momentum} can be rewritten as
\[
\begin{aligned}
\bigl(\partial_t\Bu+(\Bu\cdot\nabla)\Bu,\Bv\bigr)
&=
(\partial_t^\bullet\Bu,\Bv)
+
\frac12(\Bu,\Bv\,\nabla\cdot\Bw)
+
b(\Bu-\Bw;\Bu,\Bv).
\end{aligned}
\]

\section{Fully Discrete ALE--FEM Schemes}
\label{sec3}
In this section, we present the fully discrete ALE--FEM scheme. The main idea
is to decompose the construction of the mesh velocity into two parts. We first
determine the fluid variables \(\Bu\), \(p\), together with the mesh velocity
on the free boundary \(\Gamma_F\) using the BGN method. 
It is then extended
to the bulk mesh by the MDR method.

The rest of this section is organized as follows. In
Subsection~\ref{sec:discretization}, we introduce the finite element spaces and
the basic fully discrete notation. In Subsection~\ref{sec:direct}, we
present the direct ALE--FEM discretization. The proposed decoupled schemes in
two and three dimensions are then given in
Subsections~\ref{sec:BGN-2d} and~\ref{sec:BGN-3d}, respectively. Finally,
Subsection~\ref{sec:MDR} describes the MDR extension used to propagate the
boundary mesh velocity into the interior of the computational domain.

\subsection{Discrete Setting and Finite Element Spaces}
\label{sec:discretization}
Let \(\Omega_h^0\) be the initial computational domain, equipped with a shape-regular and quasi-uniform triangulation \(\mathcal{K}_h^0\) of mesh size \(h\). Its boundary is decomposed into the discrete free boundary and the discrete solid boundary,
\[
\vspace{-5pt}
\Gamma_h^0=\Gamma_{F,h}^0\cup\Gamma_{s,h}^0,
\qquad
\Gamma_{c,h}^0=\Gamma_{F,h}^0\cap\Gamma_{s,h}^0,
\]
where \(\Gamma_{c,h}^0\) denotes the discrete contact line.

Let \(t_m=m\tau\), \(m=0,1,\ldots,M\), be a uniform partition of the time interval. At time \(t_m\), the discrete domain is denoted by \(\Omega_h^m\), with boundary decomposition
\[
\Gamma_h^m=\Gamma_{F,h}^m\cup\Gamma_{s,h}^m,
\qquad
\Gamma_{c,h}^m=\Gamma_{F,h}^m\cap\Gamma_{s,h}^m.
\]
The mesh is updated by the discrete ALE flow map
\[
\psi_h^{m+1}:=\bbI+\tau\Bw_h^{m+1},
\qquad
\Omega_h^{m+1}=\psi_h^{m+1}(\Omega_h^m).
\]
For convenience, we set \(\Omega_h^{-1}=\Omega_h^0\) and \(\psi_h^0=\bbI\).
The initial velocity is given as \(\Bu_h^0 = \Ci_h \Bu_0\) on \(\Omega_h^0\), 
where \(\Ci_h\) denotes the interpolation operator. 
In addition, we set \(\Bw_h^0 = \mathbf{0}\).

We denote by \(\BX_h^m\) the position vector of the discrete domain at time \(t_m\). Hence the vertical coordinate at the new time level is
\[
z_h^{m+1}
=
\BX_h^{m+1}\cdot\Be_z
=
(\BX_h^m+\tau\Bw_h^{m+1})\cdot\Be_z .
\]

Let \(\mathcal{K}_h^m\) be the simplicial triangulation of \(\Omega_h^m\). We assume that \(\mathcal{K}_h^m\) remains shape-regular and quasi-uniform for all time steps. For the bulk velocity and pressure, we use the Taylor--Hood \(P_2\)-\(P_1\) finite element pair $(\mathcal{U}_h^m, \mathcal{Q}_h^m)$
and $\mathcal{W}_h^m$ is the linear finite element space
\[
\begin{aligned}
\mathcal{U}_h^m
&:=
\bigl\{
\Bv_h\in H^1(\Omega_h^m)^d:
\Bv_h|_K\in\mathbb{P}_2(K)^d
\quad \forall K\in\mathcal{K}_h^m
\bigr\},
\\
\mathring{\mathcal{U}}_h^m
&:=
\bigl\{
\Bv_h\in H_0^1(\Omega_h^m)^d:
\Bv_h|_K\in\mathbb{P}_2(K)^d
\quad \forall K\in\mathcal{K}_h^m
\bigr\},
\\
\mathcal{Q}_h^m
&:=
\bigl\{
q_h\in H^1(\Omega_h^m):
q_h|_K\in\mathbb{P}_1(K)
\quad \forall K\in\mathcal{K}_h^m
\bigr\}, 
\\
\mathcal{W}_h^m 
&:= \bigl\{ \Bw_h \in H^1(\Omega_h^m)^d : \Bw_h|_K \in P^1(K) \; \forall K \in \mathcal{K}_h^m \bigr\}.
\end{aligned}
\]
As established in \cite{Boffi2013}, these spaces satisfy the inf-sup condition: there exists a constant \(\beta>0\), independent of \(h\), such that for all \(q_h\in\mathcal{Q}_h^m\),
\[
\beta\|q_h-\bar q_h\|_{L^2(\Omega_h^m)}
\le
\sup_{\Bv_h\in\mathring{\mathcal{U}}_h^m\setminus\{\mathbf{0}\}}
\frac{(\nabla\cdot\Bv_h,q_h)_{\Omega_h^m}}
{\|\Bv_h\|_{H^1(\Omega_h^m)}},
\]
where
\(
\bar q_h
:=
\frac{1}{|\Omega_h^m|}
\int_{\Omega_h^m}q_h\,\Rd x .
\)
In fact, all other inf-sup stable pairs can be used; see \cite{BGN2020HoNA}.

On a discrete boundary, we denote the scalar-valued and vector-valued continuous piecewise linear spaces by \(S_h(\cdot)\) and \(S_h(\cdot)^d\), respectively. 
The corresponding piecewise quadratic spaces are denoted by \(S_{h,2}(\cdot)\) and \(S_{h,2}(\cdot)^d\).

We use the standard mass-lumped inner product defined in Definition~43 of~\cite{BGN2020HoNA}, denoted by
$\langle\cdot,\cdot\rangle_{\Gamma_h^m}^{h}$.
To impose the normal constraint on the solid boundary, we introduce the weighted nodal pairing
\begin{align}\label{eq:mass_lumping_nodal}
\langle \lambda_s \Bn_s, \Bu_h\rangle_{\Gamma_{s,h}^m}^{s,h}
:=
\sum_{\sigma\in\mathcal T_h(\Gamma_{s,h}^m)}
\frac{ |\sigma|}{N_2}
\sum_{P\in\mathcal N_2(\sigma)}
\lambda_s(P) \Bn_s(P) \cdot \Bu_h(P),
\end{align}
where $\mathcal{N}_2(\sigma)$ is the collection of all vertices 
and edge midpoints corresponding to the degrees of freedom for quadratic elements, 
and \( N_2=\dim\mathbb P_2(\sigma)=\binom{d+2}{d}. \)

\subsection{Two-dimensional direct ALE--FEM scheme}
\label{sec:direct}
In this subsection, we first present a two-dimensional direct ALE--FEM scheme. 
In this reference scheme, the gravitational force is treated in the standard way as a body-force term on the right-hand side of the momentum equation. 
Although this treatment is natural and leads to a direct extension of the ALE--BGN--MDR method without gravity, 
it does not, in general, match the discrete gravitational work with the variation of the gravitational potential energy on the moving domain. 
After introducing the scheme, we therefore analyze the resulting inconsistent error in the discrete energy balance.

For the time discretization of \eqref{2:PDE-1}, 
the ALE time-derivative term and the geometric-divergence term are treated together to preserve
the energy-dissipation structure. Let $\widetilde{\Bv}$ denote the ALE
transport of $\Bv$, satisfying
\[
\partial_t^\bullet\widetilde{\Bv}=0,
\qquad
\widetilde{\Bv}(\cdot,t_m)=\Bv.
\]
Then, by the transport theorem,
\begin{align*}
&\left( \partial_t^\bullet \Bu,\Bv \right)_{\Omega_h^m}
+
\frac{1}{2}\left( \Bu,\Bv\,\nabla\cdot\Bw \right)_{\Omega_h^m}
=
\frac{1}{2}
\left( \partial_t^\bullet \Bu,\Bv \right)_{\Omega_h^m}
+
\frac{1}{2}
\frac{\Rd}{\Rd t}\Big|_{t=t_m}
\left( \Bu,\tilde{\Bv} \right)_{\Omega(t)} \notag \\
&\qquad 
\approx \frac{1}{2\tau}
\left(
\Bu_h^{m+1}\circ\psi^m_h - \Bu_h^m,
\Bv_h\circ\psi^m_h
\right)_{\Omega_h^{m-1}} 
+ \frac{1}{2\tau}
\left[
\left(
\Bu_h^{m+1},\Bv_h
\right)_{\Omega_h^m}
-
\left(
\Bu_h^m,\Bv_h\circ\psi^m_h
\right)_{\Omega_h^{m-1}}
\right].
\end{align*}
This is the temporal discretization used in the fully discrete ALE--FEM schemes below.

Based on the continuous formulation \eqref{eq:weak_momentum} and \eqref{eq:meanH},
we replace the continuous variables and test spaces by their finite element counterparts on the discrete domain $\Omega_h^m$, and use the
temporal discretization introduced above. 
This leads to the following
fully discrete ALE--FEM scheme.

Given \(\Bu_h^m\), \(\Bw_h^m\) and domain \(\Omega_h^m\), we seek
\(
(\Bu_h^{m+1}, p_h^{m+1}, \kappa_{s,h}^{m+1}, H_h^{m+1}, \widetilde \Bw_h^{m+1}, \varphi_{c,h}^{m+1})
\)
such that
\begin{subequations}\label{BGN-old}
\begin{align}
&\frac{1}{2\tau}( \Bu_h^{m+1}, \Bv)_{\Omega_h^{m}}
+ \frac{1}{2\tau}( \Bu_h^{m+1}\circ \psi_h^m , \Bv \circ \psi_h^m )_{\Omega_h^{m-1}}
+ b\!\left( \Bu_h^m - \Bw_h^m; \Bu_h^{m+1}\circ \psi_h^m, \Bv \circ \psi_h^m\right)_{\Omega_h^{m-1}} \notag\\
&\quad + (\frac{2}{Re} D( \Bu_h^{m+1}), D(\Bv))_{\Omega_h^{m}}
- (p_h^{m+1}, \nabla \cdot \Bv)_{\Omega_h^{m}}
+ (q, \nabla \cdot \Bu_h^{m+1})_{\Omega_h^{m}}
+ \frac{1}{W\!e} \langle H_h^{m+1} \Bn_{F,h}^m, \Bv \rangle_{\Gamma^{m}_{F,h}} \notag\\
&\quad - \langle \kappa_{s,h}^{m+1} \Bn_s, \Bv \rangle_{\Gamma^{m}_{s,h}}^{ s, h }
+ \langle \lambda_s \Bn_s, \Bu_h^{m+1} \rangle_{\Gamma^{m}_{s,h}}^{ s, h }
+ \beta_s \langle \mathbb P_s\Bu_h^{m+1},\mathbb P_s\Bv\rangle_{\Gamma_{s,h}^m} 
+ \langle \widetilde \Bw_h^{m+1} \cdot \Bn_s, \phi_c \rangle^{ h }_{\Gamma_{c,h}^m} \notag\\
&\quad =
\frac{1}{\tau}(\Bu_h^{m}, \Bv\circ \psi_h^m )_{\Omega_h^{m-1}}
- g(\Be_z,\Bv)_{\Omega_h^m},
\label{BGN-fluid-old} \\
&-\langle H_h^{m+1} \Bn_{F,h}^m, \bm{\eta} \rangle_{\Gamma_{F,h}^m}^{ h }
= -\langle \nabla_{\Gamma} (\bbI + \tau \widetilde \Bw_h^{m+1}), \nabla_{\Gamma} \bm{\eta} \rangle_{\Gamma_{F,h}^m} 
+ \langle \varphi_{c,h}^{m+1} \Bn_s, \bm{\eta} \rangle^{ h }_{\Gamma_{c,h}^m} \notag\\
&\hspace{50pt}
+ \int_{\Gamma_{c,h}^m}
\left[
\cos \theta_Y \bm{\mu}_s
- W\!e\,\beta_c (\widetilde{\Bw}_h^{m+1}\cdot\bm{\mu}_s) \bm{\mu}_s
+ \sin\theta_{c,h}^{m}\,\Bn_s
\right]\cdot \bm{\eta}\,\Rd s, 
\label{BGN-a-old} \\
&\langle \widetilde \Bw_h^{m+1} \cdot \Bn_{F,h}^m, \phi_F \rangle^{ h }_{\Gamma_{F,h}^m}
=
\langle \Bu_h^{m+1} \cdot \Bn_{F,h}^m, \phi_F \rangle_{\Gamma_{F,h}^m},
\label{BGN-b-old}
\end{align}
\end{subequations}
for all \((\Bv, q, \lambda_s, \phi_F, \bm{\eta}, \phi_c) \in 
\mathcal{U}_h^m \times \mathcal{Q}_h^m\times S_{h,2}(\Gamma_{s,h}^m) \times S_h(\Gamma_{F,h}^m) \times S_h(\Gamma_{F,h}^m)^d \times S_h(\Gamma_{c,h}^m)\),
where $\theta_{c,h}^m$ in \eqref{BGN-a-old} is defined by 
\begin{equation}\label{eq:def_theta_ch}
\theta_{c,h}^m
=
\arccos\!\left(\Pi_{[-1,1]}(c_{c,h}^m)\right),
\qquad
\Pi_{[-1,1]}(r):=\min\{1,\max\{-1,r\}\},
\end{equation}
with $c_{c,h}^m
:=
\mathcal{I}_{h,0}\left( \cos\theta_Y
-
{W\!e}\,\beta_c
\widetilde{\Bw}_h^m\cdot\bm{\mu}_s \right)$, where $\mathcal I_{h,0}$ denotes the elementwise midpoint interpolation operator onto the space of piecewise constant functions on $\Gamma_{c,h}^m$.
For the two-dimensional case, the conormals $\bm \mu_s$ at the contact points $\Gamma_c$ are constants, such as $\bm \mu_s = (-1, 0)^\top$ for the left contact point and $\bm \mu_s = (1, 0)^\top$ for the right contact point, see the left subfigure of Figure~\ref{fig:domain}.

In order to analyze the gravitational contribution in the scheme \eqref{BGN-old},
we introduce the intermediate ALE configurations between
$\Omega_h^m$ and $\Omega_h^{m+1}$, which will be used to represent
the moving-domain contributions.
For \(\vartheta\in[0,1]\), define the intermediate ALE flow map,
\[
\psi_h^m(\vartheta):=\bbI+\tau\vartheta\Bw_h^{m+1}.
\]
The corresponding intermediate domain and boundary are
\[
\Omega_h^m(\vartheta):=\psi_h^m(\vartheta)(\Omega_h^m),
\qquad
\Gamma_h^m(\vartheta):=\psi_h^m(\vartheta)(\Gamma_h^m).
\]
Restricting $\psi_h^m(\vartheta)$ to the free boundary $\Gamma^m_{F,h}$ gives the parametrization of $\Gamma^{m}_{F,h}(\vartheta)$:
\begin{align}\label{eq:X_2d}
\BX_{F,h}^m(\vartheta)
:=
\bbI_{\Gamma_{F,h}^m}
+
\tau\vartheta\widetilde{\Bw}_h^{m+1},
\qquad
\vartheta\in[0,1],
\end{align}
where
\(
\widetilde{\Bw}_h^{m+1}
=
\Bw_h^{m+1}|_{\Gamma_{F,h}^m}
\)
is the mesh velocity on the discrete free boundary. 
Further, the height of the intermediate discrete free boundary $\Gamma_{F,h}^m(\vartheta)$ is denoted as 
\begin{equation}\label{eq:X_2d_z}
z^m_h(\vartheta) := \BX_{F,h}^m(\vartheta) \cdot \Be_z.
\end{equation}
We emphasize that $z_h^m(\vartheta)$ is defined on $\Gamma_{F,h}^m$, and satisfies
\begin{align}\label{def:zdom}
z^m_h(\vartheta) = z\circ \BX_{F,h}^m(\vartheta) .
\end{align}
In two dimensions, let \(s\) denote the arc-length parameter on \(\Gamma_{F,h}^m\). The pulled-back tangent vector on the intermediate free boundary is
\begin{align}\label{eq:hat_t_2d}
\hat{\Bt}_{F,h}(\vartheta)
=
\partial_s\BX_{F,h}^m(\vartheta)
=
\Bt_{F,h}^m
+
\vartheta\tau\,\partial_s\widetilde{\Bw}_h^{m+1},
\end{align}
where $\Bt_{F,h}^m$ is the tangent vector on \(\Gamma_{F,h}^m\) such that $(\Bt_{F,h}^m)^\perp = \Bn_{F,h}^m$.
By rotating this tangent vector by \(90^\circ\), we define the pulled-back unnormalized outward normal vector 
\begin{align}\label{eq:hat_n_2d}
\hat{\Bn}_{F,h}(\vartheta) := \bigl(\hat{\Bt}_{F,h}(\vartheta)\bigr)^\perp 
= \Bn_{F,h}^m + \vartheta\tau \bigl(\partial_s\widetilde{\Bw}_h^{m+1}\bigr)^\perp . 
\end{align} 
Here the rotation is chosen so that \(\hat{\Bn}_{F,h}(0) = \Bn_{F,h}^m\) and \(\hat{\Bn}_{F,h}(\vartheta)\) is linear in \(\vartheta\).

With these intermediate configurations at hand, 
we first recall the continuous balance between the gravitational work and the variation of the gravitational potential energy, and then show that this balance is not exactly preserved by the direct discretization.
\begin{align}\label{dEc}\notag
\frac{\Rd}{\Rd t}E_{\rm pot}(t)
&=
g\frac{\Rd}{\Rd t}\int_{\Omega(t)}z\,\Rd x
=
g\int_{\partial\Omega(t)}
z\,\Bu\cdot\Bn\,\Rd s
\\
&
= g\int_{\Omega(t)}
\nabla\cdot(z\Bu)\,\Rd x
= g\int_{\Omega(t)}
\Bu\cdot\nabla z\,\Rd x
= g\int_{\Omega(t)}
\Bu\cdot \boldsymbol{e}_z\,\Rd x.
\end{align}
Here, we have used 
$\Bu\cdot\Bn_s=0$
on $\Gamma_s(t)$,
and the incompressibility condition $\nabla\cdot\Bu=0$.
Thus, we have
\begin{align*}
E_{\textrm{pot}}(t_{m+1}) - E_{\textrm{pot}}(t_{m}) =
\int_{t_m}^{t_{m+1}} g 
\int_{\Omega(t)} \Bu\cdot \boldsymbol{e}_z\,\Rd x \Rd t.
\end{align*}

However, the analogous property cannot be preserved by the direct ALE--FEM scheme \eqref{BGN-old} at the discrete level.
In other words, we can show 
\begin{align}\label{notequal}
    E_{\textrm{pot}}^{m+1}
    -E_{\textrm{pot}}^{m}\not=
    \tau g\int_{\Omega_h^m}
\Bu_h^{m+1}\cdot\Be_z\,\Rd x,
\end{align}
where the right hand side is the gravitational contribution obtained by testing $\Bv_h=\Bu_h^{m+1}$ in \eqref{BGN-fluid-old} and the discrete gravitational potential energy is defined by
\[
E_{\rm pot}^{m}
:=
g\int_{\Omega_h^m} z\,\Rd x.
\]
The exact discrete potential-energy variation induced by the ALE update is
\begin{align}
E_{\rm pot}^{m+1} - E_{\rm pot}^m 
&= g \int_{\Omega_h^{m+1}} z \,\Rd x - g \int_{\Omega_h^m} z \,\Rd x  
= g \int_{0}^{1} \frac{d}{d\vartheta}\int_{\Omega_h^{m}(\vartheta)} z \,\Rd x \, d\vartheta \notag \\
&= \tau g \int_{0}^{1} \int_{\Gamma_{F,h}^{m}(\vartheta)} z\big[ 
(\widetilde{\Bw}_h^{m+1})\circ {(\bm{X}_{F,h}^{m}(\vartheta) )}^{-1} \cdot \Bn_{F,h}^m(\vartheta) \big] \,\Rd s(\vartheta) \, d\vartheta \notag \\
&= \tau g \int_{\Gamma_{F,h}^{m}}  
\Big[\int_{0}^{1} z_h^m(\vartheta) 
\hat{\Bn}_{F,h}(\vartheta) d\vartheta\Big]\,\cdot \widetilde{\Bw}_h^{m+1}  \Rd s,
\label{eq:exact_E_pot}
\end{align}
where the last equality follows from $z_h^m(\vartheta) = z\circ \bm{X}_{F,h}^{m}(\vartheta)$ in \eqref{def:zdom} and the change of variables below. In general, for any scalar function $f$ defined on $\Gamma_{F,h}^{m}(\vartheta)$, we have
\begin{align*}
\int_{\Gamma_{F,h}^{m}(\vartheta)} 
f \Bn^m_{F,h}(\vartheta) \Rd s(\vartheta)
&= \int_{\Gamma_{F,h}^{m}}
[f \circ \bm{X}_{F,h}^{m}(\vartheta) ]
\frac{\hat{\Bn}_{F,h}(\vartheta)}{| \hat \Bt_{F,h}(\vartheta)|} |\partial_s \BX_{F,h}^{m}(\vartheta)| \Rd s\\
&= \int_{\Gamma_{F,h}^{m}} [f \circ \bm{X}_{F,h}^{m}(\vartheta) ]
\hat{\Bn}_{F,h}(\vartheta) \Rd s. 
\end{align*}

By contrast, the gravitational work entering the direct body-force scheme \eqref{BGN-old} is
\begin{align}
&\tau g\int_{\Omega_h^m}
\Bu_h^{m+1}\cdot \Be_z\,\Rd x =
\tau g
\int_{\Omega_h^{m}} \nabla z\cdot \Bu_h^{m+1}\,\Rd x
\notag\\
&\qquad=
\tau g
\int_{\Gamma_{F,h}^m}
z(\Bu_h^{m+1}\cdot\Bn_{F,h}^m)\,\Rd s
+
\tau g
\int_{\Gamma_{s,h}^m}
z(\Bu_h^{m+1}\cdot\Bn_s)\,\Rd s
-
\tau g
\int_{\Omega_h^m}
z\nabla\cdot\Bu_h^{m+1}\,\Rd x
\notag\\
&\qquad=
\tau g
\langle
\widetilde{\Bw}_h^{m+1}\cdot
\Bn_{F,h}^m,z
\rangle_{\Gamma_{F,h}^m}^h
= \tau g
\int_{\Gamma_{F,h}^m}
z(\widetilde{\Bw}_h^{m+1}\cdot\Bn_{F,h}^m)\,\Rd s + \mathcal{O}(\tau h^2), \label{eq:inexact_E_pot}
\end{align}
where we have used \eqref{BGN-fluid-old} and \eqref{BGN-b-old} in the last two equalities.
Comparing \eqref{eq:exact_E_pot} and \eqref{eq:inexact_E_pot}, we observe that there is an inconsistency in discrete energy dissipation
\begin{align}
\mathcal{R}_g^{m+1}:&= E_{\rm pot}^{m+1} - E_{\rm pot}^m - \tau g\int_{\Omega_h^m}
\Bu_h^{m+1}\cdot \Be_z\,\Rd x \notag \\
& =
\tau g
\int_{\Gamma_{F,h}^m}
\widetilde{\Bw}_h^{m+1}\cdot
\left(
\int_0^1 z_h(\vartheta)\hat{\Bn}_{F,h}(\vartheta)\,\Rd\vartheta
-
z\Bn_{F,h}^m
\right)\,\Rd s+ \mathcal{O}(\tau h^2)
\notag\\
&\approx \mathcal{O}(\tau^2+\tau h^2).
\label{eq:gravity-residual-direct}
\end{align}
Although the local consistency residual
\(
\mathcal R_g^{m+1}
=
\mathcal O(\tau^2+\tau h^2)
\)
is small for sufficiently small $\tau$ and $h$, it has no definite sign and does not represent any physical dissipation. 
Consequently, the direct treatment of gravity does not yield a closed discrete energy-dissipation law for the total energy including $E_{\rm pot}^m$. 
In long-time simulations, especially near equilibrium where the true energy variation is already very small, the accumulated effect of \(\mathcal{R}_g^{m+1}\) may become comparable to the physical dissipation. 
As a result, the scheme may produce artificial spurious velocities near equilibrium.

To overcome this issue, we propose a gravitationally consistent scheme in the next section, which preserves the correct discrete energy structure over long times.

\subsection{Two-dimensional gravitationally consistent scheme}
\label{sec:BGN-2d}
Our main idea is to find a better discretization of the work of gravitational force $-g(\boldsymbol{e}_z,\boldsymbol{v})_{\Omega(t)}$ in \eqref{eq:weak_momentum}
than $-g(\boldsymbol{e}_z,\boldsymbol{v})_{\Omega^m_h}$ in \eqref{BGN-fluid-old} so that the
discrete gravitational work is consistent with the change of gravitational potential energy.
We begin with the following reformulation of the continuous problem. By integration by parts, we obtain
\begin{align}\notag
g(\boldsymbol{e}_z,\boldsymbol{v})_{\Omega(t)}
&= g\int_{ \Omega(t)} \nabla z \cdot \Bv \,\Rd x\\
&=
g
\int_{\Gamma_{F}(t)} z 
\Bv\cdot \Bn_{F}\,\Rd s
+g
\int_{\Gamma_{s}(t)} z \Bv\cdot \Bn_{s}\,\Rd s
-g
\int_{ \Omega(t)} z\nabla\cdot\Bv\,\Rd x.
\label{eq:grav_z}
\end{align}
This inspires us to rewrite \eqref{eq:weak_momentum} as
\begin{align}
    &(\partial_t\Bu+(\Bu\cdot\nabla)\Bu,\Bv)
+
\frac{2}{Re}
(D(\Bu),D(\Bv))
-
(\tilde p,\nabla\cdot\Bv)
+(q,\nabla\cdot \Bu)
\notag\\
&- \langle \tilde \kappa_s \Bn_s, \Bv \rangle_{\Gamma_s(t)}
+ \langle \lambda_s \Bn_s, \Bu \rangle_{\Gamma_s(t)}
+ \beta_s
\langle
\mathbb P_s\Bu,
\mathbb P_s\Bv
\rangle_{\Gamma_s(t)}
+ \langle \tilde H\Bn_F,\Bv\rangle_{\Gamma_F(t)}
= 0, \label{eq:weak_momentum_grav}
\end{align}
where $\tilde p$ is the generalized dynamic pressure that absorbs the hydrostatic contribution:
\begin{align}
    \tilde p = p+gz,\quad 
    \tilde \kappa_s = \kappa - gz, \quad
    \tilde H = H/W\!e + gz.
\end{align}
Further, \eqref{eq:meanH} leads to
\begin{subequations}\label{eq:meanH_grav}
\begin{align}
-\langle\tilde H\Bn_F,\bm{\eta}\rangle_{\Gamma_F(t)}
&\label{eq:meanH_grav_1}
= -\langle
gz\boldsymbol{n}_F,\boldsymbol{\eta}
\rangle_{\Gamma_F(t)}
-\frac1{W\!e}\langle\nabla_\Gamma\mathrm{id},\nabla_\Gamma\bm{\eta}\rangle_{\Gamma_F(t)}\\
& + \frac1{W\!e}
\langle
\cos\theta_Y\,\bm{\mu}_s
+\sin\theta_c\,\Bn_s
-W\!e\,\beta_c(\Bu\cdot\bm{\mu}_s)\bm{\mu}_s,
\bm{\eta}
\rangle_{\Gamma_c(t)}. \label{eq:meanH_grav_2}
\end{align}
\end{subequations}
The next idea is to seek an appropriate approximation of $\langle
gz\boldsymbol{n}_F,\boldsymbol{\eta}
\rangle_{\Gamma_F(t)}$ in \eqref{eq:meanH_grav_1}, namely, approximating $z\boldsymbol{n}_F$ by some $\mathcal{G}$ such that when testing $\boldsymbol{\eta}$ as the ALE velocity, we recover the difference of potential energy at the discrete level:
\begin{align}
\tau  \langle
g \mathcal{G}, \widetilde \Bw^{m+1}_h
\rangle_{\Gamma_{F,h}^m} = E^{m+1}_{\textrm{pot}}
- E^{m}_{\textrm{pot}}.
\end{align}
To achieve this, we are inspired by \eqref{eq:exact_E_pot} to take 
\begin{align}\label{2dMot}
    \mathcal{G}= \int_{0}^{1} z_h^m(\vartheta) 
\hat{\Bn}_{F,h}(\vartheta) d\vartheta.
\end{align}
Note that $z_h^m(\vartheta)$ and $\hat{\Bn}_{F,h}(\vartheta)$ are both linear functions of $\vartheta$ in two dimensions.
Therefore, we can evaluate the exact time integral analytically by applying Simpson's quadrature rule:
\begin{align}\label{def:generalG}
\mathcal{G}= \frac{1}{6}z_h^m \hBn_{F,h}^m + \frac{2}{3}z_h^{m+1/2} \hBn_{F,h}^{m+1/2} + \frac{1}{6}z_h^{m+1} \hBn_{F,h}^{m+1},
\end{align}
with notations 
\begin{align}
&z_h^{m+1/2} = z_h^m + \frac{\tau}{2} \widetilde{\Bw}_h^{m+1}\cdot\Be_z, \label{hZ_1/2} \\
&\hBn_{F,h}^m = \hBn_{F,h}(0) = \Bn^m_{F,h},  \label{hBn_0} \\
&\hBn_{F,h}^{m+1/2} = \hBn_{F,h}(\frac{1}{2}) = \Bn^m_{F,h} + \frac{\tau}{2} (\partial_s \widetilde{\Bw}_h^{m+1})^\perp,  \label{hBn_1/2} \\
&\hBn_{F,h}^{m+1} = \hBn_{F,h}(1) = \Bn^m_{F,h} + 
\tau (\partial_s \widetilde{\Bw}_h^{m+1})^\perp.
\label{hBn_1}
\end{align}
According to \eqref{hZ_1/2}--\eqref{hBn_1}, $\mathcal{G}$ is a nonlinear function of $\tilde{\boldsymbol{w}}_h^{m+1}$.
Note that \eqref{hBn_0}--\eqref{hBn_1} only hold for the two-dimensional case. 
After substituting them into \eqref{def:generalG},
we denote the resulting $\mathcal{G}$ by $\mathcal{G}^{\textrm{2d}}$.
Therefore, we obtain the following structure-preserving BGN-type scheme: Given \(\Bu_h^m\), \(\Bw_h^m\) and domain \(\Omega_h^m\), we seek: 
\[(\Bu_h^{m+1}, \Tp_h^{m+1}, \kappa_{s,h}^{m+1}, \TH_h^{m+1}, \widetilde \Bw_h^{m+1}, \varphi_{c,h}^{m+1})\] 
in 
\(\mathcal{U}_h^m \times \mathcal{Q}_h^m\times S_{h,2}(\Gamma_{s,h}^m) \times S_h(\Gamma_{F,h}^m) \times  S_h(\Gamma_{F,h}^m)^d \times S_h(\Gamma_{c,h}^m)\) satisfying
\begin{subequations}\label{BGN-2d}
\begin{align}
&\frac{1}{2 \tau}( \Bu_h^{m+1}, \Bv)_{\Omega_h^{m}} 
+ \frac{1}{2 \tau}( \Bu_h^{m+1}\circ \psi_h^m , \Bv \circ \psi_h^m )_{\Omega_h^{m-1}} 
+ b\left( \Bu_h^m - \Bw_h^m; \Bu_h^{m+1}\circ \psi_h^m, \Bv \circ \psi_h^m\right)_{\Omega_h^{m-1}}\notag \\
&\qquad + (\frac{2}{Re} D( \Bu_h^{m+1}), D(\Bv))_{\Omega_h^{m}}   
- (\Tp_h^{m+1}, \nabla \cdot \Bv)_{\Omega_h^{m}}  
+ (q, \nabla \cdot \Bu_h^{m+1})_{\Omega_h^{m}} + \langle \TH_h^{m+1} \Bn_{F,h}^m, \Bv \rangle_{\Gamma^{m}_{F, h}} \notag \\
&\qquad 
- \langle \kappa_{s, h}^{m+1} \Bn_s, \Bv \rangle_{\Gamma^{m}_{s,h}}^{ s, h } 
+ \langle \lambda_s \Bn_s, \Bu_h^{m+1} \rangle_{\Gamma^{m}_{s, h}}^{ s, h } 
+ \beta_s \langle \mathbb P_s\Bu_h^{m+1},\mathbb P_s\Bv\rangle_{\Gamma_{s,h}^m} 
+ \langle \widetilde \Bw_h^{m+1} \cdot \Bn_s, \phi_c \rangle^{ h }_{\Gamma_{c,h}^m} \notag \\
&\qquad 
= \frac{1}{\tau}(\Bu_h^{m}, \Bv\circ \psi_h^m )_{\Omega_h^{m-1}}, \label{BGN-fluid} \\
&-\langle \TH_h^{m+1} \hBn_{F,h}^{m+1/2}, \bm{\eta} \rangle_{\Gamma_{F,h}^m}^{ h }
= -\frac{1}{W\!e}\langle \nabla_{\Gamma} (\bbI + \tau \widetilde \Bw_h^{m+1}), \nabla_{\Gamma} \bm{\eta} \rangle_{\Gamma_{F,h}^m} + \frac{1}{W\!e}\langle \varphi_{c,h}^{m+1} \Bn_s, \bm{\eta} \rangle^{ h }_{\Gamma_{c,h}^m} \notag \\
& \qquad
+ \frac{1}{W\!e} \int_{\Gamma_{c,h}^m}
\left[
\cos \theta_Y \bm{\mu}_s
- W\!e\,\beta_c (\widetilde{\Bw}_h^{m+1}\cdot\bm{\mu}_s) \bm{\mu}_s
+ \sin\theta_{c,h}^{m}\,\Bn_s
\right]\cdot \bm{\eta}\,\Rd s  \notag \\
& \qquad
- g \langle \mathcal{G}^{\rm 2d}(\widetilde{\Bw}_h^{m+1}) , \bm{\eta} \rangle_{\Gamma_{F,h}^m}, \label{BGN-2d-a} \\
& \langle \widetilde \Bw_h^{m+1} \cdot \hBn_{F,h}^{m+1/2}, \phi_F \rangle^{ h }_{\Gamma_{F,h}^m}
= \langle \Bu_h^{m+1} \cdot \Bn_{F,h}^m, \phi_F \rangle_{\Gamma_{F,h}^m}, \label{BGN-2d-b} 
\end{align}
\end{subequations}
for all \((\Bv, q, \lambda_s, \phi_F, \bm{\eta}, \phi_c) \in \mathcal{U}_h^m \times \mathcal{Q}_h^m\times S_{h,2}(\Gamma_{s,h}^m) \times S_h(\Gamma_{F,h}^m) \times S_h(\Gamma_{F,h}^m)^d \times S_h(\Gamma_{c,h}^m)\). 
In the above scheme, inspired by the structure-preserving parametric FEM \cite{Bao2021, Jiang2021},
we introduce \(\hBn_{F,h}^{m+1/2} \) defined in \eqref{hBn_1/2} as the equivalent intermediate normal vector.

\subsection{Three-dimensional gravitationally consistent scheme}
\label{sec:BGN-3d}
The reformulation \eqref{eq:weak_momentum_grav}--\eqref{eq:meanH_grav}
remains valid in three dimensions. 
The main difference from the two-dimensional case lies in the construction of the geometric
quantity $\mathcal G$.
For the three-dimensional case, the free boundary \(\Gamma_{F,h}^m\) is a surface. 
Let \(s_1\) and \(s_2\) denote two local surface parameters on \(\Gamma_{F,h}^m\). 
The pulled-back tangential vectors on the intermediate surface are given by
\[
\hat{\Bt}_{F,1}(\vartheta)
:=
\partial_{s_1}\bm{X}_{F,h}^{m}(\vartheta)
=
\Bt_{F,1}^m+\vartheta\tau\,\partial_{s_1}\widetilde{\Bw}_h^{m+1},
\quad
\hat{\Bt}_{F,2}(\vartheta)
:=
\partial_{s_2}\bm{X}_{F,h}^{m}(\vartheta)
=
\Bt_{F,2}^m+\vartheta\tau\,\partial_{s_2}\widetilde{\Bw}_h^{m+1}.
\]
Taking an appropriate ordering of the tangential vectors, we define
the pulled-back unnormalized outward surface vector by
\[
\hat{\Bn}_{F,h}(\vartheta)
:=
\hat{\Bt}_{F,1}(\vartheta)\times \hat{\Bt}_{F,2}(\vartheta).
\]

The gravitational potential-energy difference can then be written as
\begin{equation}
\label{eq:energy_change_3d}
E_{\rm pot}^{m+1}-E_{\rm pot}^m
=
\tau g
\int_{\Gamma_{F,h}^m}
\widetilde{\Bw}_h^{m+1}\cdot
\left(
\int_0^1
z_h^m(\vartheta)
\widehat{\Bn}_{F,h}(\vartheta)
\,\Rd\vartheta
\right)
\,\Rd s .
\end{equation}
Motivated by the two-dimensional construction in \eqref{2dMot}, we define
\begin{equation}
\label{eq:G3d_definition}
\mathcal G^{\rm 3d}
\bigl(\widetilde{\Bw}_h^{m+1}\bigr)
:=
\int_0^1
z_h^m(\vartheta)
\widehat{\Bn}_{F,h}(\vartheta)
\,\Rd\vartheta.
\end{equation}
Consequently,
\begin{equation}
\label{eq:G3d_energy_identity}
\tau g
\left\langle
\mathcal G^{\rm 3d}
\bigl(\widetilde{\Bw}_h^{m+1}\bigr),
\widetilde{\Bw}_h^{m+1}
\right\rangle_{\Gamma_{F,h}^m}
=
E_{\rm pot}^{m+1}-E_{\rm pot}^m.
\end{equation}

Since $\widehat{\Bn}_{F,h}(\vartheta)$ is quadratic in $\vartheta$, while $z_h^m(\vartheta)$ is linear in $\vartheta$. 
Hence their product is a polynomial of degree at most three, and Simpson's rule is exact:
\begin{align}
\mathcal G^{\rm 3d}
\bigl(\widetilde{\Bw}_h^{m+1}\bigr)
&=
\frac16 z_h^m\widehat{\Bn}_{F,h}^m
+
\frac23 z_h^{m+1/2}
\widehat{\Bn}_{F,h}^{m+1/2}
+
\frac16 z_h^{m+1}
\widehat{\Bn}_{F,h}^{m+1},
\label{eq:G3d_Simpson}
\end{align}
where
\begin{align}
&\hat{\Bn}_{F,h}^m 
= \hat{\Bn}_{F,h}(0)  = \Bn_{F,h}^m, \\
&\hat{\Bn}_{F,h}^{m+1/2} 
= \hat{\Bn}_{F,h}(\frac{1}{2}) 
= \Bn_{F,h}^m
+\frac{\tau}{2}
\bigl(
\partial_{s_1}\widetilde{\Bw}_h^{m+1}\times \Bt_{F,2}^m
+\Bt_{F,1}^m\times \partial_{s_2}\widetilde{\Bw}_h^{m+1}
\bigr) \notag \\
&\quad\qquad\qquad\qquad\qquad\qquad
+ \frac{\tau^2}{4}
\bigl(
\partial_{s_1}\widetilde{\Bw}_h^{m+1}\times \partial_{s_2}\widetilde{\Bw}_h^{m+1}
\bigr), \label{eq:n_m+1/2_3D} \\
&\hat{\Bn}_{F,h}^{m+1} 
= \hat{\Bn}_{F,h}(1)
= \Bn_{F,h}^m
+\tau
\bigl(
\partial_{s_1}\widetilde{\Bw}_h^{m+1}\times \Bt_{F,2}^m
+\Bt_{F,1}^m\times \partial_{s_2}\widetilde{\Bw}_h^{m+1}
\bigr) \notag \\
&\quad\qquad\qquad\qquad\qquad\qquad 
+\tau^2\bigl(
\partial_{s_1}\widetilde{\Bw}_h^{m+1}\times \partial_{s_2}\widetilde{\Bw}_h^{m+1}
\bigr). \label{eq:n_m+1_3D}
\end{align}

For the volume-conserving kinematic condition, we also introduce the
averaged pulled-back surface vector
\begin{equation}
\label{eq:averaged_normal_3d}
\widetilde{\Bn}_{F,h}^{m+1/2}
:=
\int_0^1
\widehat{\Bn}_{F,h}(\vartheta)\,\Rd\vartheta
=
\frac16\widehat{\Bn}_{F,h}^m
+
\frac23\widehat{\Bn}_{F,h}^{m+1/2}
+
\frac16\widehat{\Bn}_{F,h}^{m+1}.
\end{equation}
Since $\widehat{\Bn}_{F,h}(\vartheta)$ is quadratic in
$\vartheta$, the last equality is again exact.
Since $\widetilde\Bw_h^{m+1}$ is piecewise linear, 
$\widetilde{\Bn}_{F,h}^{m+1/2}$ and the related vectors 
$\widehat{\Bn}_{F,h}^m$, $\widehat{\Bn}_{F,h}^{m+1/2}$, $\widehat{\Bn}_{F,h}^{m+1}$ 
are elementwise constants.
Furthermore, following \cite{Hu2026_Droplet}, we employ the averaged
solid conormal introduced in \cite[(3.7)]{Bao2023}:
\begin{equation}
\label{conormal-3d}
\bm{\mu}_{s,h}^{m+1/2}
=
\frac12
\left[
\partial_s
\left(
\bbI\big|_{\Gamma_{c,h}^m}
\right)
+
\partial_s
\left(
\bbI\big|_{\Gamma_{c,h}^m}
+
\tau\widetilde{\Bw}_h^{m+1}
\big|_{\Gamma_{c,h}^m}
\right)
\right]
\times\Bn_s,
\end{equation}
where $s$ denotes the arc-length parameter along
$\Gamma_{c,h}^m$, and the orientation is chosen consistently with
the outward conormal direction.

In addition, we impose a tangential constraint on the mesh velocity along \(\Gamma_{c,h}^m\). 
More precisely, let \(\bm{\xi}_{c,h}^m\) denote the tangential direction on the discrete contact line. We require
\[
\widetilde{\Bw}_h^{m+1}\cdot \bm{\xi}_{c,h}^m = 0
\qquad \text{on } \Gamma_{c,h}^m,
\]
which is enforced weakly by introducing a Lagrange multiplier \(\psi_{c,h}^{m+1}\). 
This gives the additional weak terms
\[
\langle \widetilde{\Bw}_h^{m+1},\chi_c\bm{\xi}_{c,h}^m\rangle_{\Gamma_{c,h}^m}^{ h }
\quad\text{and}\quad
\langle \psi_{c,h}^{m+1}\bm{\xi}_{c,h}^m,\bm{\eta}\rangle_{\Gamma_{c,h}^m}^{ h }
\]
in the discrete formulation. This constraint prevents tangential mesh rotation at the contact line and helps preserve the mesh quality near the contact-line region.
With these definitions, the three-dimensional gravity-consistent ALE--FEM scheme is given as follows.

Given the fluid velocity \(\Bu_h^{m}\), mesh velocity \(\Bw_h^{m}\), flow map \(\psi_h^m\) defined on \(\Omega_h^{m-1}\), and discretized domain \(\Omega_h^{m}\) with boundary \(\Gamma_{F,h}^m\) at time \(t_m\), we seek: \[
(\Bu_h^{m+1}, \Tp_h^{m+1}, \kappa_{s,h}^{m+1},
\tilde{H}_h^{m+1}, \widetilde \Bw_h^{m+1}, \varphi_{c,h}^{m+1}, \psi_{c,h}^{m+1}) 
\] 
in 
\(
\mathcal{U}_h^m \times \mathcal{Q}_h^m\times S_{h,2}(\Gamma_{s,h}^m)
\times S_h(\Gamma_{F,h}^m) \times S_h(\Gamma_{F,h}^m)^d \times S_h(\Gamma_{c,h}^m) \times S_h(\Gamma_{c,h}^m)
\) 
satisfying
\begin{subequations}\label{BGN-3d}
\begin{align}
&\frac{1}{2 \tau}( \Bu_h^{m+1}, \Bv)_{\Omega_h^{m}} 
+ \frac{1}{2 \tau}( \Bu_h^{m+1}\circ \psi_h^m , \Bv \circ \psi_h^m )_{\Omega_h^{m-1}} 
+ b\!\left( \Bu_h^m - \Bw_h^m; \Bu_h^{m+1}\circ \psi_h^m, \Bv \circ \psi_h^m\right)_{\Omega_h^{m-1}}\notag \\
&\quad + (\frac{2}{Re} D( \Bu_h^{m+1}), D(\Bv))_{\Omega_h^{m}} - (\Tp_h^{m+1}, \nabla \cdot \Bv)_{\Omega_h^{m}} + (q, \nabla \cdot \Bu_h^{m+1})_{\Omega_h^{m}} 
+ \langle \tilde{H}_h^{m+1} \Bn_{F,h}^m, \Bv \rangle_{\Gamma^{m}_{F, h}} \notag \\ 
&\quad 
- \langle \kappa_{s, h}^{m+1} \Bn_{s}, \Bv \rangle_{\Gamma^{m}_{s,h}}^{ s, h } 
+ \langle \lambda_s \Bn_{s}, \Bu_h^{m+1} \rangle_{\Gamma^{m}_{s, h}}^{ s, h } 
+ \beta_s \langle \mathbb P_s\Bu_h^{m+1},\mathbb P_s\Bv\rangle_{\Gamma_{s,h}^m} \notag \\
&\quad 
+ \frac{1}{W\!e}\langle \widetilde \Bw_h^{m+1}, \phi_c \Bn_{s} \rangle^{ h }_{\Gamma_{c,h}^m} 
+ \frac{1}{W\!e}\langle \widetilde \Bw_h^{m+1}, \chi_c \bm{\xi}_{c, h}^m \rangle^{ h }_{\Gamma_{c,h}^m}
= \frac{1}{\tau}(\Bu_h^{m}, \Bv\circ \psi_h^m )_{\Omega_h^{m-1}}, \label{BGN-3d-fluid} \\
& - \langle \tilde{H}_h^{m+1} \tilde{\Bn}_{F,h}^{m+1/2}, \bm{\eta} \rangle_{\Gamma_{F,h}^m}^{ h }
= - \frac{1}{W\!e} \langle \nabla_{\Gamma} (\bbI + \tau \widetilde \Bw_h^{m+1}), \nabla_{\Gamma} \bm{\eta} \rangle_{\Gamma_{F,h}^m} 
+ \frac{1}{W\!e} \langle \varphi_{c,h}^{m+1} \Bn_{s}, \bm{\eta} \rangle^{ h }_{\Gamma_{c,h}^m}\notag \\
& \quad
+ \frac{1}{W\!e} \int_{\Gamma_{c,h}^m}
\left[
\cos \theta_Y \bm{\mu}_{s,h}^{m+1/2}
- W\!e\,\beta_c (\widetilde{\Bw}_h^{m+1}\cdot\bm{\mu}_{s,h}^m) \bm{\mu}_{s,h}^m
+ \sin\theta_{c,h}^{m}\,\Bn_s
\right]\cdot \bm{\eta}\,\Rd s   \notag \\
& \quad
+ \frac{1}{W\!e}\langle \psi_{c,h}^{m+1} \bm{\xi}_{c, h}^m, \bm{\eta} \rangle^{ h }_{\Gamma_{c,h}^m} - g \langle \mathcal{G}^{\rm 3d}(\widetilde{\Bw}_h^{m+1}) , \bm{\eta} \rangle_{\Gamma_{F,h}^m}
\label{BGN-3d-a} \\
& \langle \widetilde \Bw_h^{m+1} \cdot \tilde{\Bn}_{F,h}^{m+1/2}, \phi_F \rangle^{ h }_{\Gamma_{F,h}^m}
= \langle \Bu_h^{m+1} \cdot \Bn_{F,h}^m, \phi_F \rangle_{\Gamma_{F,h}^m}, \label{BGN-3d-b} 
\end{align}
\end{subequations}
for all \((\Bv, q, \lambda_s, \phi_F, \bm{\eta}, \phi_c,\chi_c) \in \mathcal{U}_h^m \times \mathcal{Q}_h^m\times 
S_{h,2}(\Gamma_{s,h}^m) \times S_h(\Gamma_{F,h}^m) \times S_h(\Gamma_{F,h}^m)^d \times S_h(\Gamma_{c,h}^m)\times S_h(\Gamma_{c,h}^m)\).

\subsection{Mesh velocity in the bulk}
\label{sec:MDR}
Once the mesh velocity \(\widetilde{\Bw}\) on the free boundary \(\Gamma_{F}\) is determined, we extend it into the bulk by the minimal deformation rate (MDR) approach \cite{Hu2022,Hu2026_Droplet}. 
Although harmonic extension is simple and widely used, it may lead to mesh concentration and severe distortion near the moving contact line. 
In contrast, the MDR method also controls the mesh deformation along the solid boundary \(\Gamma_s(t)\), and therefore gives better mesh regularity both in the bulk and near the contact-line region.
More precisely, the bulk mesh velocity is defined as the minimizer of the functional
\begin{equation}\label{defE}
\mathcal{E}(\Bw)
:=
\frac{1}{2} \int_{\Omega(t)} |D(\Bw)|^2 \,\Rd x
+ \frac{1}{2} \int_{\Gamma_s(t)} |\nabla_{\Gamma} \Bw|^2 \,\Rd s,
\end{equation}
subject to the constraints
\[
\Bw=\widetilde{\Bw}\quad \text{on } \Gamma_F(t),
\qquad
\Bw\cdot \Bn_s=0\quad \text{on } \Gamma_s(t).
\]

Following \cite{Hu2026_Droplet}, at the discrete level we use the piecewise linear space
\[
\mathring{\Cw}_h^m
:=
\left\{
\bm{\zeta}\in \Cw_h^m:
\bm{\zeta}=\mathbf{0}\ \text{on the vertices of }\Gamma_{F,h}^m
\right\},
\]
which incorporates the homogeneous Dirichlet condition on the free boundary. 
Then, given the free-boundary mesh velocity \(\widetilde{\Bw}_h^{m+1}\), we compute the bulk mesh velocity \(\Bw_h^{m+1}\) by seeking
\[
(\Bw_h^{m+1},\varphi_{s,h}^{m+1})
\in
(\widetilde{\Bw}_h^{m+1}+\mathring{\Cw}_h^m)\times S_h(\Gamma_{s,h}^m)
\]
such that
\begin{subequations}\label{MDR}
\begin{align}
    (D (\Bw_h^{m+1}), D (\bm{\zeta}) )_{\Omega_h^m} 
    + \langle \nabla_\Gamma \Bw_h^{m+1}, \nabla_\Gamma \bm{\zeta} \rangle_{\Gamma_{s,h}^m}
    &= \langle \varphi_{s,h}^{m+1} \Bn_s, \bm{\zeta} \rangle_{\Gamma_{s,h}^m}^{ h },
    \label{MDR-a}\\
    \langle \Bw_h^{m+1} \cdot \Bn_s, \phi_s \rangle_{\Gamma_{s,h}^m}^{ h }
    &= 0,
    \label{MDR-c}
\end{align}
\end{subequations}
for all
\(
(\bm{\zeta},\phi_s)\in \mathring{\Cw}_h^m\times S_h(\Gamma_{s,h}^m).
\)

\section{Structure-preserving properties}
\label{sec:theorem}
In this section, we establish the main properties of the proposed ALE--FEM schemes. 
We first prove that the schemes preserve the discrete volume exactly and satisfy a discrete total energy-dissipation law, where the gravitational potential energy is included in the total energy. 
Then, based on this energy estimate and a Korn-type inequality, we further show that the discrete velocity vanishes in the long-time regime. 
This result indicates that the proposed method can eliminate gravity-induced spurious velocities.

\begin{theorem}\label{thm:structure_preserving}
Let $d\in\{2,3\}$, and let $(\Bu_h^{m+1}, \Tp_h^{m+1}, \kappa_{s,h}^{m+1}, \TH_h^{m+1}, \widetilde \Bw_h^{m+1}, \varphi_{c,h}^{m+1})$ be the solution of the 2D scheme \eqref{BGN-2d}
and let \( (\Bu_h^{m+1}, \Tp_h^{m+1}, \kappa_{s,h}^{m+1}, \TH_h^{m+1}, \widetilde \Bw_h^{m+1}, \varphi_{c,h}^{m+1}, \psi_{c,h}^{m+1}) \) be the solution of the 3D scheme \eqref{BGN-3d}. 
Then, the proposed schemes are strictly structure-preserving in the following sense:
\begin{itemize}
\item[\textbf{(i)}] \textbf{Volume Conservation:} The total volume of the fluid domain is rigorously conserved at the discrete level, i.e., 
\( |\Omega_h^{m+1}| = |\Omega_h^m| \).
\item[\textbf{(ii)}] \textbf{Total Energy Dissipation:} 
The discrete total energy associated with the numerical solution is defined as
\begin{align}
\label{eq:def_total_energy}
E_h^{m} := \frac{1}{2} \|\Bu^{m}_h \|^2_{L^2(\Omega_h^{m-1})} 
+ \frac{1}{W\!e} |\Gamma_{F, h}^{m}| 
- \frac{\cos\theta_Y}{W\!e} |\Gamma_{s, h}^{m}|
+ g\int_{\Omega_h^{m}} z \Rd x.
\end{align}
Then the discrete total energy decreases monotonically over time, satisfying:
\begin{align}\label{eq:energy_dissipation}
E_h^{m+1} - E_h^{m} 
&\leq -\frac{1}{2} \| \Bu_h^{m+1} \circ \psi_h^m - \Bu_h^m \|^2_{\Omega_h^{m-1}} 
- \frac{2\tau}{Re} \| D(\Bu_h^{m+1}) \|^2_{\Omega_h^m} \notag \\
&\quad - \tau \beta_s \| \bbP_s \Bu_h^{m+1} \|^2_{\Gamma_{s,h}^m}
-\tau \beta_c \|\widetilde \Bw_h^{m+1} \cdot \bm \mu_{s,h}^m \|^2_{\Gamma_{c,h}^m} \leq 0, 
\quad m \ge 0.
\end{align}
\end{itemize}
Here $|\Omega_h^m|$ denotes the $d$-dimensional Lebesgue measure, whereas $|\Gamma_{F,h}^m|$ and $|\Gamma_{s,h}^m|$ denote the $(d-1)$-dimensional Hausdorff measures. 
\end{theorem}

For clarity, we prove for the three-dimensional scheme \eqref{BGN-3d}; the proof for the two-dimensional case is analogous and therefore omitted.
\begin{proof}
\textbf{Part I: Proof of Volume Conservation.} 
Since $\Bn_s$ is constant on the flat solid boundary and
$\Bu_h^{m+1}$ is piecewise quadratic, 
$\Bu_h^{m+1}\cdot\Bn_s\in S_{h,2}(\Gamma_{s,h}^m)$.
Taking $\lambda_s=\Bu_h^{m+1}\cdot\Bn_s$
in the discrete impermeability constraint in \eqref{BGN-3d-fluid} gives
\[
\left\langle
(\Bu_h^{m+1}\cdot\Bn_s)\Bn_s,\Bu_h^{m+1}
\right\rangle_{\Gamma_{s,h}^m}^{s,h}=0.
\]
Since all nodal weights in \eqref{eq:mass_lumping_nodal} are positive, this implies that
$\Bu_h^{m+1}\cdot\Bn_s$ vanishes at all quadratic Lagrange nodes, and hence vanishes identically on $\Gamma_{s,h}^m$.
Therefore
\begin{equation}\label{eq:u_ns_0}
\int_{\Gamma_{s,h}^m}
\Bu_h^{m+1}\cdot\Bn_s \,\Rd s
=0= \langle \Bu_h^{m+1}\cdot \Bn_s, 1 \rangle_{\Gamma_{s,h}^m}^{s,h}.
\end{equation}

By choosing \(\Bv = \mathbf{0}\), \(q = \phi_F = 1\), \(\phi_c = \chi_c = 0\) and \( \lambda_s = -1 \) in the weak formulation \eqref{BGN-3d}, we obtain, using the divergence theorem and \eqref{eq:u_ns_0},
\begin{align*}
0 
&= \int_{\Omega_h^m} \nabla \cdot \Bu_h^{m+1} \Rd x
+ \langle \tilde{\Bn}_{F,h}^{m+1/2} \cdot \widetilde{\Bw}_h^{m+1}, 1 \rangle_{\Gamma_{F,h}^m}^{h}
- \langle \Bn_{F,h}^m \cdot \Bu_h^{m+1}, 1 \rangle_{\Gamma_{F,h}^m}
- \langle \Bn_s \cdot \Bu_h^{m+1},1 \rangle_{\Gamma_{s,h}^m}^{s,h} \\
&= \langle \tilde{\Bn}_{F,h}^{m+1/2} \cdot \widetilde{\Bw}_h^{m+1}, 1 \rangle_{\Gamma_{F,h}^m}^{h}. \notag
\end{align*} 
By \eqref{eq:averaged_normal_3d},
$\widetilde{\Bn}_{F,h}^{m+1/2}$ is elementwise constant on
$\Gamma_{F,h}^m$, whereas
$\widetilde{\Bw}_h^{m+1}$ is piecewise linear. Hence
$\widetilde{\Bw}_h^{m+1}\cdot
\widetilde{\Bn}_{F,h}^{m+1/2}$ is piecewise linear.
Since the mass-lumped quadrature is exact for piecewise linear
functions, we have
\[
\int_{\Gamma_{F,h}^m}
\widetilde{\Bw}_h^{m+1}\cdot
\widetilde{\Bn}_{F,h}^{m+1/2}\,\Rd s
=
\int_{\Gamma_{F,h}^m}^h
\widetilde{\Bw}_h^{m+1}\cdot
\widetilde{\Bn}_{F,h}^{m+1/2}\,\Rd s
=0.
\]
According to \eqref{eq:exact_E_pot}, the volume change from $t_{m}$ to $t_{m+1}$ can be written as
\begin{align*}
|\Omega_h^{m+1}| - |\Omega_h^m| 
= \int_{0}^{1} \frac{\Rd}{\Rd \vartheta } \int_{\Omega_h^{m}(\vartheta)} 1 \Rd x \Rd \vartheta 
= \tau \int_{\Gamma_{F,h}^m} \widetilde{\Bw}_h^{m+1} \cdot \left( \int_0^1 \hBn_{F,h}(\vartheta) \Rd \vartheta \right) \Rd s . 
\end{align*}
Since $\hBn_{F,h}(\vartheta) = \hat{\Bt}_{F,1}(\vartheta)\times \hat{\Bt}_{F,2}(\vartheta) $ is a quadratic function of $\vartheta$, the integration can be exactly represented by the Simpson formula \eqref{eq:averaged_normal_3d} as $\int_0^1 \hBn_{F,h}(\vartheta) \, \Rd \vartheta 
= \tilde \Bn_{F,h}^{m+1/2}$.
Thus, the scheme is strictly volume preserving,
\begin{align}\label{eq:volume_relation}
|\Omega_h^{m+1}| - |\Omega_h^m|
= \tau \int_{\Gamma_{F,h}^m} \widetilde{\Bw}_h^{m+1} \cdot \tilde \Bn_{F,h}^{m+1/2} \Rd s
= 0.
\end{align}

\textbf{Part II: Proof of Energy Dissipation.} 
We test the fully discrete scheme \eqref{BGN-3d} with
\(\Bv = \Bu_h^{m+1}\), \(q = \Tp_h^{m+1}\), 
\(\lambda_s = \kappa_{s,h}^{m+1}\),  
\(\phi_F = \TH_h^{m+1} \), \(\bm{\eta} = \widetilde{\Bw}_h^{m+1}\), \(\phi_c = \varphi^{m+1}_{c,h}\) and \(\chi_c = \psi_{c,h}^{m+1} \).
Adding the resulting equations together, the pressure-incompressibility terms, the solid-boundary multiplier terms, and the contact-line constraint terms cancel. We obtain
\begin{align}\label{energy-inter}
&\frac1{2} \big[ \| \Bu_h^{m+1}\|_{ \Omega_h^{m}}^2
- \| \Bu_h^{m}\|_{ \Omega_h^{m-1}}^2 
+ \| \Bu_h^{m+1}\circ \psi_h^m - \Bu_h^{m}\|_{\Omega_h^{m-1}}^2 \big] 
+ \frac{2\tau}{Re} \| D(\Bu_h^{m+1}) \|^2_{\Omega_h^{m}} \notag\\
&\quad 
+ \tau \beta_s \| \bbP_s \Bu_h^{m+1}\|^2_{\Gamma_{s,h}^m} 
+ \frac{1}{W\!e} \langle \nabla_{\Gamma} (\bbI + \tau \widetilde \Bw_h^{m+1}), 
\nabla_{\Gamma} (\tau \widetilde \Bw_h^{m+1}) \rangle_{\Gamma_{F,h}^m}
\notag \\
&\quad 
- \frac{1}{W\!e} \int_{\Gamma_{c,h}^m} \left( \cos \theta_Y \bm \mu_{s,h}^{m+1/2} + \sin \theta_{c,h}^m\, \Bn_s \right) \cdot \tau \widetilde \Bw_h^{m+1} \Rd s 
+ \tau \beta_c \| \widetilde \Bw_h^{m+1} \cdot \bm \mu^m_{s,h} \|^2_{\Gamma^m_{c,h}} \notag \\
&\quad + \tau g\langle \mathcal{G}^{\rm 3d}(\widetilde{\Bw}_h^{m+1}) , \widetilde{\Bw}_h^{m+1} \rangle_{\Gamma_{F,h}^m}
= 0
\end{align}
where we used the identity
\begin{align*}
&\frac{1}{2 \tau} \| \Bu_h^{m+1}\|_{ \Omega_h^{m}}^2 
+ \frac{1}{2 \tau} \| \Bu_h^{m+1}\circ \psi_h^m - \Bu_h^{m} \|_{\Omega_h^{m-1}}^2
- \frac{1}{\tau}( \Bu_h^{m}, \Bu_h^{m+1}\circ \psi_h^m)_{\Omega_h^{m-1}} \\
&= \frac1{2 \tau} \left[ \| \Bu_h^{m+1}\|_{ \Omega_h^{m}}^2
- \| \Bu_h^{m}\|_{ \Omega_h^{m-1}}^2 
+ \| \Bu_h^{m+1}\circ \psi_h^m - \Bu_h^{m} \|_{\Omega_h^{m-1}}^2 \right].
\end{align*}

It remains to identify the geometric terms appearing in the energy estimate. 
Recall that \(\theta_{c,h}^m\) defined in \eqref{eq:def_theta_ch} is piecewise constant, \(\widetilde \Bw_h^{m+1}\) is piecewise linear and \(\Bn_s\) is constant on \(\Gamma_{c,h}^m\), 
by testing $\phi_c = 1$ 
in \eqref{BGN-3d-fluid} gives
\begin{align}
\int_{\Gamma_{c,h}^m}
\sin \theta_{c,h}^m\,\Bn_s\cdot \tau \widetilde \Bw_h^{m+1}\,\Rd s = 
\tau \langle \sin \theta_{c,h}^m  \widetilde \Bw_h^{m+1}, \Bn_{s} \rangle^{ h }_{\Gamma_{c,h}^m} 
=0.
\end{align}

The remaining surface and contact-line terms can be related to the variations of the free-surface area and the wetted solid area. 
For the free surface, we use the area-decreasing inequality \cite[(2.21)]{Barrett2007_BGN2},
\begin{align}\label{area-decreasing-property}
\int_{\Gamma_{F,h}^m}
\nabla_{\Gamma}(\mathrm{id}+\tau \widetilde \Bw_h^{m+1})
\cdot
\nabla_{\Gamma}\tau \widetilde \Bw_h^{m+1}
\,\Rd s
\geq
|\Gamma_{F,h}^{m+1}|-|\Gamma_{F,h}^m|.
\end{align}
For the wetted solid surface, the averaged conormal vector \(\mubf_{s,h}^{m+1/2}\) defined in \eqref{conormal-3d} yields the exact identity
\[
\int_{\Gamma_{c,h}^m}\cos \theta_Y\mubf_{s,h}^{m+1/2}
\cdot
\tau\widetilde{\Bw}_h^{m+1}
\,\Rd s
= \cos \theta_Y \left(
|\Gamma_{s,h}^{m+1}|-|\Gamma_{s,h}^{m}| \right).
\]
Thus, the geometric terms in the discrete weak formulation control precisely the changes of the free surface and the wetted solid surface.
Based on Simpson's rule and \eqref{eq:G3d_energy_identity}, 
the gravitational potential energy change becomes
\begin{align*}
\tau g \langle \mathcal{G}^{\rm 3d}(\widetilde \Bw_h^{m+1}) , \widetilde{\Bw}_h^{m+1} \rangle_{\Gamma_{F,h}^m} 
= E^{m+1}_{\rm pot} - E^{m}_{\rm pot}
= g\int_{\Omega_h^{m+1}} z \,\Rd x - g\int_{\Omega_h^m} z \,\Rd x .
\end{align*} 
Substituting these identities into \eqref{energy-inter}, we deduce the discrete energy inequality
\begin{align}
E_h^{m+1} - E_h^{m} 
&\leq -\frac{1}{2} \| \Bu_h^{m+1} \circ \psi_h^m - \Bu_h^m \|^2_{\Omega_h^{m-1}} 
- \frac{2\tau}{Re} \| D(\Bu_h^{m+1}) \|^2_{\Omega_h^m} \notag \\
&\quad - \tau \beta_s \| \bbP_s \Bu_h^{m+1} \|^2_{\Gamma_{s,h}^m}
-\tau \beta_c \|\widetilde \Bw_h^{m+1} \cdot \bm \mu_{s,h}^m \|^2_{\Gamma_{c,h}^m} \leq 0.
\end{align}
This completes the proof of energy dissipation in the three-dimensional case. 
\hfill\end{proof}

\begin{remark}\upshape
The above structure-preserving property only needs the mesh velocity on the free surface $\Gamma_{F,h}^m$, since the bulk mesh velocity does not influence it.
Therefore, by extending the mesh velocity $\widetilde{\Bw}_h^{m+1}$ along the free boundary, the MDR method does not destroy the discrete structure-preserving property proved in Theorem \ref{thm:structure_preserving}.
\end{remark}

Theorem~\ref{thm:structure_preserving} shows that the total energy always dissipates for the proposed scheme with gravity.
To study the long-time behavior of the numerical solution, we first establish a Korn-type inequality under the mixed boundary conditions. 
This auxiliary result is stated in the following lemma and will be used together with the energy estimate to show the decay of the discrete velocity.

\begin{lemma}[Korn's Inequality with Mixed Boundary]\label{lem:korn}
Let $\Omega \subset \mathbb{R}^d$ ($d=2,3$) be a bounded Lipschitz domain with its boundary partitioned as $\partial \Omega = \overline{\Gamma}_1 \cup \overline{\Gamma}_2$, where $\Gamma_1$ and $\Gamma_2$ are disjoint open subsets of $\partial \Omega$. 
Assume further that $|\Gamma_1| > 0$. 
Then, there exists a constant $C_1 > 0$,
such that
\[
\|\Bu\|^2_{L^2(\Omega)} 
\le C_1 \left( \|D(\Bu)\|^2_{L^2(\Omega)} 
+ \|\Bu\|^2_{L^2(\Gamma_1)} \right)  
\qquad \forall \Bu \in H^1(\Omega)^d ,
\]
where $D(\Bu) = \frac{1}{2}(\nabla \Bu + \nabla \Bu^\top)$ is the symmetric strain tensor.
\end{lemma}

\begin{proof}
Suppose that no such constant $C_1$ exists. 
Then there exists a sequence
\(\{\Bu_n\}_{n\ge 1}\) with
\(
\|\Bu_n\|_{L^2(\Omega)}^2=1,
\)
such that
\[ 
1 = \| \Bu_n \|_{L^2(\Omega)}^2 
\ge n \left( \|D(\Bu_n)\|_{L^2(\Omega)}^2 + \| \Bu_n\|_{L^2(\Gamma_1)}^2 \right). 
\]
We hence get the existence of a $C_2 > 0$ such that
\[ 
\| \Bu_n \|_{L^2(\Omega)}^2 
+ \|D(\Bu_n)\|_{L^2(\Omega)}^2
\le C_2.
\]
The coercivity of $D(\Bu)$ (see \cite[Lemma 62.20]{Zeidler1988IV}) gives the existence of a $C_3 > 0$ such that 
\[
\| \Bu_n \|_{H^{1}(\Omega)} \le C_3.
\]
As the embeddings $H^{1}(\Omega) \hookrightarrow L^2(\partial \Omega)$ and $H^{1}(\Omega) \hookrightarrow L^2(\Omega)$ are compact, we get for a subsequence $\{ \Bu_{n_k} \}$ there exists $\Bu \in H^{1}(\Omega)^d$ s.t. 
\begin{align}
\Bu_{n_k} \rightharpoonup \Bu \quad \text{in}\,\, H^{1}(\Omega),\quad 
\Bu_{n_k} \to \Bu \quad \text{in}\,\, L^2(\Omega), \quad
\Bu_{n_k} \to \Bu \quad \text{in}\,\, L^2(\partial\Omega). 
\end{align}

As $D(\Bu_n) \to 0$ in $L^2(\Omega)$ and $\Bu_n \to 0$ in $L^2(\Gamma_1)$, we obtain from
\cite[Lemma 62.20]{Zeidler1988IV}, 
\[
\Bu_{n_k} \to \Bu \quad \text{in}\,\, H^{1}(\Omega).
\]
By the continuity of the operator $D(\cdot)$ and the uniqueness of the limit, it follows that
\[
D(\Bu) = \mathbf{0}.
\]

Therefore, in the three-dimensional case,  $\nabla \Bu$ is strictly skew-symmetric, meaning $\partial_j u_i + \partial_i u_j = 0$. Differentiating this relation and permuting the indices yields $2\partial_k \partial_j u_i = 0$, which implies that all second-order derivatives of $\boldsymbol{u}$ vanish. Therefore,
\[ 
\Bu(\Bx) = \Ba + \Bb \times \Bx,\quad 
\text{and}\quad \int_{\Gamma_1} |\Bu|^2 \Rd s = 0, 
\]
where $\Ba, \Bb \in \bbR^3$ are constant vectors. 
This gives 
$\Ba + \Bb \times \Bx = \mathbf{0}$ 
on $\Gamma_1$.
Lemma 62.16 in \cite{Zeidler1988IV} gives $\Ba, \Bb = \mathbf{0}$ which implies that $\Bu = \mathbf{0}$.
This contradicts $\| \Bu \|_{L^2(\Omega)} = 1$, 
so there exists $C_1 > 0$, such that 
\[
\|\Bu\|^2_{L^2(\Omega)} \le C_1 \left( \|D(\Bu)\|^2_{L^2(\Omega)} 
+ \|\Bu\|^2_{L^2(\Gamma_1)} \right) 
\qquad \forall\, \Bu \in H^1(\Omega)^d .
\]

For the two-dimensional case, a rigid body motion can only vanish along a one-dimensional curve segment $\Gamma_1$ if it is identically zero, as any non-trivial rotation in 2D admits at most a single stationary point.
\hfill\end{proof}

\begin{theorem}\label{coro}
Assume that $\beta_s, \beta_c > 0$, that $E_h^m$ is bounded from below, and
that $|\Gamma_{s,h}^m|>0$ for all $m$. 
We further assume that the
family $\{\Omega_h^m\}_{m\geq0}$ admits a uniform Korn inequality
with a constant $C_\Omega>0$ independent of $m$.
Then, for a fixed time step $\tau>0$,
the discrete velocity of the ALE-FEM scheme \eqref{BGN-3d} asymptotically vanishes:
\begin{align}\label{eq:u_converge_0}
\lim_{m\to\infty} \|\Bu_h^{m+1}\|_{\Omega_h^m} = 0.
\end{align} 
\end{theorem}

\begin{proof}
Since the discrete energy is bounded from below, let 
$$
E_* := \inf_{m\ge 0} E_h^m > -\infty. 
$$ 
By summing the discrete energy dissipation inequality \eqref{eq:energy_dissipation} from $k=0$ to $m$, we obtain
\begin{align*}
\tau \sum_{k=0}^{m}  
&\left( \frac{1}{2\tau} \| \Bu_h^{k+1} \circ \psi_h^k - \Bu_h^k \|^2_{\Omega_h^{k-1}} + \frac{2}{Re} \|D(\Bu_h^{k+1})\|^2_{\Omega_h^k} \right. \notag \\
&\quad \left. + \beta_s \| \bbP_s \Bu_h^{k+1} \|^2_{\Gamma_{s,h}^k}
+ \beta_c \|\widetilde \Bw_h^{k+1} \cdot \bm \mu_{s,h}^k \|^2_{\Gamma_{c,h}^k} \right) \le E_h^0 - E_h^{m+1} < E_h^0 - E_* < \infty. 
\end{align*}
Since the partial sums on the left-hand side are nondecreasing and
uniformly bounded, the corresponding infinite series converges.
Consequently, for the fixed time step $\tau>0$, each nonnegative
dissipation term vanishes asymptotically. In particular, since $\beta_s > 0$:
\[
\|D(\Bu_h^{m+1})\|_{\Omega_h^m} \to 0 \quad \text{and} \quad \|\bbP_s \Bu_h^{m+1}\|_{\Gamma_{s,h}^m} \to 0.
\] 
Moreover, the discrete impermeability constraint in \eqref{BGN-3d-fluid},
together with the weighted nodal pairing defined in \eqref{eq:mass_lumping_nodal}, implies
\[
\Bu_h^{m+1}\cdot \Bn_s=0
\qquad\text{on }\Gamma_{s,h}^m,
\]
since the pairing involves all quadratic boundary degrees of freedom.
Hence $\Bu_h^{m+1} = \bbP_s \Bu_h^{m+1}$ on $\Gamma_{s,h}^m$, and therefore
\[
\|\Bu_h^{m+1}\|_{\Gamma_{s,h}^m}
=
\|\bbP_s \Bu_h^{m+1}\|_{\Gamma_{s,h}^m}
\to 0.
\] 
By our geometric assumptions on the discrete domains $\Omega_h^m$ and the solid boundary $\Gamma_{s,h}^m$, we apply Lemma \ref{lem:korn} to the discrete velocity $\Bu_h^{m+1}$, which yields
\[
\|\Bu_h^{m+1}\|_{\Omega_h^m}^2 \le C_\Omega \big( \|D(\Bu_h^{m+1})\|_{\Omega_h^m}^2 + \|\Bu_h^{m+1}\|_{\Gamma_{s,h}^m}^2 \big).
\]
Taking the limit $m \to \infty$ on both sides, the right-hand side goes to zero, which leads to
$$
\lim_{m\to\infty} \|\Bu_h^{m+1}\|_{\Omega_h^m} = 0.
$$
This completes the proof. 
\end{proof}

\begin{remark}\upshape
Theorem~\ref{coro} shows that the
proposed scheme drives the discrete velocity to zero as
\(m\to\infty\). In this sense, the numerical solution
approaches a rest state, and no persistent spurious fluid velocity remains in
the long-time regime.

This conclusion relies essentially on the exact cancellation of the
gravitational work with the discrete variation of the gravitational potential
energy. If gravity is instead treated directly as a body force in the
momentum equation, according to \eqref{eq:gravity-residual-direct}, the discrete energy relation contains an additional term,
\begin{align}
E_h^{m+1} + \frac{1}{2} \| \Bu_h^{m+1} \circ \psi_h^m - \Bu_h^m \|^2_{\Omega_h^{m-1}} 
+ \frac{2\tau}{Re} \| D(\Bu_h^{m+1}) \|^2_{\Omega_h^m}& \notag \\
+ \tau \beta_s \| \bbP_s \Bu_h^{m+1} \|^2_{\Gamma_{s,h}^m}
+ \tau \beta_c \|\widetilde \Bw_h^{m+1} \cdot \bm \mu_{s,h}^m \|^2_{\Gamma_{c,h}^m} 
&\leq E_h^m + \mathcal R_g^{m+1}.
\end{align}
In general, \(\mathcal R_{g}^{m+1}\neq 0\) and has no definite sign. Therefore, one cannot conclude that
\[
\Bu_h^{m+1}\to \bm{0}.
\]
Consequently, for the direct treatment of gravity, the numerical solution may
fail to relax to a true discrete rest state, and a nonzero residual velocity
may persist near equilibrium. 
Such a persistent nonphysical velocity can be interpreted as a spurious
velocity.  
\end{remark}

\begin{remark}\upshape
\label{rmk:two-phase}
Although this paper focuses on a single-phase free-boundary problem, the proposed treatment for gravity can be naturally extended to two-phase flows, where the gravitational contribution can also be rewritten as an interface integral associated with the motion of the phase boundary. 
We note that Duan et al. \cite{Duan2022} treated the gravitational potential energy by imposing a divergence-free constraint on the mesh velocity, namely
\(\nabla\cdot \Bw_h^m = 0\). 
In contrast, our approach does not require the bulk mesh velocity to be divergence-free. 
\end{remark}

\section{Newton iteration method}
\label{sec:Newton}
The structure-preserving schemes \eqref{BGN-2d} and \eqref{BGN-3d} are inherently nonlinear, because the intermediate normal vector and the gravitational term are both parameterized by the unknown mesh velocity \(\widetilde{\Bw}_h^{m+1}\). 
Therefore, these fully discrete systems require the solution of a nonlinear algebraic problem at each time step. 
To this end, we adopt Newton's method, which provides an efficient iterative solver for the present formulation.

The Newton linearization in two and three dimensions is based on the same idea. 
Since the three-dimensional case contains all essential ingredients and is technically more involved, we only present the three-dimensional construction below; the two-dimensional version follows analogously and is simpler.

We introduce the unknown vector
\begin{equation}\label{eq:U3D-def}
\BU_h^{m+1}
:=
\bigl(
\Bu_h^{m+1},
\widetilde p_h^{m+1},
\kappa_{s,h}^{m+1},
\widetilde H_h^{m+1},
\widetilde \Bw_h^{m+1},
\varphi_{c,h}^{m+1},
\psi_{c,h}^{m+1}
\bigr).
\end{equation}
Given $\BU_h^m$, the discrete problem at time level $t_{m+1}$ can be written in residual form as follows: find
$\BU_h^{m+1}$ such that
\begin{equation}\label{eq:residual-system-3D}
R_i(\BU_h^{m+1})=0,
\qquad i=1,2,3,
\end{equation}
where the residual equations correspond to
\eqref{BGN-3d-fluid}--\eqref{BGN-3d-b}.
More precisely, we define
\begin{align}
&R_1(\BU; \Bv, q, \lambda_s, \phi_c, \chi_c)
:=\;
\frac{1}{2\tau}(\Bu, \Bv)_{\Omega_h^m}
+\frac{1}{2\tau}(\Bu\circ\psi_h^m, \Bv\circ\psi_h^m)_{\Omega_h^{m-1}} \notag \\
&\qquad + b(\Bu_h^m-\Bw_h^m; \Bu\circ\psi_h^m, \Bv\circ\psi_h^m)_{\Omega_h^{m-1}}
+\Big(\frac{2}{Re} D(\Bu),D(\Bv)\Big)_{\Omega_h^m}
-(\widetilde p, \nabla\cdot \Bv)_{\Omega_h^m}
\nonumber\\
&\qquad
+(q, \nabla\cdot \Bu)_{\Omega_h^m}
+\langle \widetilde H\, \Bn_{F,h}^m, \Bv \rangle_{\Gamma_{F,h}^m}
-\langle \kappa_s \Bn_{s}, \Bv\rangle_{\Gamma_{s,h}^m}^{ s, h }
+ \langle \lambda_s \Bn_{s}, \Bu\rangle_{\Gamma_{s,h}^m}^{ s, h }
+ \beta_s \langle \bbP_s \Bu, \bbP_s \Bv \rangle_{\Gamma_{s,h}^m}
\notag \\
&\qquad
+ \frac{1}{W\!e} \langle \widetilde \Bw\cdot \Bn_{s}, \phi_c\rangle_{\Gamma_{c,h}^m}^{ h } 
+ \frac{1}{W\!e} \langle \widetilde \Bw\cdot \bm{\xi}_{c,h}^m, \chi_c\rangle_{\Gamma_{c,h}^m}^{ h }
- \frac{1}{\tau}(\Bu_h^m,\Bv\circ\psi_h^m)_{\Omega_h^{m-1}},
\label{eq:R1-3D}
\\[0.3em]
&R_2(\BU;\bm{\eta})
:=\;
-\langle \widetilde H\, \mathcal N(\widetilde\Bw), \bm{\eta}\rangle_{\Gamma_{F,h}^m}^{ h }
+ \frac{1}{W\!e} \langle \nabla_\Gamma( {\rm id} + \tau \widetilde \Bw), \nabla_\Gamma \bm{\eta} \rangle_{\Gamma_{F,h}^m}
- \frac{1}{W\!e} \langle \varphi_c \Bn_{s}, \bm{\eta} \rangle_{\Gamma_{c,h}^m}^{ h }
\nonumber\\
&\qquad
-\frac{1}{W\!e}\int_{\Gamma_{c,h}^m} \bigl(\cos\theta_Y \,\mathcal M(\widetilde\Bw)+\sin\theta_{c,h}^m\,\Bn_{s}\bigr)\cdot \bm{\eta}\,\Rd s 
+ \beta_c \int_{\Gamma_{c,h}^m} (\widetilde \Bw \cdot \bm{\mu}_{s,h}^m)\bm{\mu}_{s,h}^m \cdot \bm \eta \Rd s \notag \\
&\qquad 
- \frac{1}{W\!e} \langle \psi_c\, \bm{\xi}_{c,h}^m, \bm{\eta}\rangle_{\Gamma_{c,h}^m}^{ h }
+g \langle \mathcal G^{\rm 3d}(\widetilde \Bw), \bm{\eta}\rangle_{\Gamma_{F,h}^m},
\label{eq:R2-3D}
\\[0.3em]
&R_3(\BU;\phi_F)
:=\;
\langle \widetilde \Bw\cdot \mathcal N(\widetilde\Bw), \phi_F\rangle_{\Gamma_{F,h}^m}^{ h }
-\langle \Bu\cdot \Bn_{F,h}^m, \phi_F\rangle_{\Gamma_{F,h}^m},
\label{eq:R3-3D}
\end{align}
where we use the notations
$$
\mathcal N(\widetilde\Bw)
:=
\frac16 \hat{\Bn}_{F,h}^{m}
+\frac23 \hat{\Bn}_{F,h}^{m+1/2}(\widetilde \Bw)
+\frac16 \hat{\Bn}_{F,h}^{m+1}(\widetilde \Bw),\quad 
\mathcal{M}(\widetilde{\Bw}) := \bm{\mu}_{s,h}^{m+1/2}.
$$

To solve \eqref{eq:residual-system-3D}, we apply Newton's method.
At the $l$-th iteration, let
\[
\BU_h^{(l)}
=
\bigl(
\Bu_h^{(l)},\widetilde p_h^{(l)},\kappa_{s,h}^{(l)},
\widetilde H_h^{(l)},\widetilde{\Bw}_h^{(l)},
\varphi_{c,h}^{(l)},\psi_{c,h}^{(l)}
\bigr)
\]
be the current iterate, and let
\[
\delta \BU_h^{(l)}
=
\bigl(
\delta\Bu_h,\delta\widetilde p_h,\delta\kappa_{s,h},
\delta\widetilde H_h,\delta\widetilde{\Bw}_h,
\delta\varphi_{c,h},\delta\psi_{c,h}
\bigr)
\]
denote the Newton increment.
Starting from an initial guess $\BU_h^{(0)}$, usually taken as
$\BU_h^{(0)}=\BU_h^m$,
we seek $\delta \BU_h^{(l)}$ such that
\begin{align}
DR_1(\BU_h^{(l)})[\delta \BU_h^{(l)}; \Bv, q, \lambda_s, \phi_c, \chi_c]
&= -R_1(\BU_h^{(l)}; \Bv, q, \lambda_s, \phi_c, \chi_c), \label{eq:Newton-system-3D-1} \\
DR_2(\BU_h^{(l)})[\delta \BU_h^{(l)}; \bm{\eta}]
&= -R_2(\BU_h^{(l)}; \bm{\eta}), 
\label{eq:Newton-system-3D-2} \\
DR_3(\BU_h^{(l)})[\delta \BU_h^{(l)}; \phi_F]
&= -R_3(\BU_h^{(l)}; \phi_F).
\label{eq:Newton-system-3D-3}
\end{align}
After solving \eqref{eq:Newton-system-3D-1}--\eqref{eq:Newton-system-3D-3}, the iterate is updated by
\[
\BU_h^{(l+1)}=\BU_h^{(l)}+\delta \BU_h^{(l)}.
\]

The first residual \(R_1\) is linear with respect to the unknowns. Therefore, its Fr\'echet derivative is simply given by the corresponding bilinear form in \eqref{eq:R1-3D}, and we omit its explicit expression for brevity.
The main difficulty lies in the geometric nonlinearities in $R_2$ and $R_3$.

To make the geometric nonlinearities explicit, we define
\[
\mathcal{A}(\widetilde{\Bw})
:= \partial_{s_1}\widetilde{\Bw}\times \hat{\Bt}_{F,2}^m
   + \hat{\Bt}_{F,1}^m\times \partial_{s_2}\widetilde{\Bw},
\qquad
\mathcal{B}(\widetilde{\Bw})
:= \partial_{s_1}\widetilde{\Bw}\times \partial_{s_2}\widetilde{\Bw}.
\]
The corresponding Jacobian is obtained by differentiating these quantities. We first note that
\begin{align}
D\mathcal A(\widetilde\Bw)[\delta\widetilde\Bw]
&=
\partial_{s_1}\delta\widetilde{\Bw}\times \hat{\Bt}_{F,2}^m
+ \hat{\Bt}_{F,1}^m\times \partial_{s_2}\delta\widetilde{\Bw},
\label{eq:DA-3D}
\\
D\mathcal B(\widetilde\Bw)[\delta\widetilde\Bw]
&=
\partial_{s_1}\delta\widetilde{\Bw}\times \partial_{s_2}\widetilde{\Bw}
+ \partial_{s_1}\widetilde{\Bw}\times \partial_{s_2}\delta\widetilde{\Bw}.
\label{eq:DB-3D}
\end{align}
Hence, we have 
\begin{align*}
D \mathcal N(\widetilde\Bw)[\delta\widetilde\Bw]
&=
\frac{\tau}{2}D\mathcal A(\widetilde\Bw)[\delta\widetilde\Bw]
+\frac{\tau^2}{3} D\mathcal B(\widetilde\Bw)[\delta\widetilde\Bw],
\quad 
D\mathcal M(\widetilde\Bw)[\delta\widetilde\Bw]
=
\frac{\tau}{2}(\partial_s\delta\widetilde\Bw)\times \Bn_{s},
\end{align*}
moreover, for the gravitational term, we have
\begin{align*}
D\mathcal G^{\rm 3d}(\widetilde\Bw)[\delta\widetilde\Bw]
&= \tau (\delta \widetilde\Bw \cdot \Be_z)
\left(\frac12 \Bn_{F,h}^m + \frac{\tau}{3}\mathcal{A}(\widetilde{\Bw})
+ \frac{\tau^2}{4}\mathcal{B}(\widetilde{\Bw}) \right) \notag \\
&\quad +\frac{2}{3} z_h^{m+1/2}(\widetilde \Bw) \left( \frac{\tau}{2}D\mathcal A(\widetilde\Bw)[\delta\widetilde\Bw]
+\frac{\tau^2}{4}D\mathcal B(\widetilde\Bw)[\delta\widetilde\Bw] \right) \notag \\
&\quad + \frac{1}{6} z_h^{m+1}(\widetilde \Bw)\left(
\tau D\mathcal A(\widetilde\Bw)[\delta\widetilde\Bw]
+ \tau^2 D\mathcal B(\widetilde\Bw)[\delta\widetilde\Bw] \right).
\end{align*}

Using the above identities, the Fr\'echet derivative of $R_2$ is given by
\begin{align}
&DR_2(\BU_h^{(l)})[\delta \BU_h^{(l)}; \bm{\eta}]
=\;
- \langle \delta\widetilde H_h\, \mathcal N (\widetilde{\Bw}_h^{(l)}), \bm{\eta} \rangle_{\Gamma_{F,h}^m}^{ h }
- \langle \widetilde H_h^{(l)}\,D \mathcal N(\widetilde{\Bw}_h^{(l)})[\delta\widetilde{\Bw}_h], \bm{\eta} \rangle_{\Gamma_{F,h}^m}^{ h }
\notag\\
&\qquad
+ \frac{\tau}{W\!e} \langle \nabla_{\Gamma}\delta\widetilde{\Bw}_h, \nabla_{\Gamma}\bm{\eta} \rangle_{\Gamma_{F,h}^m}
- \frac{1}{W\!e} \langle \delta\varphi_{c,h}\,\Bn_{s}, \bm{\eta} \rangle_{\Gamma_{c,h}^m}^{ h } 
- \frac{1}{W\!e} \langle \delta\psi_{c,h}\,\bm{\xi}_{c,h}^m, \bm{\eta} \rangle_{\Gamma_{c,h}^m}^{ h }
\notag \\
&\qquad
- \frac{1}{W\!e} \int_{\Gamma_{c,h}^m}
\cos\theta_Y\,
D\mathcal M(\widetilde{\Bw}_h^{(l)})[\delta\widetilde{\Bw}_h]\cdot \bm{\eta}\, \Rd s
+ \beta_c \int_{\Gamma_{c,h}^m}
(\delta \widetilde \Bw_h \cdot \bm \mu_{s,h}^m)\bm \mu_{s,h}^m \cdot \bm \eta \Rd s
\notag\\
&\qquad
+ g \langle D\mathcal G^{\rm 3d}(\widetilde{\Bw}_h^{(l)})[\delta\widetilde{\Bw}_h], \bm{\eta} \rangle_{\Gamma_{F,h}^m},
\label{eq:DR2-3D}
\end{align}
while the derivative of $R_3$ reads
\begin{align}
DR_3(\BU_h^{(l)})[\delta \BU_h^{(l)}; \phi_F]
=\;&
\Bigl\langle
\delta\widetilde{\Bw}_h\cdot \mathcal N (\widetilde{\Bw}_h^{(l)})
+
\widetilde{\Bw}_h^{(l)}\cdot
D \mathcal{N}(\widetilde{\Bw}_h^{(l)})[\delta\widetilde{\Bw}_h],
\phi_F
\Bigr\rangle_{\Gamma_{F,h}^m}^{ h }
\notag\\
&\;
- \langle \delta\Bu_h \cdot \Bn_{F,h}^m, \phi_F \rangle_{\Gamma_{F,h}^m}.
\label{eq:DR3-3D}
\end{align}

Thus, each Newton step amounts to solving the linear system
\eqref{eq:Newton-system-3D-1}--\eqref{eq:Newton-system-3D-3} for the increment
$\delta \BU_h^{(l)}$.

In practice, the iteration is terminated once the relative change
of the free-boundary mesh velocity becomes sufficiently small:
\begin{equation}
\frac{
 \bigl\|
 \widetilde{\Bw}_h^{(l + 1)}
 -
 \widetilde{\Bw}_h^{(l)}
 \bigr\|_{L^2(\Gamma_{F,h}^m)}
}{
 1+
 \bigl\|
 \widetilde{\Bw}_h^{(l + 1)}
 \bigr\|_{L^2(\Gamma_{F,h}^m)}
}
<\varepsilon_N,
\qquad
\varepsilon_N=10^{-12}.
\label{eq:newton-tol}
\end{equation}

\section{Numerical tests}
\label{sec:numerical}
In this section, we present numerical experiments for the proposed ALE--FEM method in both two and three spatial dimensions. 
Unless otherwise stated, we take $\beta_c=0.01$. 
Although the theoretical decay of the discrete velocity is established for $\beta_s>0$, 
we test both the case $\beta_s=0$ and $\beta_s > 0$.
The case \(\beta_s=0\) is used as an additional robustness test beyond the scope of Theorem~\ref{coro}.
All computations are carried out using an implementation based on the finite element library NGSolve; see \url{https://ngsolve.org}.

\subsection{Two-dimensional droplet problem}
\label{sec4.3}
We first consider the two-dimensional evolution of a droplet under gravity with constants \(W\!e=1\) and
\(Re=0.5\).
The initial droplet is chosen as a semicircle, with initial contact angle close to \(90^\circ\). 
We compare the numerical solutions for different gravitational constants $g$ and contact angles $\theta_Y$. 
In particular, the wetting and dewetting cases correspond to the static contact angles
\(\theta_Y = 60^\circ\) and \(\theta_Y=120^\circ\), respectively.

Figure~\ref{fig:droplet_2D_1} shows the computed droplet profiles and meshes for \(g=0, 1, 2, 10\). 
In both the wetting and dewetting cases,
the proposed ALE--BGN--MDR method maintains good mesh quality during the evolution, especially near the moving contact points.
Compared with the case without gravity, the droplet under gravity becomes flatter, which is consistent with the physical effect of gravity.
Moreover, the droplet becomes flatter as $g$ increases and our method 
still maintains high mesh quality even for $g=10$. 

Figure~\ref{fig:droplet_2D_2} shows that the contact angle converges to $\theta_Y$ and the
time evolution of the scaled volume error $\left| |\Omega_h^m| - |\Omega_h^0| \right| / |\Omega_h^0|$ with $g=1$.
The relative volume error remains at the level of \(10^{-15}\), indicating that the proposed scheme preserves the droplet volume up to machine precision.

\begin{figure}[htbp]
\centering
\captionsetup[subfigure]{skip=-1pt}

\begin{subfigure}[b]{0.42\textwidth}
    \centering
    \includegraphics[width=\textwidth]{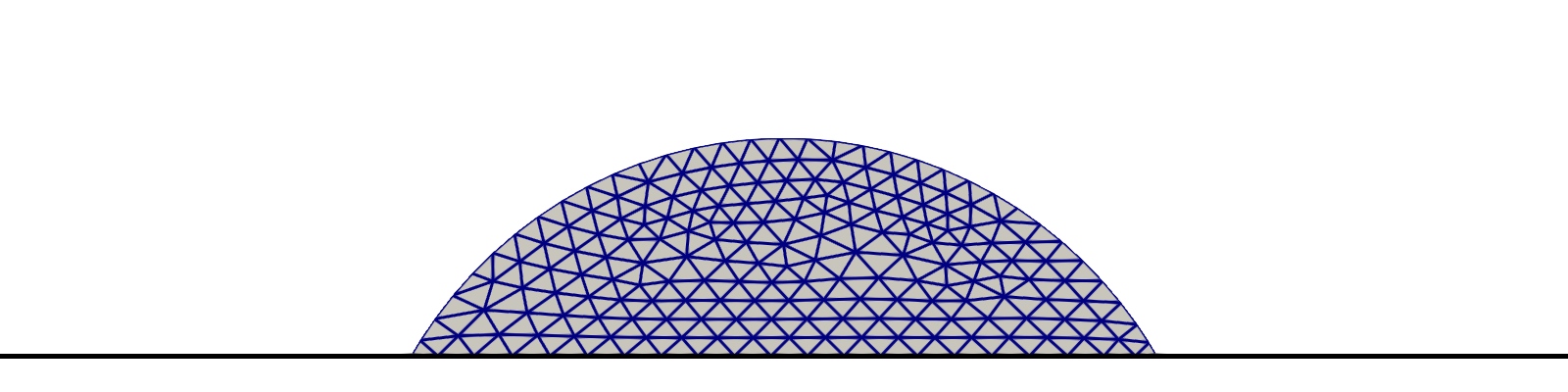}
    \caption{\(\theta_Y=60^\circ,\ g=0\)}
\end{subfigure}
\hspace{0.04\textwidth}
\begin{subfigure}[b]{0.42\textwidth}
    \centering
    \includegraphics[width=\textwidth]{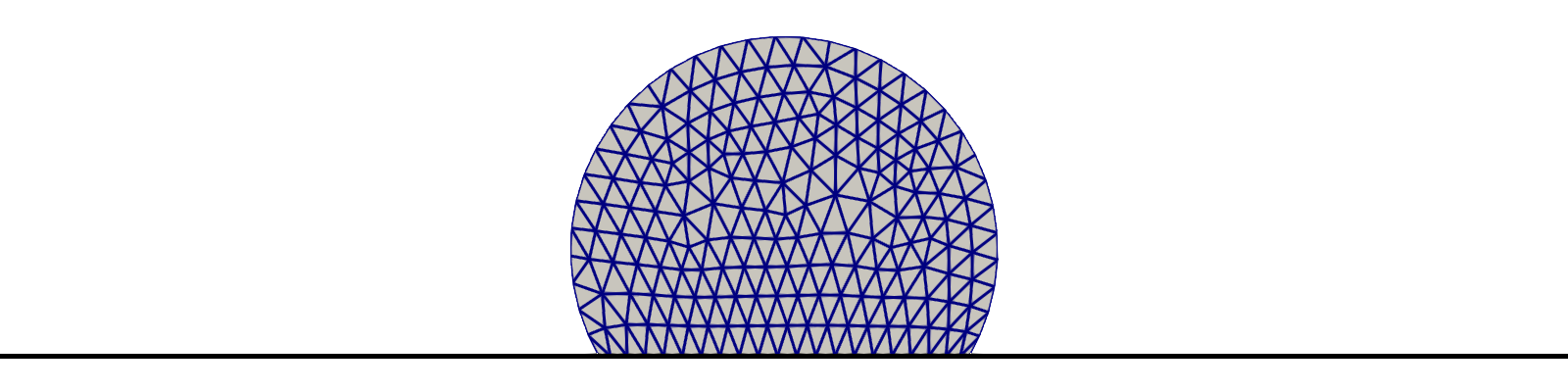}
    \caption{\(\theta_Y=120^\circ,\ g=0\)}
\end{subfigure}

\vspace{-0.5cm}

\begin{subfigure}[b]{0.42\textwidth}
    \centering
    \includegraphics[width=\textwidth]{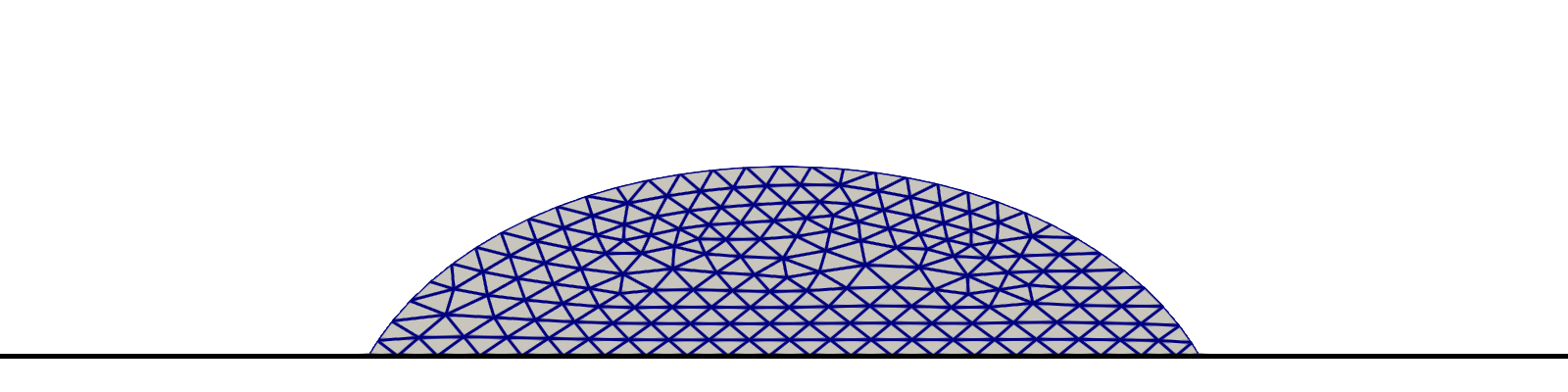}
    \caption{\(\theta_Y=60^\circ,\ g=1\)}
\end{subfigure}
\hspace{0.04\textwidth}
\begin{subfigure}[b]{0.42\textwidth}
    \centering
    \includegraphics[width=\textwidth]{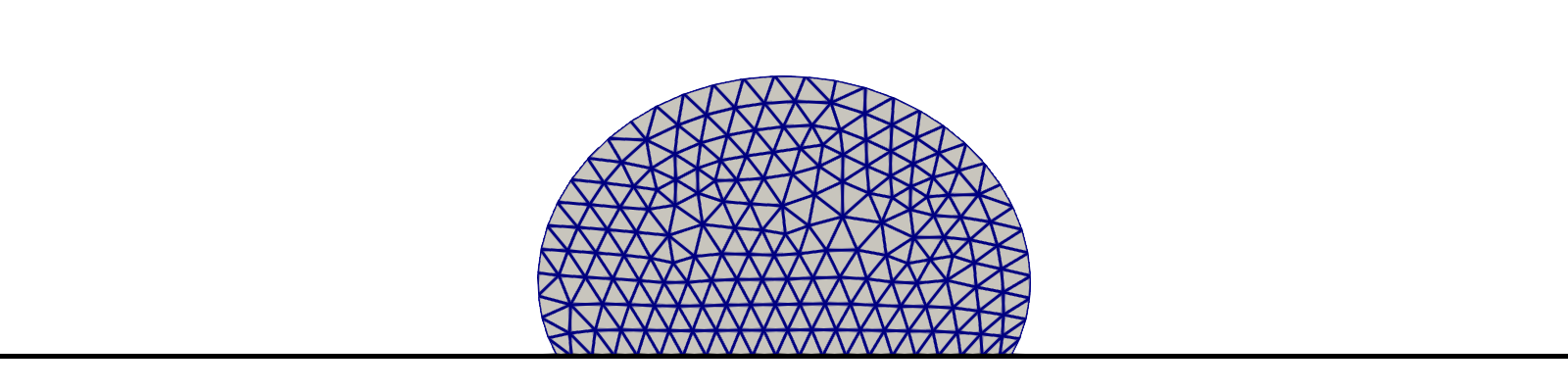}
    \caption{\(\theta_Y=120^\circ,\ g=1\)}
\end{subfigure}

\vspace{-0.5cm}

\begin{subfigure}[b]{0.42\textwidth}
    \centering
    \includegraphics[width=\textwidth]{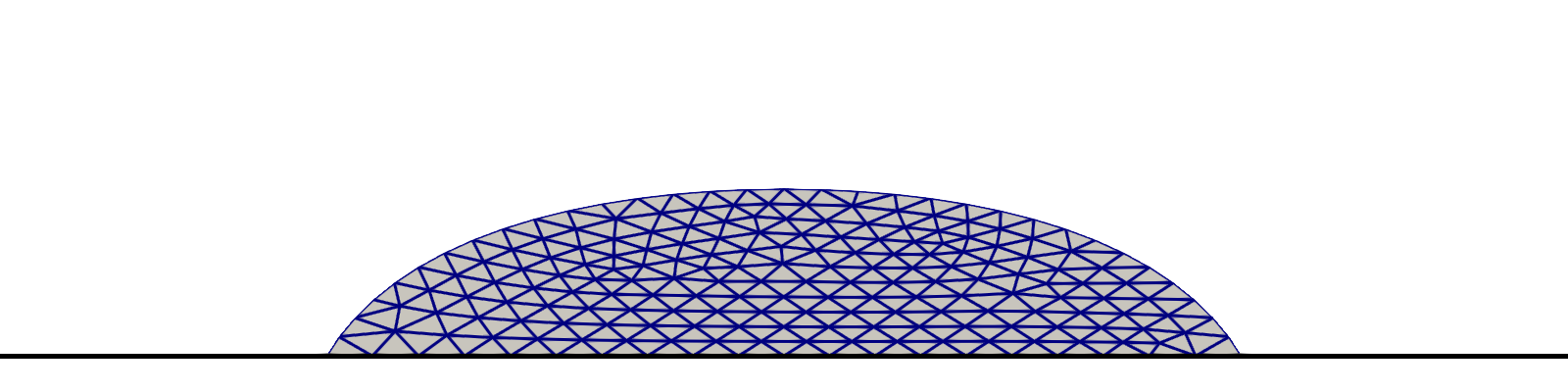}
    \caption{\(\theta_Y=60^\circ,\ g=2\)}
\end{subfigure}
\hspace{0.04\textwidth}
\begin{subfigure}[b]{0.42\textwidth}
    \centering
    \includegraphics[width=\textwidth]{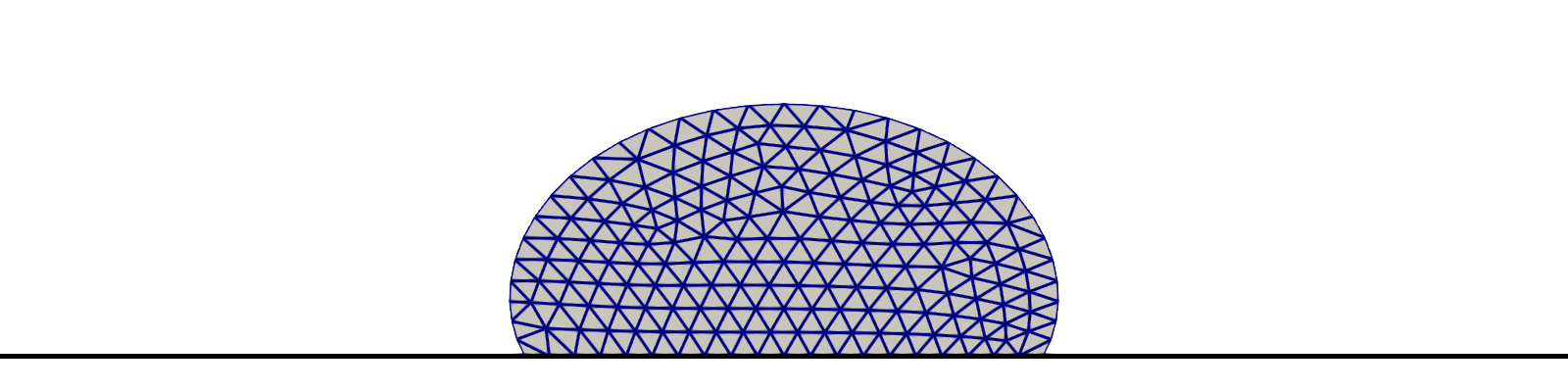}
    \caption{\(\theta_Y=120^\circ,\ g=2\)}
\end{subfigure}

\vspace{-0.5cm}

\begin{subfigure}[b]{0.42\textwidth}
    \centering
    \includegraphics[width=\textwidth]{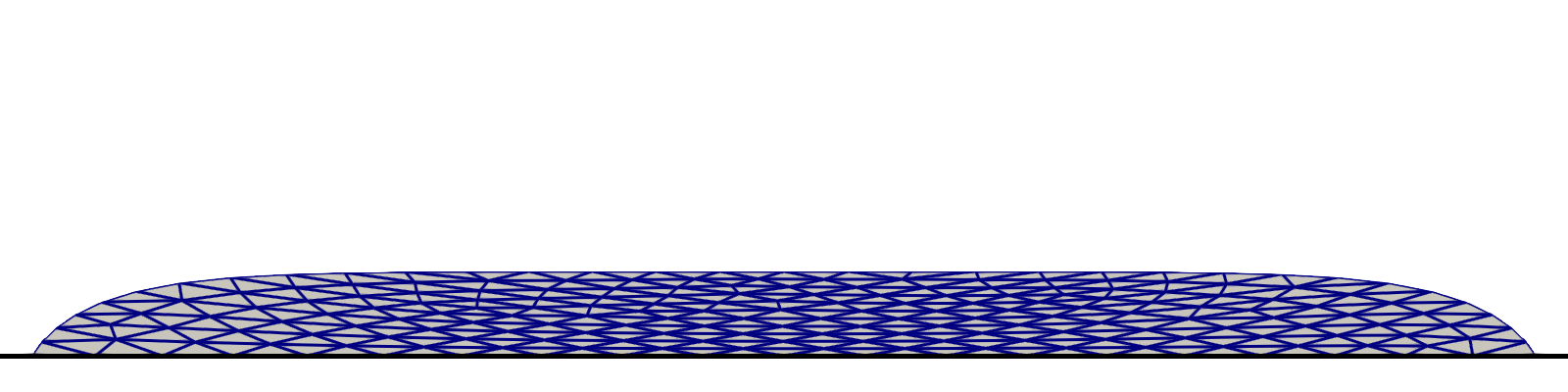}
    \caption{\(\theta_Y=60^\circ,\ g=10\)}
\end{subfigure}
\hspace{0.04\textwidth}
\begin{subfigure}[b]{0.42\textwidth}
    \centering
    \includegraphics[width=\textwidth]{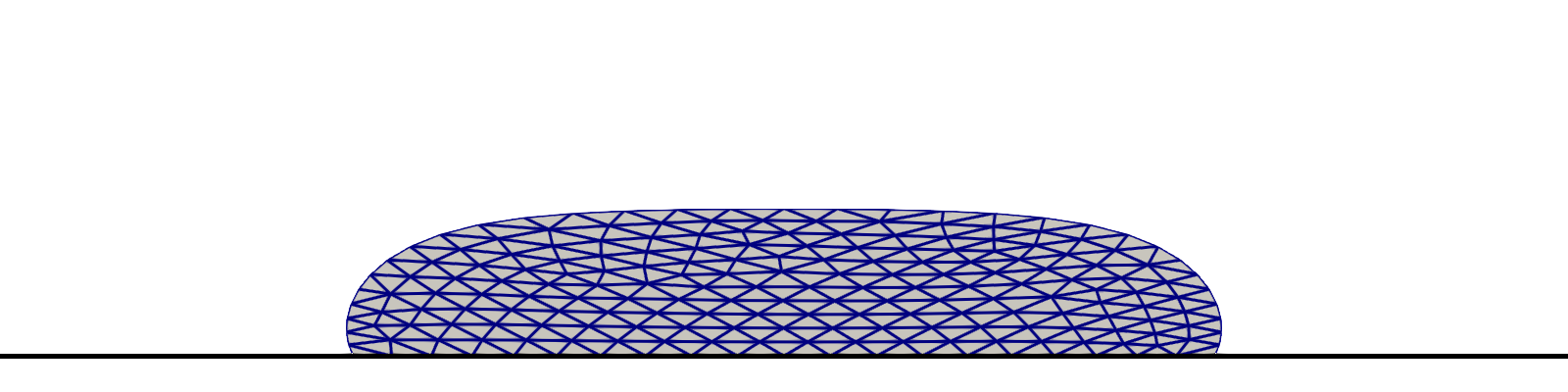}
    \caption{\(\theta_Y=120^\circ,\ g=10\)}
\end{subfigure}


\caption{Computed meshes at \(t=5.0\) for two-dimensional wetting and dewetting tests under different gravitational strengths.}
\label{fig:droplet_2D_1}
\end{figure}

\begin{figure}[!htbp]
	\centering
    \includegraphics[width=0.48\textwidth]{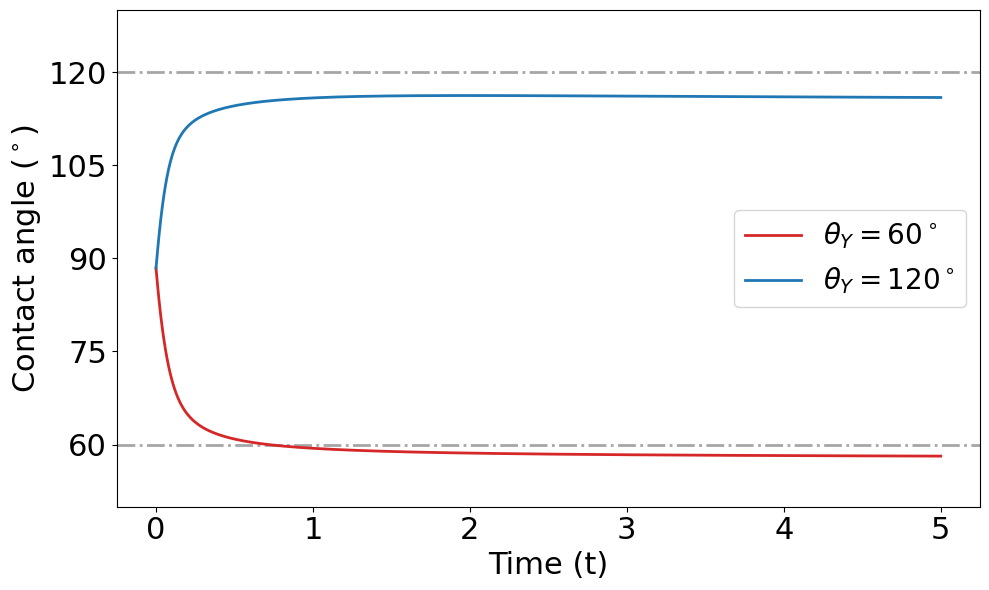}
    \includegraphics[width=0.48\textwidth]{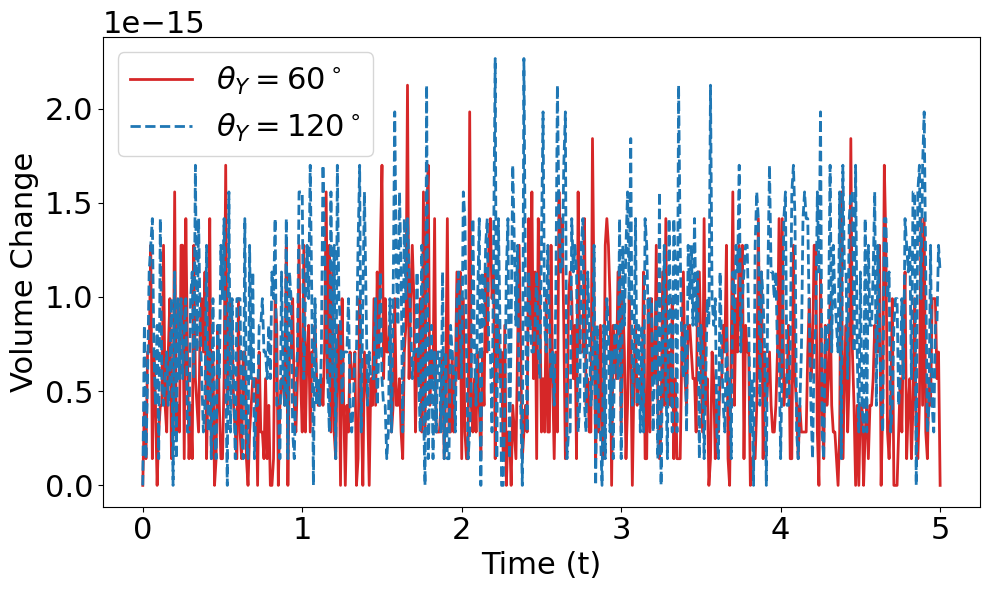}
	\vspace{-8pt}
	\caption{ Evolution of the contact angle (left) and the scaled volume error (right) $\left| |\Omega_h^m| - |\Omega_h^0| \right| / |\Omega_h^0|$ with $g=1$.} 
	\label{fig:droplet_2D_2}
\end{figure}

\subsection{Energy dissipation}
\label{sec:numerical2}
We now compare the proposed method \eqref{BGN-2d} with the standard direct treatment of gravity \eqref{BGN-old}. 
We take
\[
Re=20,\qquad W\!e=1,\qquad \theta_Y=90^\circ,\qquad g=1,
\]
and compare the numerical results for three time step sizes $\tau=0.01, 0.02, 0.04$.
Figure~\ref{fig:Droplet_Energy} shows the time histories of the discrete total energy and the kinetic energy. 
In the presence of gravity, the total energy is expected to dissipate in time and eventually approach a limiting constant as the droplet relaxes to equilibrium. 
This behavior is clearly observed for the proposed method: the total energy decreases initially and then reaches a nearly constant value in the long-time regime. 
By contrast, for the direct treatment of gravity, the discrete total energy continues to decrease even after the droplet is close to equilibrium, and the artificial decay becomes more pronounced as the time step size increases. This is consistent with the inconsistency error \(\mathcal{R}_g^{m+1}\) derived in \eqref{eq:gravity-residual-direct}. The kinetic energy in the right panel of Figure~\ref{fig:Droplet_Energy} confirms the same observation: while the proposed method drives the kinetic energy to the level of machine precision, the direct method leaves a nonzero residual kinetic energy in the long-time regime, indicating persistent spurious velocities.

To further quantify this effect, Figure~\ref{fig:Droplet_Energy_diff} plots the energy increment between two consecutive time levels,
$E_h^{m+1}-E_h^m$.
For the proposed method, this quantity tends to zero, but for the direct method, the energy increment remains at a nonzero negative level, showing that the discrete total energy keeps decreasing artificially and the computed solution fails to settle into a true equilibrium. 
These observations confirm the importance of the gravity-consistent formulation for eliminating the associated spurious velocities.

\begin{figure}[htbp]
\centering
\includegraphics[width=0.48\textwidth]{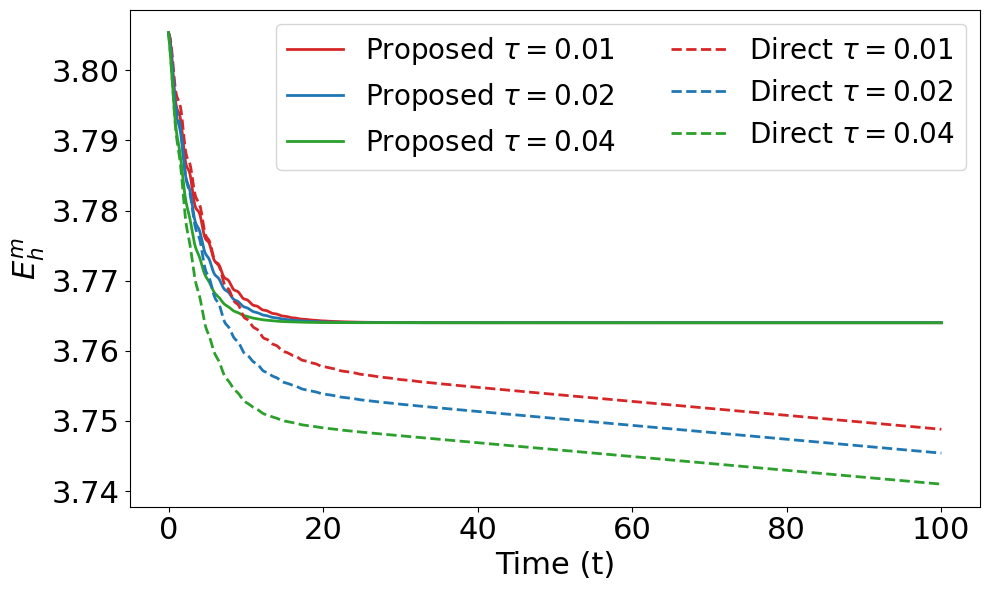}
\includegraphics[width=0.48\textwidth]{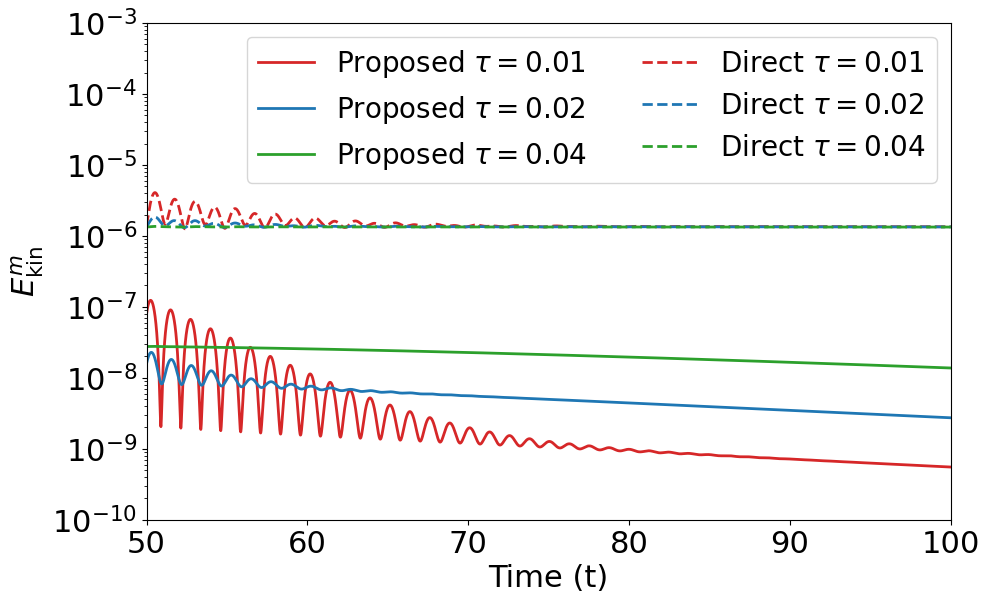}
\caption{
Comparison between the proposed method and the direct treatment of gravity for different time step sizes.
Left: total energy \(E_h^m\).
Right: kinetic energy \(E^{m}_{\rm kin}\).
}
\label{fig:Droplet_Energy}
\end{figure}

\begin{figure}[htbp]
\centering
\includegraphics[width=0.48\textwidth]{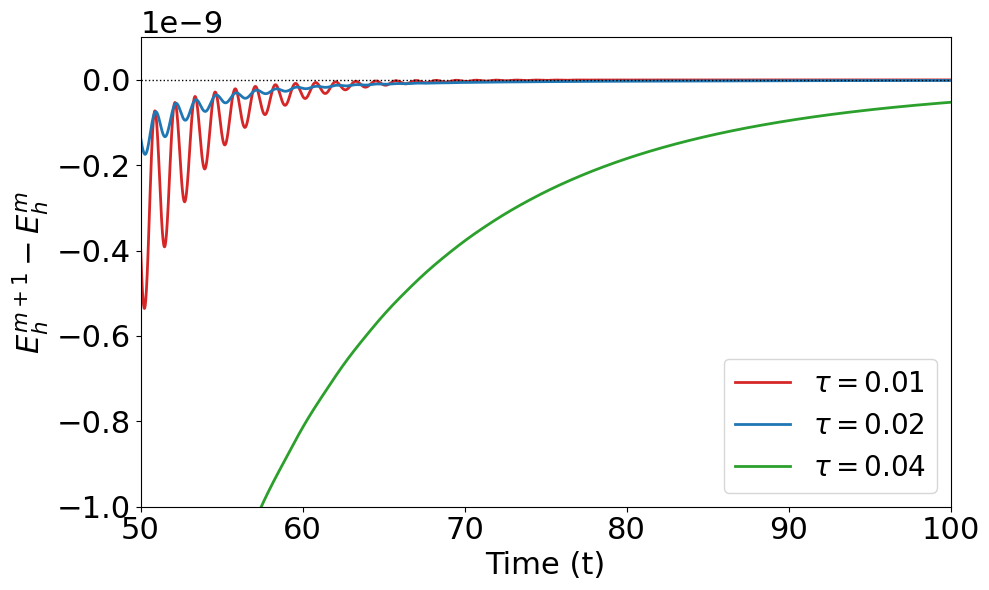}
\includegraphics[width=0.48\textwidth]{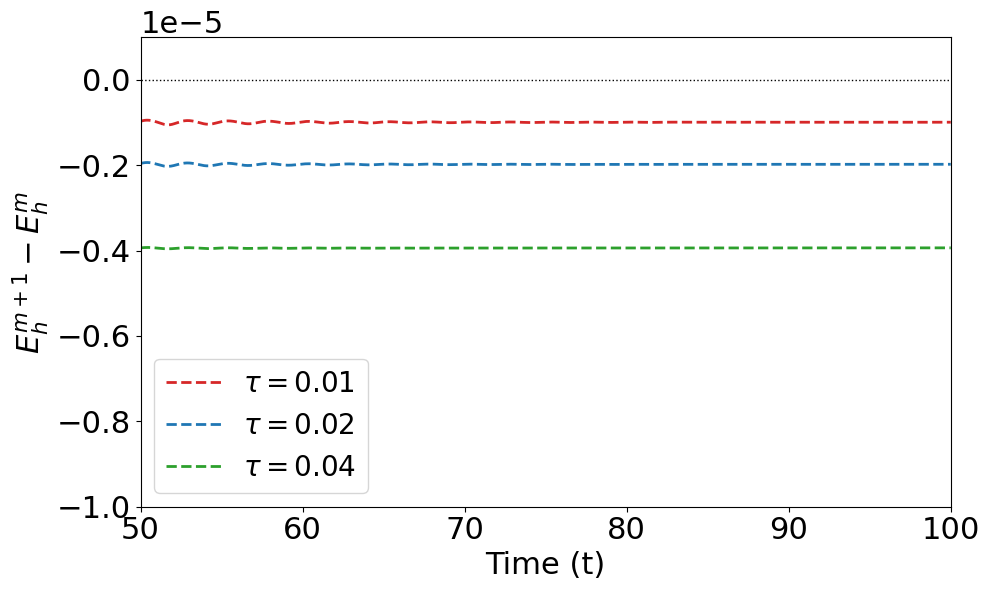}
\caption{Time evolution of the discrete energy increment \(E_h^{m+1}-E_h^m\).
Left: proposed method \eqref{BGN-2d}.
Right: direct method \eqref{BGN-old}.}
\label{fig:Droplet_Energy_diff}
\end{figure}

\subsection{Navier slip coefficient $\beta_s$}
In the Navier-slip boundary condition \eqref{2:PDE-bs2-1}, 
\( \beta_s \) measures the wall friction acting on the tangential slip velocity. Its contribution to discrete energy dissipation is $\beta_s \| \mathbb P_s \Bu_h \|^2_{\Gamma_{s,h}}$. 
To illustrate the effect of \(\beta_s\), we use the same parameter setting as in Subsection~\ref{sec:numerical2} with time step \(\tau=0.04\), and compare \(\beta_s=0,0.1,1\) for both the direct method \eqref{BGN-old} and the proposed method \eqref{BGN-2d}.
Figure~\ref{fig:Droplet_BetaS} shows the evolution of the total energy and kinetic energy. 
At early times, larger \(\beta_s\) leads to faster energy decay, since the wall friction dissipates the tangential motion more strongly.

For the proposed gravity-consistent method, the total-energy curves for different \(\beta_s\) converge to the same limiting value. 
This is expected because \(\beta_s\) affects the relaxation rate but not the static equilibrium energy, which is determined by the volume, surface tension, gravity, and contact angle.

In contrast, for the direct method, the long-time energies do not collapse to a common value. 
As shown in the previous subsection, the direct method leaves a nonzero residual kinetic energy due to gravity-induced spurious velocities. 
A larger \( \beta_s \) suppresses the tangential component of this artificial motion near the solid substrate, thereby reducing the resulting long-time energy drift. 
Hence the direct method can exhibit larger long-time total energy for larger \( \beta_s \), reflecting a \( \beta_s \)-dependent numerical state rather than a physical equilibrium.

\begin{figure}[htbp]
\centering
\captionsetup[subfigure]{skip=2pt}

\begin{subfigure}[t]{0.48\textwidth}
    \centering
    \includegraphics[width=\textwidth]{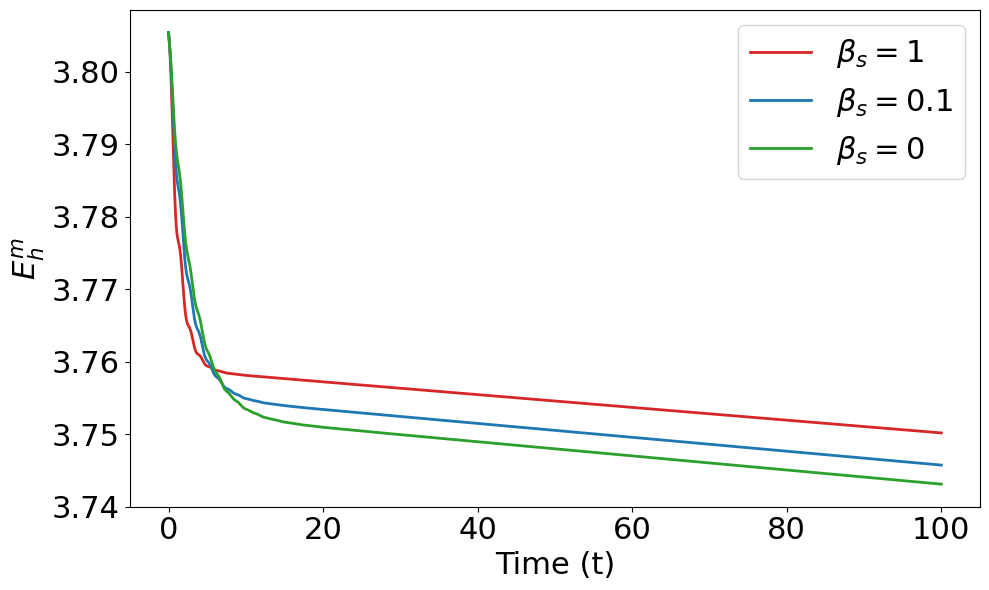}
    \caption{Total energy, direct method \eqref{BGN-old}}
\end{subfigure}
\hfill
\begin{subfigure}[t]{0.48\textwidth}
    \centering
    \includegraphics[width=\textwidth]{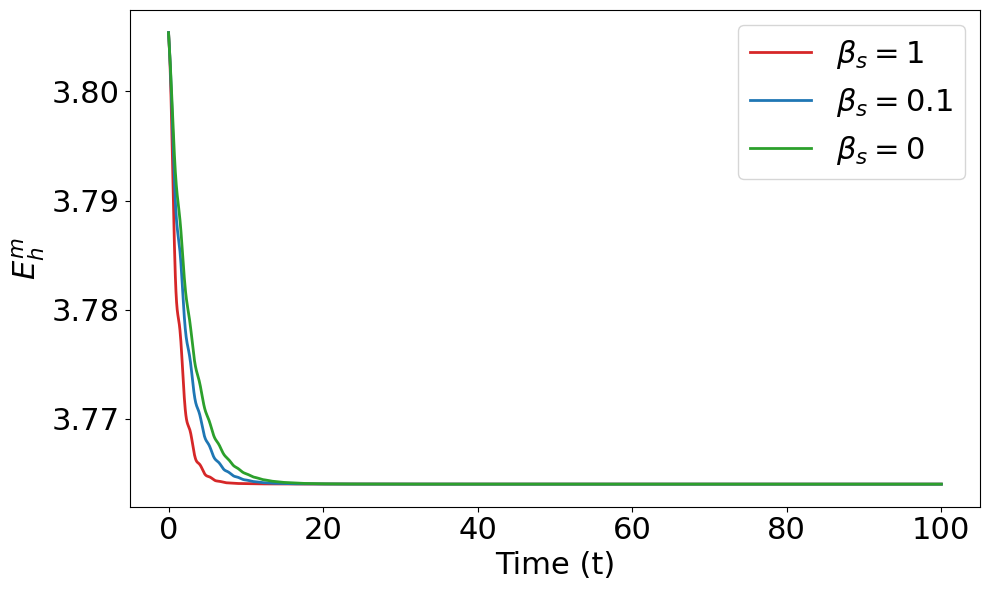}
    \caption{Total energy, proposed method \eqref{BGN-2d}}
\end{subfigure}


\begin{subfigure}[t]{0.48\textwidth}
    \centering
    \includegraphics[width=\textwidth]{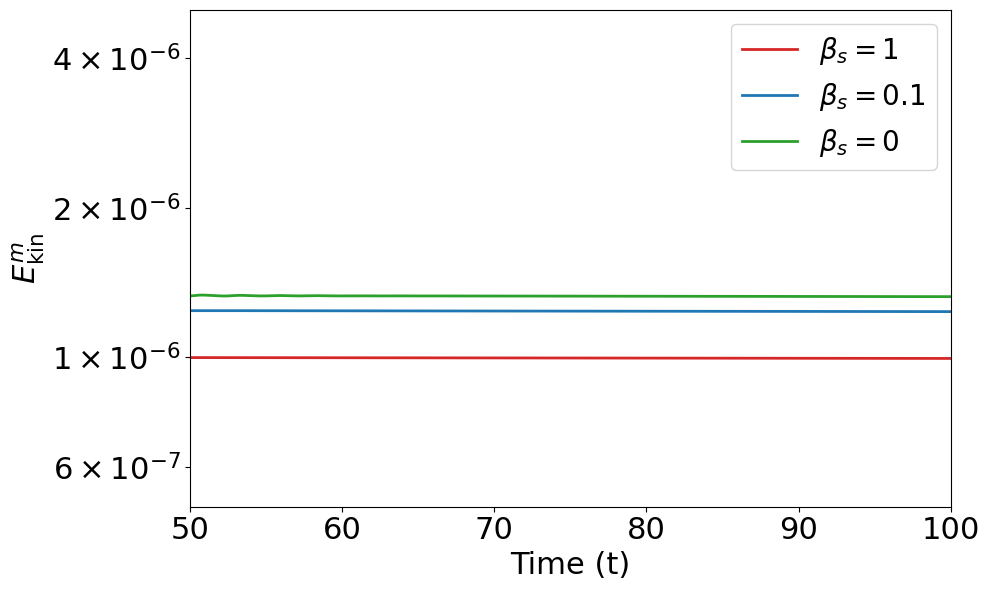}
    \caption{Kinetic energy, direct method \eqref{BGN-old}}
\end{subfigure}
\hfill
\begin{subfigure}[t]{0.48\textwidth}
    \centering
    \includegraphics[width=\textwidth]{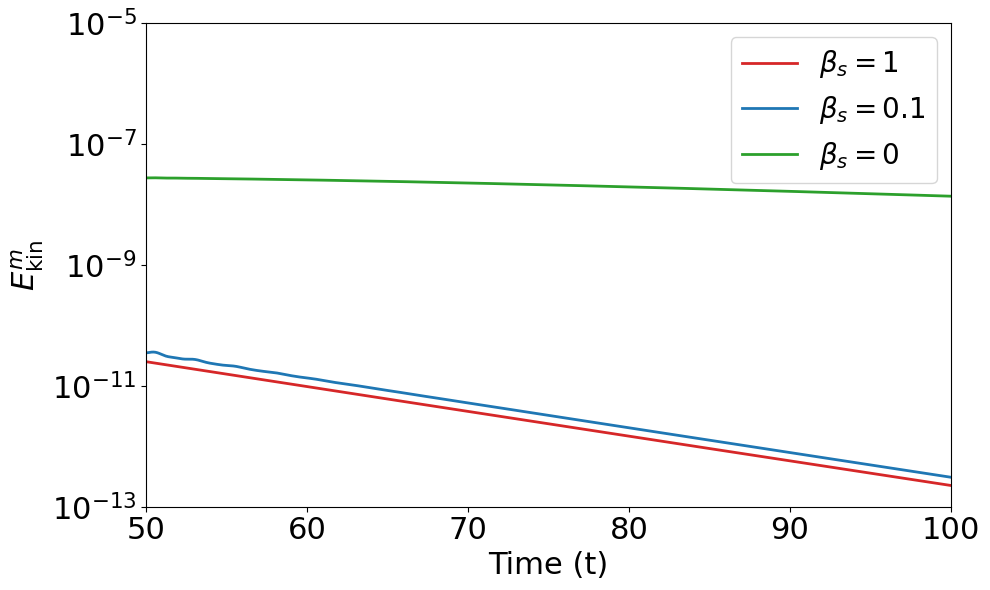}
    \caption{Kinetic energy, proposed method \eqref{BGN-2d}}
\end{subfigure}
\caption{Effect of the wall-friction coefficient \(\beta_s\). 
}
\label{fig:Droplet_BetaS}
\end{figure}

\subsection{Damping test}
Following \cite{Bansch2001}, we consider the droplet motion driven by surface tension.
Due to the initial imbalance of curvature along the free surface, the droplet starts to move and exhibits a damped oscillatory behavior.
We take
\(
Re=20, \theta_Y=60^\circ,
\)
and study the influence of different surface tension coefficients.
To quantify the capillary oscillation, we monitor the vertical coordinate of the droplet centroid,
\[
z_{\rm cen}^m
:= \frac{1}{|\Omega_h^m|}\int_{\Omega_h^m} z(\Bx) \,\Rd x,
\]
where \(z(\Bx)\) denotes the vertical coordinate. The computation is carried out up to \(T=10\). 
From the time history of \(z_{\rm cen}^m\), we identify successive local maxima at
\(
t_{p_0}<t_{p_1}<\cdots<t_{p_n}.
\)
The mean oscillation period and frequency are then computed as
\[
\overline{T}=\frac{t_{p_n}-t_{p_0}}{n},
\qquad
f=1 / \overline{T}
=\frac{n}{t_{p_n}-t_{p_0}}.
\]
Figure~\ref{fig:Droplet_CenHeight} shows the oscillation of \(z_{\rm cen}^m\) and the corresponding frequency scaling. 
In the absence of gravity, the capillary scaling \(f\sim W\!e^{-1/2}\) is reported in \cite{Bansch2001}. 
Our numerical results reproduce this behavior for \(g=0\). For the weak-gravity case considered here, the measured frequencies for \(g=1\) also follow the same scaling trend.

\begin{figure}[htbp]
\centering
\includegraphics[width=0.48\textwidth]{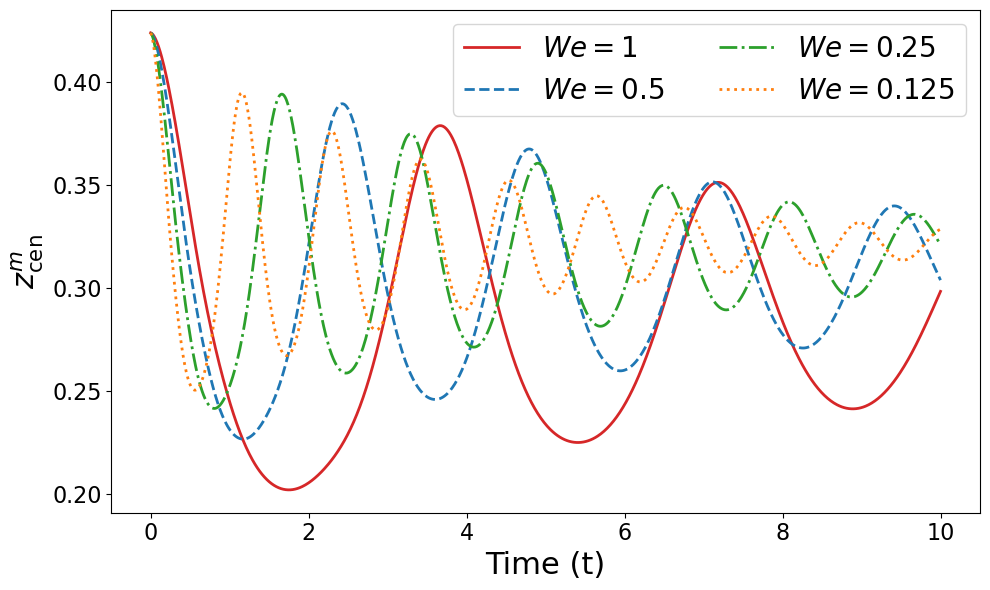}
\hfill
\includegraphics[width=0.48\textwidth]{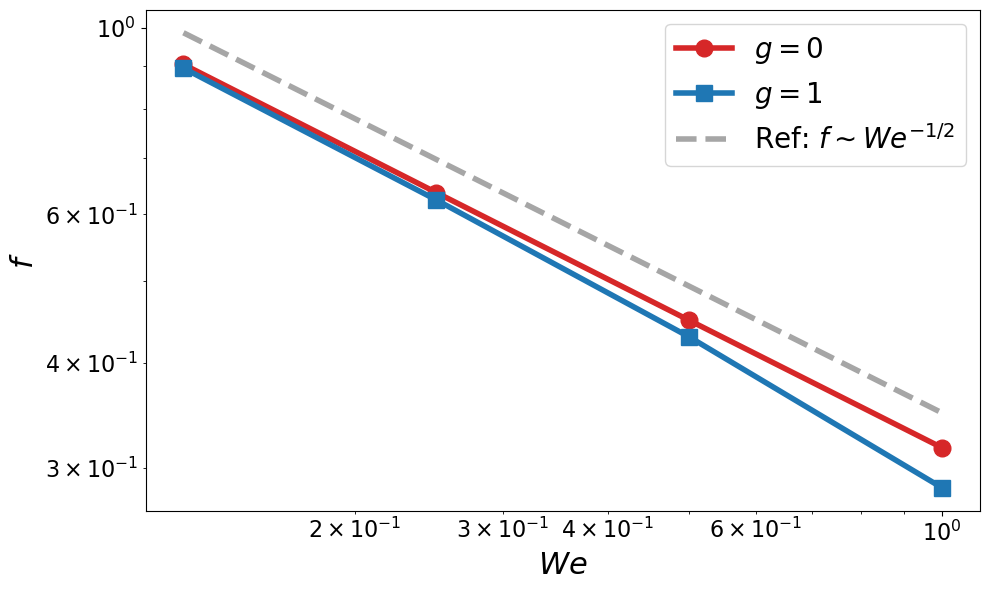}
\caption{Capillary oscillation and frequency. Left: time histories of \(z^m_{\rm cen}\).
Right: measured mean frequency for \(g=0, 1\), together with the reference \(f\sim W\!e^{-1/2}\).
}
\label{fig:Droplet_CenHeight}
\end{figure}

\subsection{Three-dimensional ellipsoid droplet}
We next consider a three-dimensional droplet problem. The initial droplet geometry is configured as a semi-ellipsoid (restricted to the upper half-space $z > 0$), with semi-principal axes of lengths 1.5, 1, and 1 along the $x$-, $y$-, and $z$-directions, respectively. 
The initial mesh size is set to $h=0.15$. 
Two different contact angles are tested: the hydrophilic case $\theta_Y=60^\circ$ and the hydrophobic case $\theta_Y=120^\circ$. 
In both simulations, the dimensionless parameters are chosen as $g=1$, $Re=20$, and $W\!e=1$.

Figure \ref{fig:Ellipsoid_3D_1} shows the evolution of the droplet surface over one oscillation cycle. 
We observe that the droplet undergoes a clear capillary-gravity-driven oscillation before approaching its equilibrium shape. 
In both cases, the proposed BGN-MDR strategy preserves excellent mesh quality throughout the evolution, with no visible mesh degeneration near the free surface or the moving contact line. 
As shown in Figure \ref{fig:Ellipsoid_3D_2}, the relative volume error remains at the level of $10^{-13}$, demonstrating that the proposed method conserves the droplet volume up to machine precision. Moreover, the Newton solver typically converges within 7 to 10 iterations per time step, indicating that the fully discrete nonlinear system can be solved efficiently.

\begin{figure}[htbp]
\centering
\setlength{\tabcolsep}{2pt}
\renewcommand{\arraystretch}{1.15}
\begin{tabular}{c c c c c}
\(t=0\) & \(t=0.6\) & \(t=1.2\) & \(t=1.8\) & \(t=2.4\) \\
\includegraphics[width=0.19\textwidth]{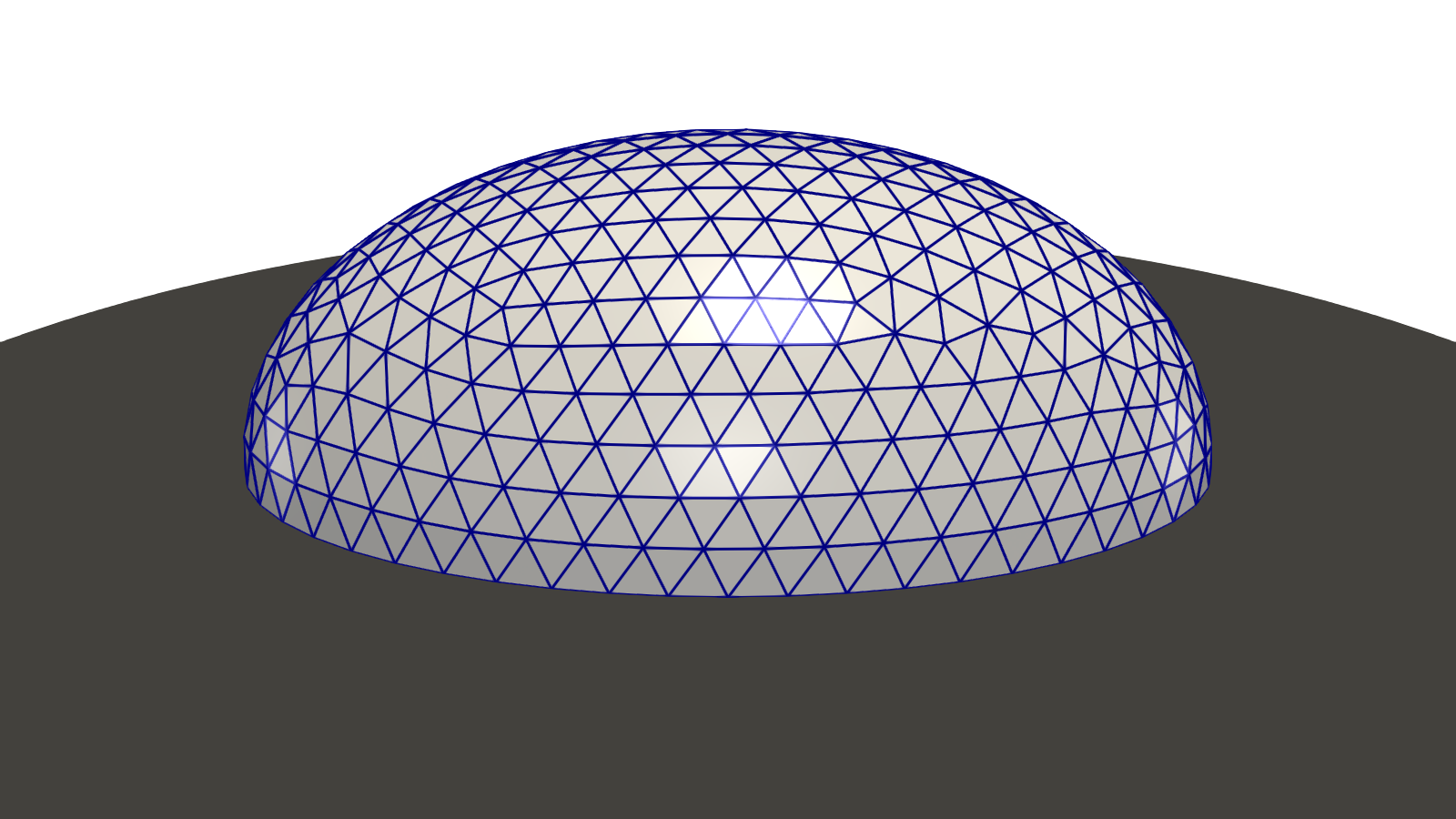} &
\includegraphics[width=0.19\textwidth]{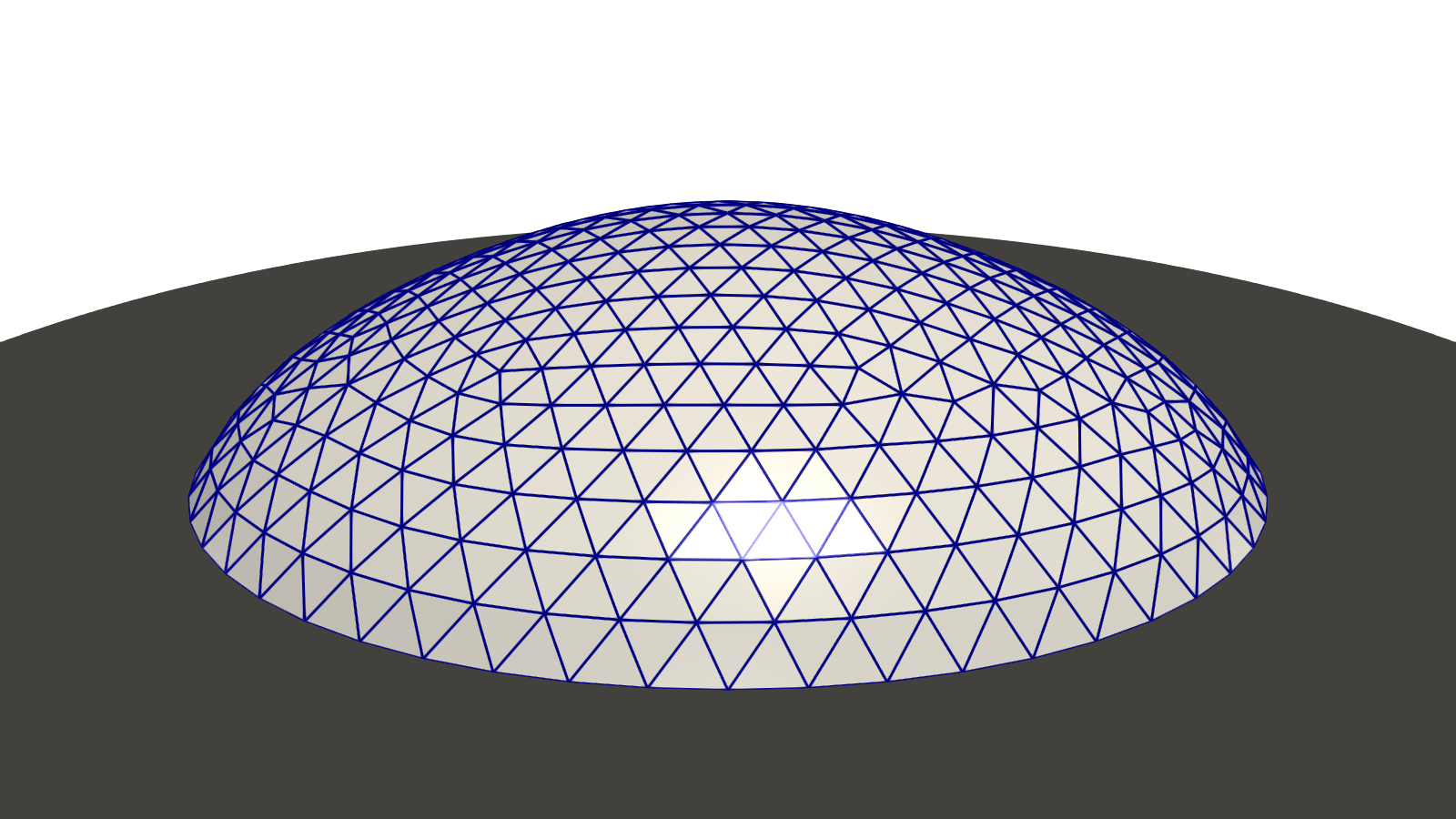} &
\includegraphics[width=0.19\textwidth]{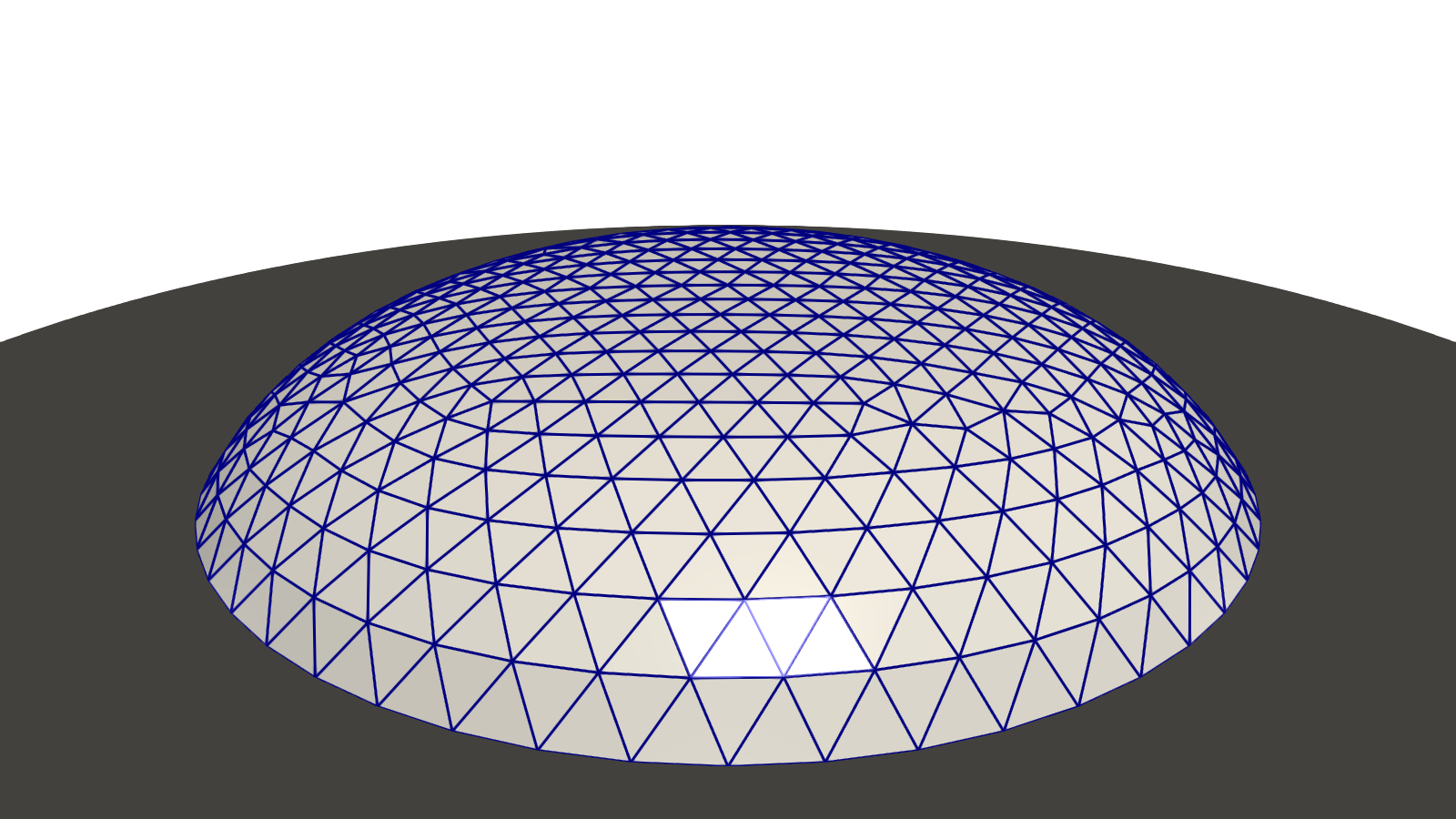} &
\includegraphics[width=0.19\textwidth]{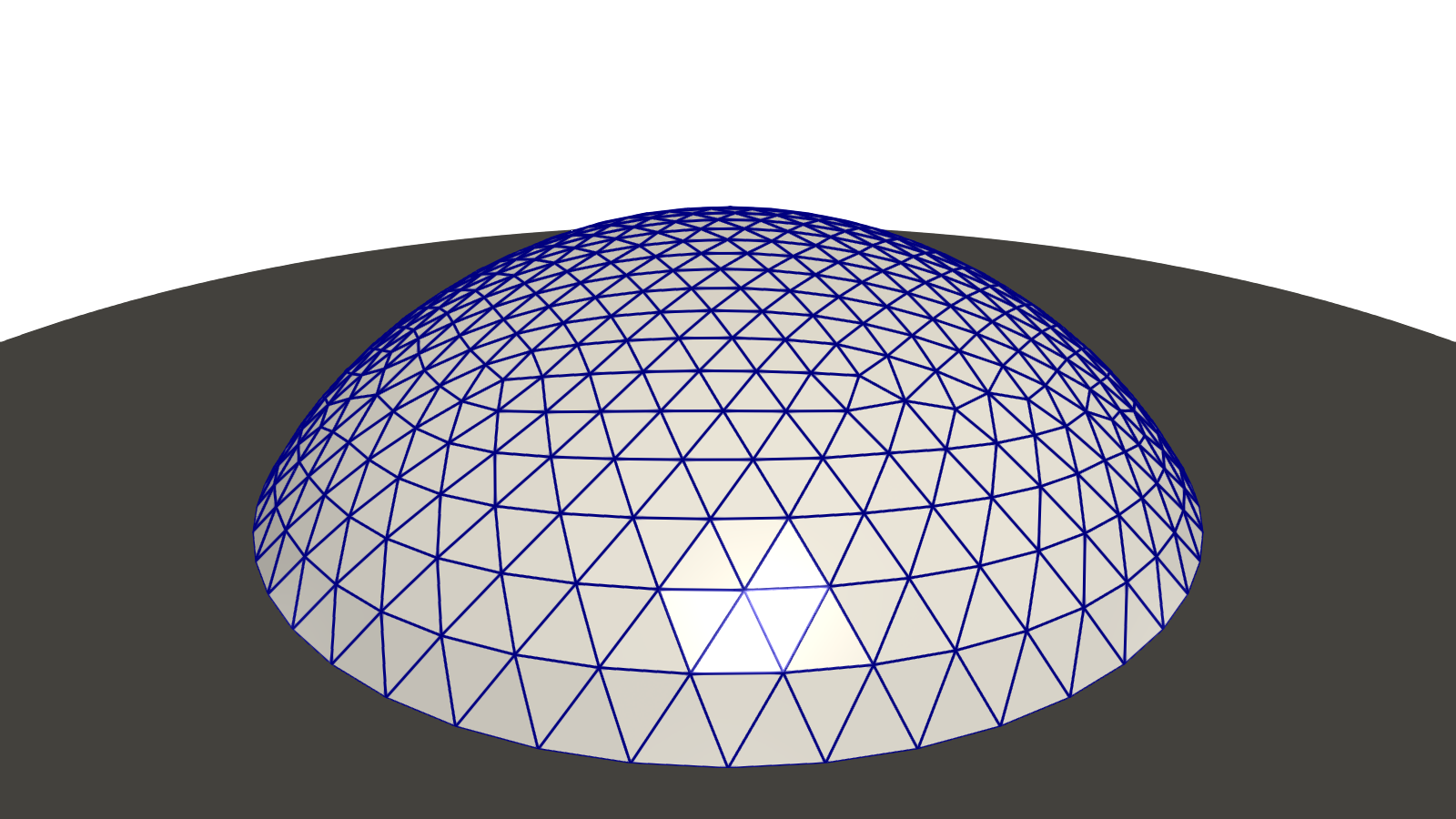} &
\includegraphics[width=0.19\textwidth]{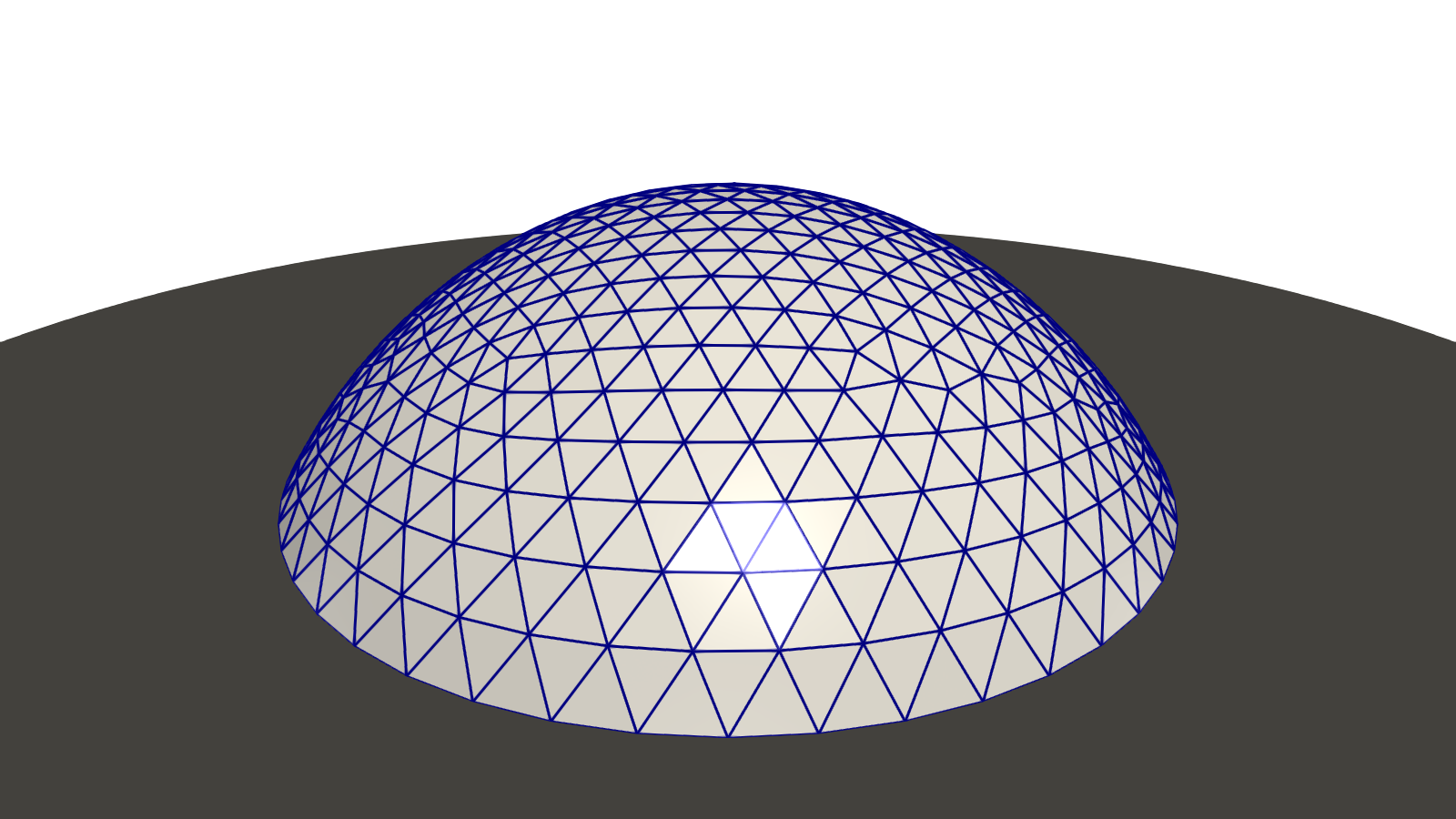} \\
\includegraphics[width=0.19\textwidth]{figure/Ellipsoid_Theta120_0.png} &
\includegraphics[width=0.19\textwidth]{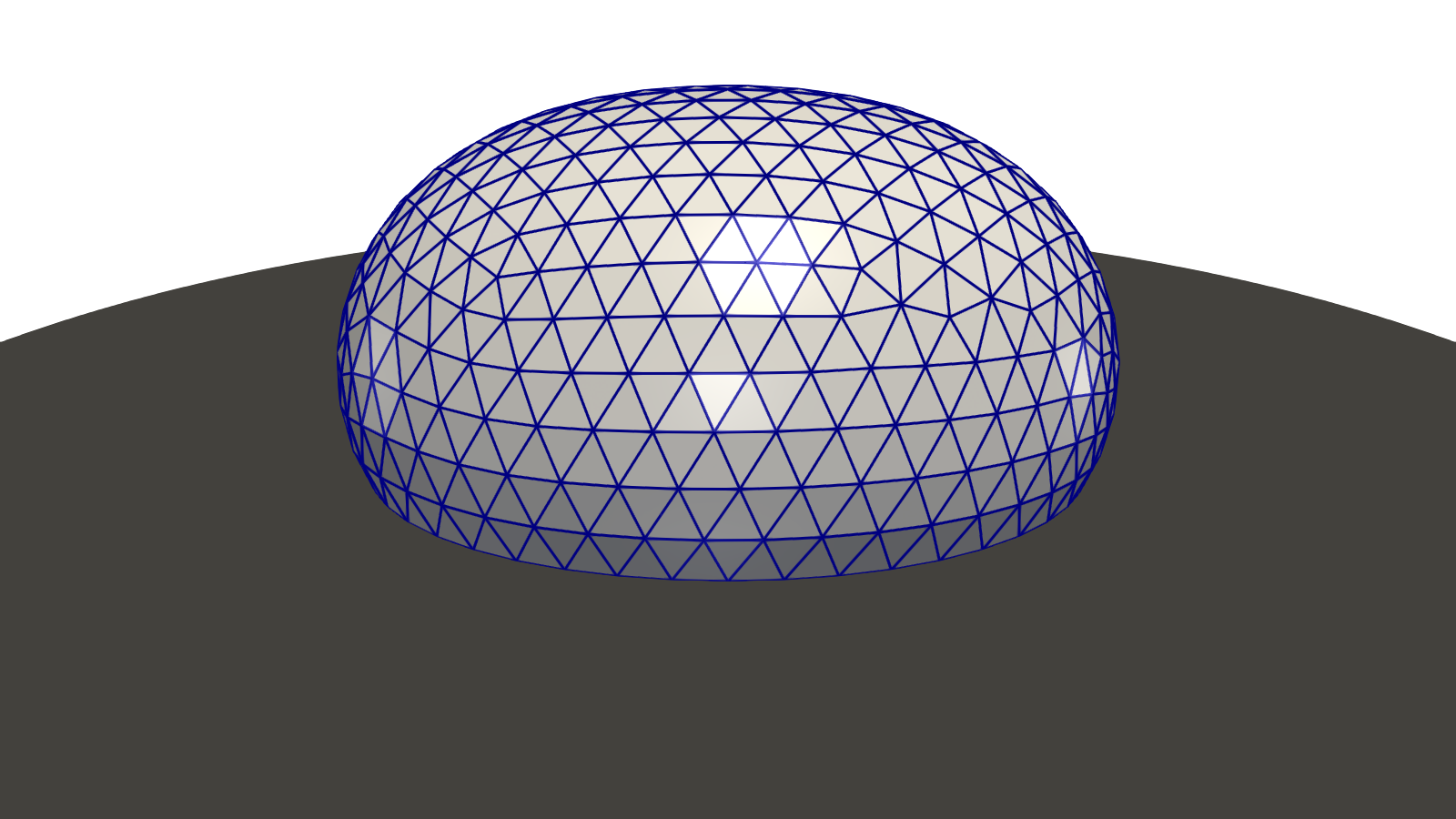} &
\includegraphics[width=0.19\textwidth]{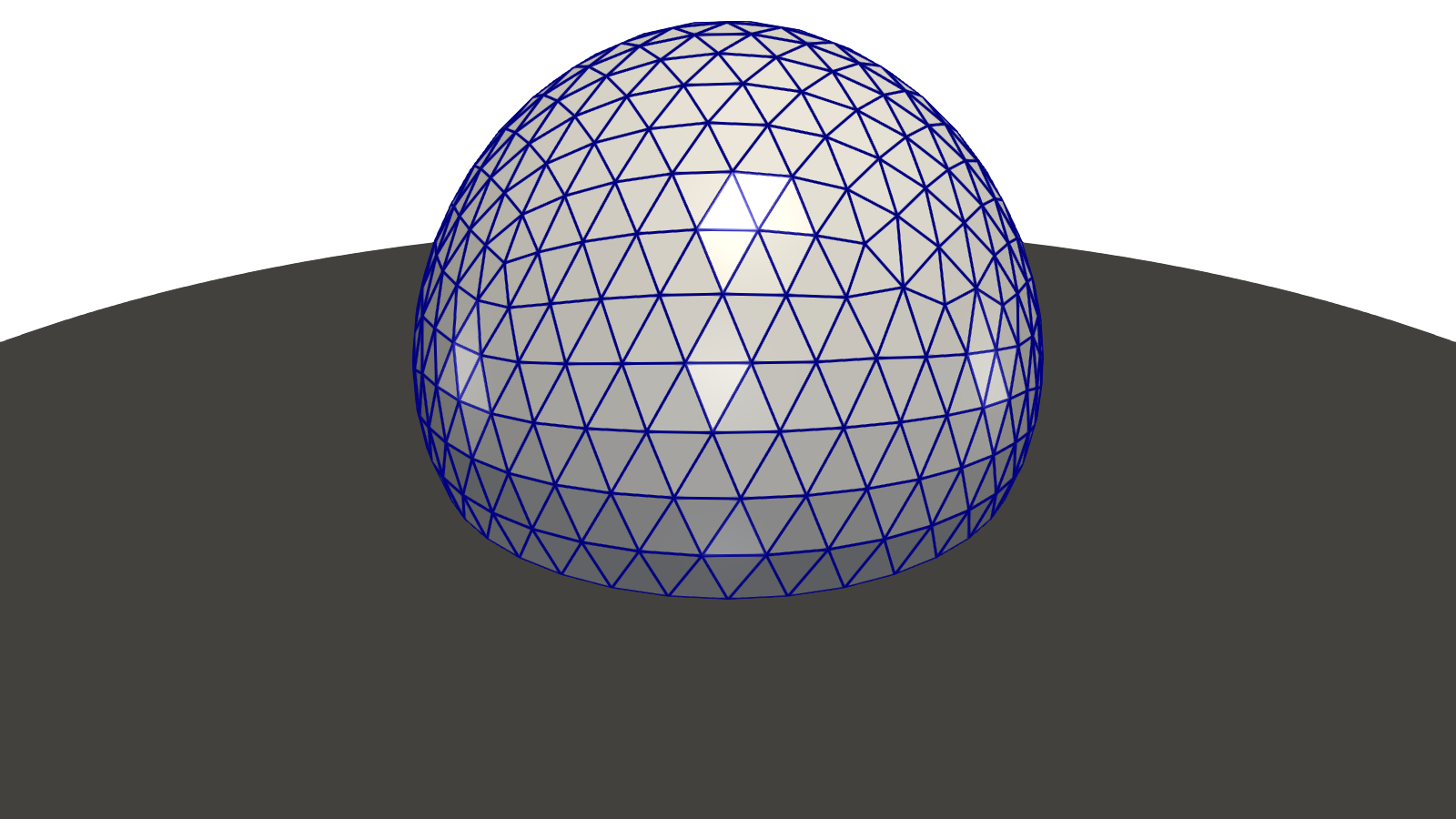} &
\includegraphics[width=0.19\textwidth]{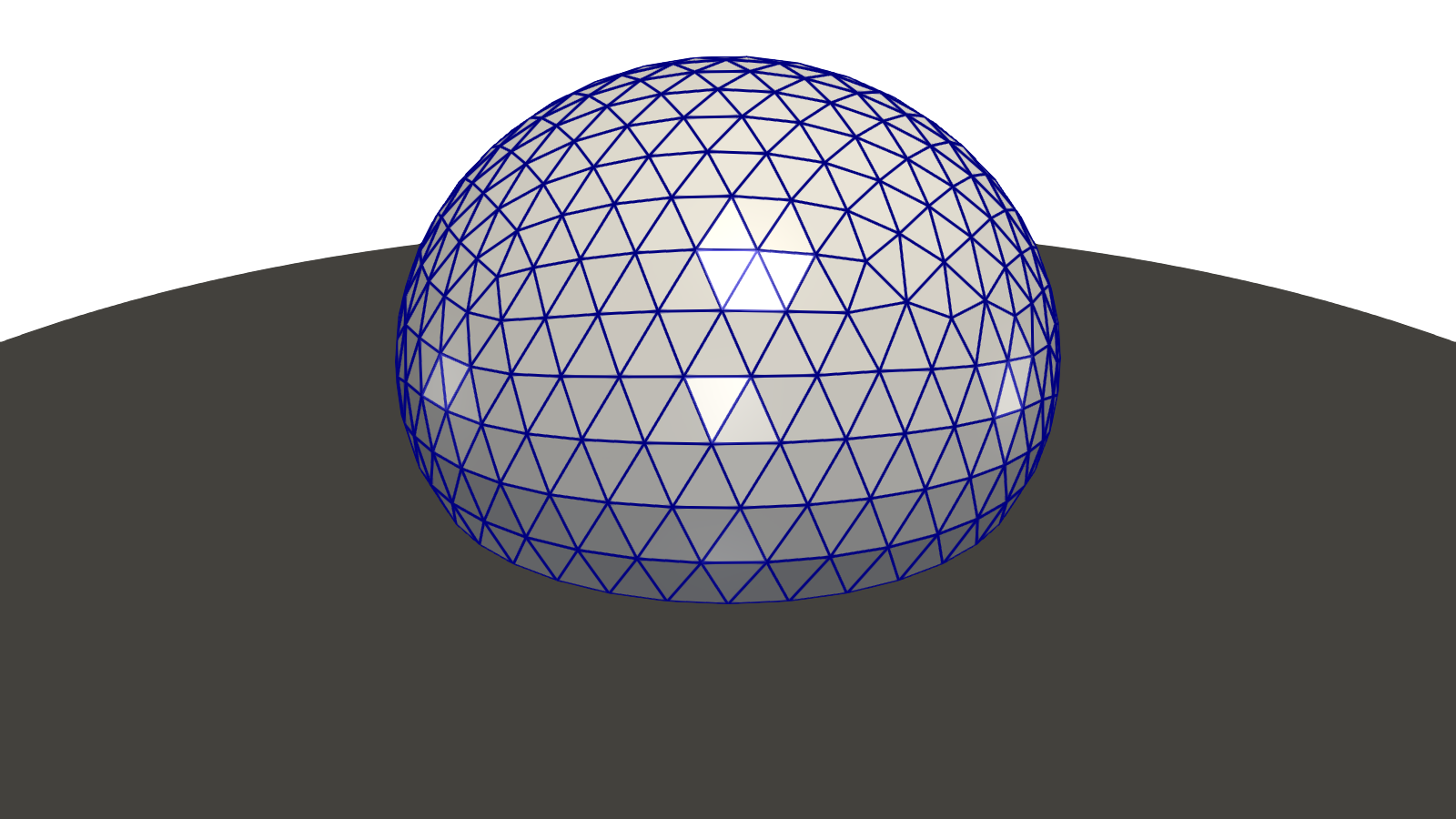} &
\includegraphics[width=0.19\textwidth]{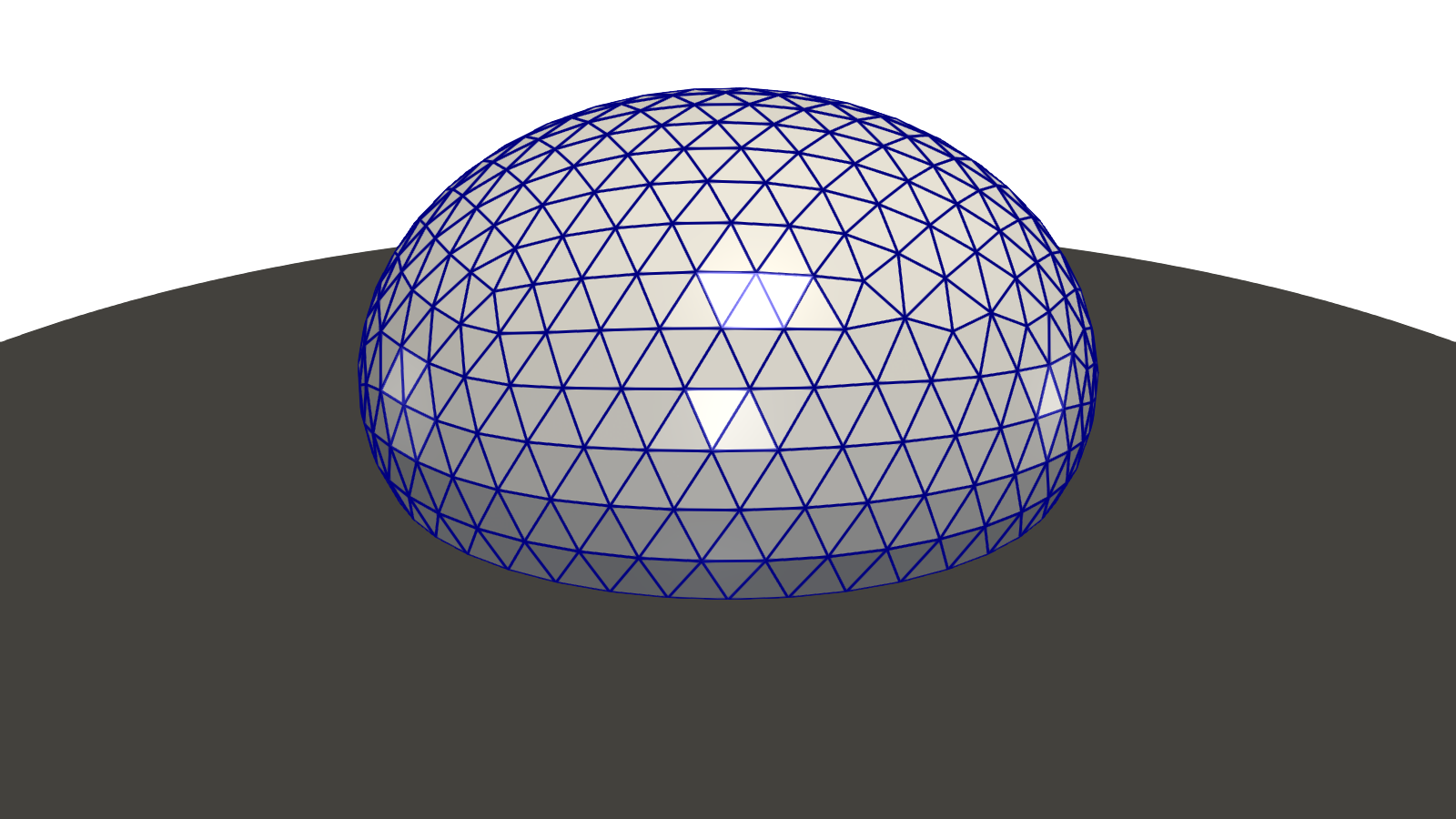}
\end{tabular}
\caption{Evolution of the ellipsoid droplet profile over one oscillation cycle. 
The top row corresponds to \(\theta_Y=60^\circ\), while the bottom row corresponds to \(\theta_Y=120^\circ\). }
\label{fig:Ellipsoid_3D_1}
\end{figure}

\begin{figure}[htbp]
\centering
\includegraphics[width=0.45\textwidth]{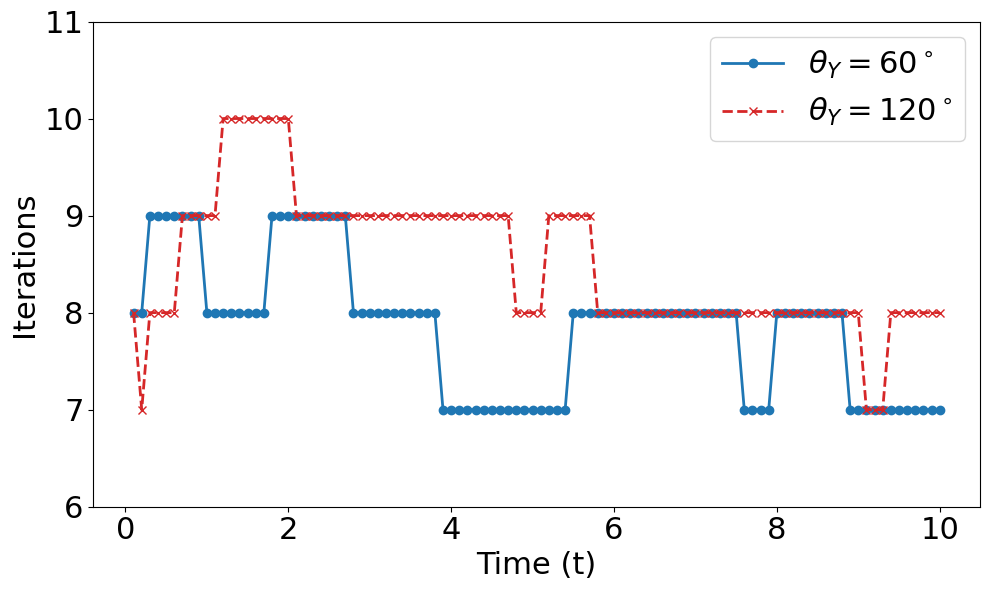}
\includegraphics[width=0.45\textwidth]{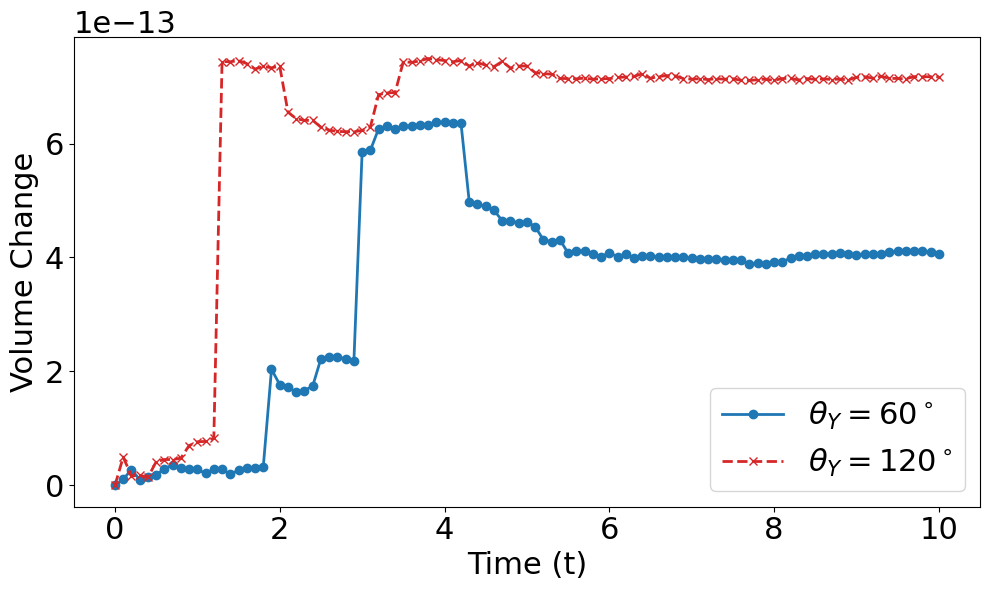}
\includegraphics[width=0.45\textwidth]{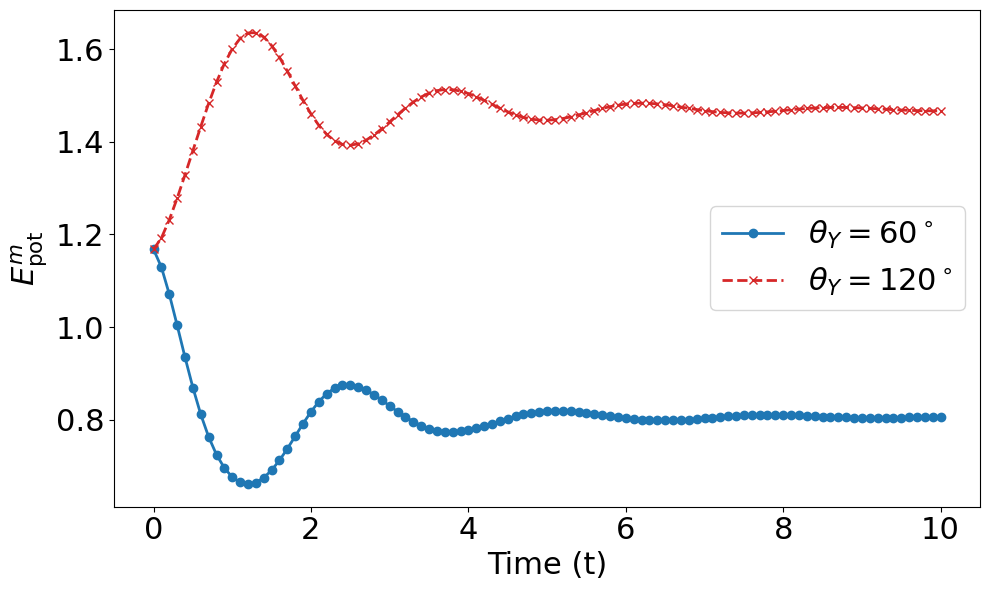}
\includegraphics[width=0.45\textwidth]{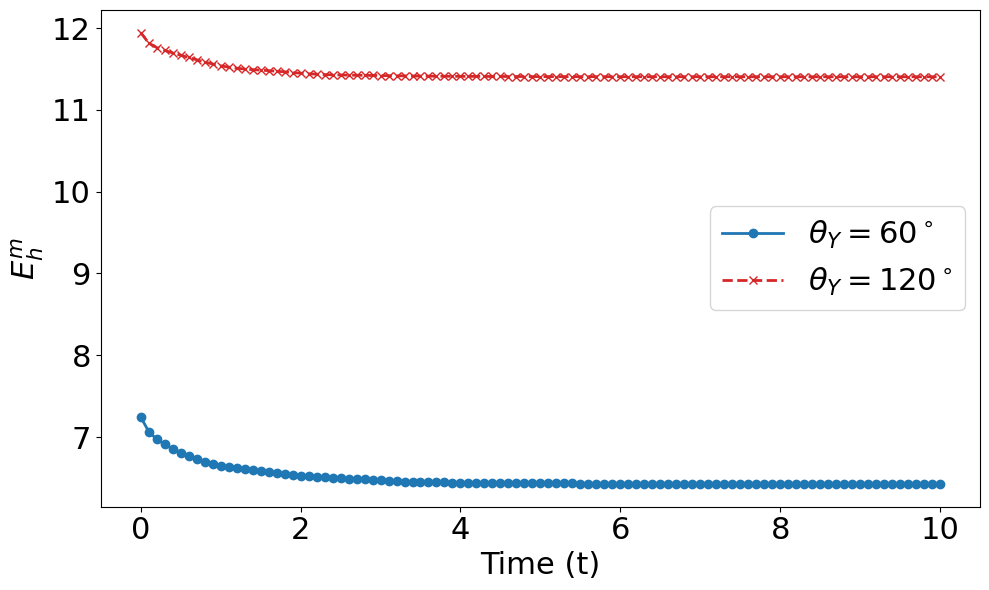}
\caption{Three-dimensional ellipsoid droplet simulations with \(\theta_Y = 60^\circ\) and \(\theta_Y = 120^\circ\) under gravity. 
Left upper: Newton iteration counts; 
Right upper: relative volume error \( \left| |\Omega_h^m|-|\Omega_h^0| \right| /|\Omega_h^0|\);
Left bottom: gravitational potential energy $E_{\rm pot}^m$; Right bottom: total energy $E_h^m$.}
\label{fig:Ellipsoid_3D_2}
\end{figure}

\FloatBarrier

\section{Conclusion}
\label{sec:conclusion}
In this paper, we develop a gravity-consistent ALE finite element method for Navier-Stokes free-boundary problems with moving contact lines. 
The key idea is to rewrite the gravitational force in geometric form and to evaluate the resulting moving-boundary contribution by Simpson's quadrature rule, so that the discrete gravitational work exactly matches the variation of the gravitational potential energy. 
As a result, the proposed nonlinear scheme preserves volume exactly, satisfies a discrete energy-dissipation law including gravity, 
and drives the discrete fluid velocity to zero in long-time simulations, thereby eliminating persistent spurious velocities. 
Numerical results in both two and three dimensions confirm the theoretical properties and demonstrate good mesh quality and correct long-time relaxation behavior.
The present framework can be extended naturally to more general interfacial flows, including two-phase problems and multiphase configurations with triple junctions.

\bibliographystyle{abbrv}
\bibliography{reference}

@article{Bansch2001,
  author  = {B{\"a}nsch, Eberhard},
  title   = {Finite element discretization of the {N}avier--{S}tokes equations with a free capillary surface},
  journal = {Numer. Math.},
  year    = {2001},
  volume  = {88},
  number  = {2},
  pages   = {203--235},
  doi     = {10.1007/PL00005443},
  url     = {https://link.springer.com/article/10.1007/PL00005443}
}

@article{Bao2021,
  title={A structure-preserving parametric finite element method for surface diffusion},
  author={Bao, Weizhu and Zhao, Quan},
  journal={SIAM J. Numer. Anal.},
  volume={59},
  number={5},
  pages={2775--2799},
  year={2021},
  publisher={SIAM}
}

@article{Barrett2007_BGN2,
title = {On the parametric finite element approximation of evolving hypersurfaces in ${R}^3$},
journal = {J. Comput. Phys.},
volume = {227},
number = {9},
pages = {4281-4307},
year = {2008},
issn = {0021-9991},
doi = {https://doi.org/10.1016/j.jcp.2007.11.023},
url = {https://www.sciencedirect.com/science/article/pii/S0021999107005141},
author = {John W. Barrett and Harald Garcke and Robert Nürnberg}
}

@incollection{BGN2020HoNA,
title = {Chapter 4 - Parametric finite element approximations of curvature-driven interface evolutions},
editor = {Andrea Bonito and Ricardo H. Nochetto},
series = {Handbook of Numerical Analysis},
publisher = {Elsevier},
volume = {21},
pages = {275-423},
year = {2020},
booktitle = {Geometric Partial Differential Equations - Part I},
issn = {1570-8659},
doi = {https://doi.org/10.1016/bs.hna.2019.05.002},
url = {https://www.sciencedirect.com/science/article/pii/S1570865919300055},
author = {John W. Barrett and Harald Garcke and Robert Nürnberg}
}

@book{Boffi2013,
  title={Mixed Finite Element Methods and Applications},
  author={Boffi, Daniele and Brezzi, Franco and Fortin, Michel},
  volume={44},
  year={2013},
  publisher={Springer}
}

@article{Duan2022,
title = {An energy diminishing arbitrary {L}agrangian–{E}ulerian finite element method for two-phase {N}avier–{S}tokes flow},
journal = {J. Comput. Phys.},
volume = {461},
pages = {111215},
year = {2022},
issn = {0021-9991},
doi = {https://doi.org/10.1016/j.jcp.2022.111215},
url = {https://www.sciencedirect.com/science/article/pii/S0021999122002777},
author = {Beiping Duan and Buyang Li and Zongze Yang}
}

@article{Garcke2023,
title = { Structure-preserving discretizations of two-phase {N}avier–{S}tokes flow using fitted and unfitted approaches},
journal = {J. Comput. Phys.},
volume = {489},
pages = {112276},
year = {2023},
issn = {0021-9991},
doi = {https://doi.org/10.1016/j.jcp.2023.112276},
url = {https://www.sciencedirect.com/science/article/pii/S0021999123003716},
author = {Harald Garcke and Robert Nürnberg and Quan Zhao},
}

@article{Hu2022,
  title = {Evolving finite element methods with an artificial tangential velocity for mean curvature flow and {{Willmore}} flow},
  author = {Hu, Jiashun and Li, Buyang},
  date = {2022},
  volume = {152},
  number = {1},
  pages = {127--181},
  doi = {10.1007/s00211-022-01309-9},
  langid = {english},
  journal = {Numer. Math.},
  year = {2022},
  dimensions = {true},
}

@article{Gao2024_rotation,
  title = {A Stabilized arbitrary {L}agrangian-{E}ulerian sliding interface method for fluid-structure
  interaction with a rotating rigid structure},
  author = {Gao, Yali and Hu, Jiashun and Li, Buyang},
  date = {},
  volume = {47},
  number = {5},
  pages = {A2533-A2558},
  year = {2025},
  doi = {10.1137/24M1695919},
  langid = {english},
  journal = {SIAM J. Sci. Comput.},
  dimensions = {true},
}

@article{Ren2007,
    author = {Ren, Weiqing and E, Weinan},
    title = {Boundary conditions for the moving contact line problem},
    journal = {Phys. Fluids},
    volume = {19},
    number = {2},
    pages = {022101},
    year = {2007},
    month = {02},
    issn = {1070-6631},
    doi = {10.1063/1.2646754},
    url = {https://doi.org/10.1063/1.2646754},
    eprint = {https://pubs.aip.org/aip/pof/article-pdf/doi/10.1063/1.2646754/15636165/022101\_1\_online.pdf},
}

@article{Ren2011,
  title={Derivation of continuum models for the moving contact line problem based on thermodynamic principles},
  author={Ren, Weiqing and E, Weinan},
  journal={Commun. Math. Sci.},
  volume={9},
  number={2},
  pages={597--606},
  year={2011}
}

@article{Zhao2020,
title = { An energy-stable finite element method for the simulation of moving contact lines in two-phase flows},
journal = {J. Comput. Phys.},
volume = {417},
pages = {109582},
year = {2020},
issn = {0021-9991},
doi = {https://doi.org/10.1016/j.jcp.2020.109582},
url = {https://www.sciencedirect.com/science/article/pii/S0021999120303569},
author = {Quan Zhao and Weiqing Ren}
}

@article{Bao2023,
  author = {W. Bao and Q. Zhao},
  title = {An energy-stable parametric finite element method for simulating solid-state dewetting problems in three dimensions},
  journal = {J. Comput. Math.},
  volume = {41},
  number = {4},
  pages = {771--796},
  year = {2023}
}

@article{Gao-Li-2025,
  title = {Geometric-structure preserving methods for surface evolution in curvature flows with minimal deformation formulations},
journal = {J. Comput. Phys.},
volume = {524},
pages = {113718},
year = {2025},
issn = {0021-9991},
doi = {https://doi.org/10.1016/j.jcp.2025.113718},
url = {https://www.sciencedirect.com/science/article/pii/S0021999125000014},
author = {Guangwei Gao and Buyang Li},
}

@article{Xu2023,
author = {Xu, Xianmin},
title = {A Unified Variational Framework on Macroscopic Computations for Two-Phase Flow with Moving Contact Lines},
journal = {SIAM J. Sci. Comput.},
volume = {45},
number = {6},
pages = {B776-B801},
year = {2023},
doi = {10.1137/23M1546816},
URL = { 
        https://doi.org/10.1137/23M1546816
},
eprint = { 
        https://doi.org/10.1137/23M1546816
}
,
}

@article{Zhang2026,
title = {Structure-preserving parametric finite element methods for two-phase {S}tokes flow based on {L}agrange multiplier approaches},
journal = {J. Comput. Phys.},
volume = {560},
pages = {114922},
year = {2026},
issn = {0021-9991},
doi = {https://doi.org/10.1016/j.jcp.2026.114922},
url = {https://www.sciencedirect.com/science/article/pii/S002199912600272X},
author = {Harald Garcke and Dennis Trautwein and Ganghui Zhang},
}

@article{Jiang2021,
title = {A perimeter-decreasing and area-conserving algorithm for surface diffusion flow of curves},
journal = {J. Comput. Phys.},
volume = {443},
pages = {110531},
year = {2021},
issn = {0021-9991},
doi = {https://doi.org/10.1016/j.jcp.2021.110531},
url = {https://www.sciencedirect.com/science/article/pii/S0021999121004265},
author = {Wei Jiang and Buyang Li},
}

@article{Hu2026_Droplet,
author = {Hu, Jiashun and Lei, Nuo and Li, Buyang and Tang, Rong},
title = {Energy dissipating {ALE}-{MDR} method for {N}avier–{S}tokes free boundary problems with moving contact line},
journal = {SIAM J. Sci. Comput.},
volume = {48},
number = {3},
pages = {A1284-A1311},
year = {2026},
doi = {10.1137/25M1784958},
URL = {https://doi.org/10.1137/25M1784958},
}

@article{Nangia2019,
title = {A robust incompressible {N}avier-{S}tokes solver for high density ratio multiphase flows},
journal = {J. Comput. Phys.},
volume = {390},
pages = {548-594},
year = {2019},
issn = {0021-9991},
doi = {https://doi.org/10.1016/j.jcp.2019.03.042},
url = {https://www.sciencedirect.com/science/article/pii/S0021999119302256},
author = {Nishant Nangia and Boyce E. Griffith and Neelesh A. Patankar and Amneet Pal Singh Bhalla},
}

@article{Patel2017,
title = {A novel consistent and well-balanced algorithm for simulations of multiphase flows on unstructured grids},
journal = {J. Comput. Phys.},
volume = {350},
pages = {207-236},
year = {2017},
issn = {0021-9991},
doi = {https://doi.org/10.1016/j.jcp.2017.08.047},
url = {https://www.sciencedirect.com/science/article/pii/S0021999117306289},
author = {Jitendra Kumar Patel and Ganesh Natarajan},
}

@article{BGN2013_spurious,
title = {Eliminating spurious velocities with a stable approximation of viscous incompressible two-phase {S}tokes flow},
journal = {Comput. Methods Appl. Mech. Engrg.},
volume = {267},
pages = {511-530},
year = {2013},
issn = {0045-7825},
doi = {https://doi.org/10.1016/j.cma.2013.09.023},
url = {https://www.sciencedirect.com/science/article/pii/S0045782513002508},
author = {John W. Barrett and Harald Garcke and Robert Nürnberg},
}

@article{Ganesan2007,
title = {On spurious velocities in incompressible flow problems with interfaces},
journal = {Comput. Methods Appl. Mech. Engrg.},
volume = {196},
number = {7},
pages = {1193-1202},
year = {2007},
issn = {0045-7825},
doi = {https://doi.org/10.1016/j.cma.2006.08.018},
url = {https://www.sciencedirect.com/science/article/pii/S004578250600291X},
author = {Sashikumaar Ganesan and Gunar Matthies and Lutz Tobiska},
}

@book{Zeidler1988IV,
  title={Nonlinear Functional Analysis and its Applications: IV: Applications to Mathematical Physics},
  author={Zeidler, Eberhard},
  year={1988},
  publisher={Springer-Verlag},
  address={New York},
  isbn={978-0-387-96499-7}
}

@article{Guo2024,
  title={Stability of contact lines in fluids: 2{D} {N}avier--{S}tokes flow},
  author={Guo, Yan and Tice, Ian},
  journal={J. Eur. Math. Soc.},
  volume={26},
  number={4},
  pages={1445--1557},
  year={2024},
}

@article{Qian2006JFM, 
title={A variational approach to moving contact line hydrodynamics}, 
volume={564}, 
DOI={10.1017/S0022112006001935}, 
journal={J. Fluid Mech.}, 
author={Qian, Tiezheng and Wang, Xiao-Ping and Sheng, Ping}, 
year={2006}, 
pages={333-–360}
}

@article{Qian2003,
  title={Molecular scale contact line hydrodynamics of immiscible flows},
  author={Qian, Tiezheng and Wang, Xiao-Ping and Sheng, Ping},
  journal={Phys. Rev. E},
  volume={68},
  number={1},
  pages={016306},
  year={2003},
  publisher={APS}
}

@article{GarckeJiangSuZhang2025SISC,
author = {Garcke, Harald and Jiang, Wei and Su, Chunmei and Zhang, Ganghui},
title = {Structure-Preserving Parametric Finite Element Method for Surface Diffusion Based on {L}agrange Multiplier Approaches},
journal = {SIAM J. Sci. Comput.},
volume = {47},
number = {3},
pages = {A1983-A2011},
year = {2025},
doi = {10.1137/24M1687546},
}

@article{TangTang2003SINUM,
author = {Tang, Huazhong and Tang, Tao},
title = {Adaptive Mesh Methods for One- and Two-Dimensional Hyperbolic Conservation Laws},
journal = {SIAM J. Numer. Anal.},
volume = {41},
number = {2},
pages = {487-515},
year = {2003},
doi = {10.1137/S003614290138437X},
}

@article{DziukElliott2013MC,
  title={{$L^2$}-estimates for the evolving surface finite element method},
  author={Dziuk, Gerhard and Elliott, Charles M.},
  journal={Math. Comput.},
  volume={82},
  number={281},
  pages={1--24},
  year={2013},
  publisher={American Mathematical Society}
}

@article{Zhao2021CMAME,
title = { A thermodynamically consistent model and its conservative numerical approximation for moving contact lines with soluble surfactants},
journal = {Comput. Methods Appl. Mech. Engrg.},
volume = {385},
pages = {114033},
year = {2021},
issn = {0045-7825},
doi = {https://doi.org/10.1016/j.cma.2021.114033},
url = {https://www.sciencedirect.com/science/article/pii/S0045782521003649},
author = {Quan Zhao and Weiqing Ren and Zhen Zhang},
}

@article{Agnese2020,
title = { Fitted front tracking methods for two-phase incompressible {N}avier–{S}tokes flow: {E}ulerian and {ALE} finite element discretizations },
author = {Agnese, Marco and N{\"u}rnberg, Robert},
volume={17}, 
url={https://www.global-sci.com/ijnam/article/view/10421}, 
number={5}, 
journal = {Int. J. Numer. Anal. Model.},
year={2020}, 
month={Aug.}, 
pages={613-–642} 
}

@article{BGN2015,
  author = {John W. Barrett and Harald Garcke and Robert Nürnberg},
  title   = { A Stable Parametric Finite Element Discretization of Two-Phase {N}avier--{S}tokes Flow},
  journal = {J. Sci. Comput.},
  volume  = {63},
  number  = {1},
  pages   = {78--117},
  year    = {2015},
  doi     = {10.1007/s10915-014-9885-2},
  url     = {https://doi.org/10.1007/s10915-014-9885-2}
}

\end{document}